\documentclass{amsart}
\usepackage[english]{babel}
\usepackage[latin1]{inputenc}
 \usepackage[all]{xy}

\usepackage{amsmath,amsfonts,amssymb,amsthm,amscd,array,stmaryrd,mathrsfs, mathdots}
\PassOptionsToPackage{option}{xcolor}
\usepackage{pstricks}

\theoremstyle{plain}
\newtheorem{thm}{Theorem}

\newtheorem{lem}{Lemma}[section]
\newtheorem{cor}[lem]{Corollary}
\newtheorem{prop}[lem]{Proposition}

\theoremstyle{definition}
\newtheorem{defn}[lem]{Definition}
\newtheorem{rem}[lem]{Remark}
\newtheorem{ex}[lem]{Example}

\let\ssection=\section
\renewcommand{\section}{\setcounter{equation}{0}\ssection}

\newcommand{\Z}{\mathbb{Z}}
\newcommand{\C}{\mathbb{C}}

\newcommand{\T}{\mathbb{T}}
\newcommand{\Q}{\mathbb{Q}}

\newcommand{\cC}{\mathcal{C}}

\newcommand{\F}{\mathcal{F}}
\newcommand{\G}{\mathcal{G}}

\newcommand{\Pc}{\mathcal{P}} 
\newcommand{\Qc}{\mathcal{Q}} 
\newcommand{\Rc}{\mathcal{R}}
\newcommand{\Sc}{\mathcal{S}}

\newcommand{\cX}{\mathcal{X}}

\newcommand{\id}{\textup{Id}}

\newcommand{\Id}{\mathrm{Id}}
\newcommand{\GL}{\mathrm{GL}}
\newcommand{\SL}{\mathrm{SL}}
\newcommand{\PSL}{\mathrm{PSL}}

\newcommand{\half}{\frac{1}{2}}

\def\r{\rho}
\def\s{\sigma}

\begin{document}

\title{$q$-deformed rationals and $q$-continued fractions}

\author{Sophie Morier-Genoud, Valentin Ovsienko}

\address{Sophie Morier-Genoud,
Sorbonne Universit\'e, Universit\'e Paris Diderot, CNRS,
Institut de Math\'ematiques de Jussieu-Paris Rive Gauche,
 F-75005, Paris, France
}

\address{
Valentin Ovsienko,
Centre national de la recherche scientifique,
Laboratoire de Math\'ematiques 
U.F.R. Sciences Exactes et Naturelles 
Moulin de la Housse - BP 1039 
51687 Reims cedex 2,
France}
\email{sophie.morier-genoud@imj-prg.fr, valentin.ovsienko@univ-reims.fr}

\keywords{$q$-rationals, $q$-continued fractions, total positivity, Farey tessellation, quiver representations.}


\begin{abstract}
We introduce a notion of $q$-deformed rational numbers and
$q$-deformed continued fractions.
A $q$-deformed rational is encoded by a triangulation of a polygon
and can be computed recursively.
The recursive formula is analogous to the
$q$-deformed Pascal identitiy for the Gaussian binomial coefficients, but
the Pascal triangle is replaced by the Farey graph.
The coefficients of the polynomials
defining the $q$-rational count quiver subrepresentations of the
maximal indecomposable representation of the graph dual to the triangulation.
Several other properties, such as total positivity properties,
$q$-deformation of the Farey graph, matrix presentations and $q$-continuants are given,
as well as a relation to the Jones polynomial of rational knots.
\end{abstract}

\maketitle

\thispagestyle{empty}

\tableofcontents

\section*{Introduction}\label{IntSec}

Let $q$ be a formal parameter,
recall the classical notion of $q$-integers:
$$
[a]_{q}:=\frac{1-q^a}{1-q}=1+q+q^{2}+\cdots+q^{a-1},
$$
where $a$ is a non-negative integer.
Surprisingly, a similar notion in the case of rational numbers is missing in the literature.
A straightforward attempt to $q$-deform a rational number~$\frac{r}{s}$
would be
$\frac{[r]_q}{[s]_q}$, or $\frac{q^{\frac{r}{s}}-1}{q-1}$, (which is essentially the same formula
modulo rescaling).
In those formulas the numerator and denominator are $q$-deformed separately, as $q$-integers.
However, such definitions would not satisfy interesting combinatorial properties.

The goal of this paper is to define a version of $q$-analogue of rational numbers,
using a combinatorial approach.
In our approach the numerator and denominator affect each other's $q$-deformation.
The $q$-deformation of a rational $\frac{r}{s}$ is a quotient of two polynomials:
$$
\left[\frac{r}{s}\right]_{q}=\frac{\Rc(q)}{\Sc(q)},
$$
where~$\Rc$ and~$\Sc$ both depend on~$r$ and~$s$.
The simplest example to illustrate this is:
$$
\left[\frac{5}{2}\right]_{q}=
\frac{1+2q+q^{2}+q^{3}}{1+q},
\qquad\qquad
\left[\frac{5}{3}\right]_{q}=
\frac{1+q+2q^{2}+q^{3}}{1+q+q^{2}}
$$
where the ``quantized'' $5$ in the numerator is different depending on the denominator.

The notion of $q$-rational we introduce has a certain similarity to that of the 
classical Gaussian $q$-binomial coefficient
${r\choose{}s}_q$, which is a polynomial depending on~$r$ and~$s$ that $q$-deforms 
the usual binomial coefficient ${r\choose{}s}$.
This similarity manifests in the comparison of the weighted Farey sum rule and the $q$-version of
Pascal's triangle.
However, the Farey graph is more complicated than the Pascal triangle,
since every rational has infinitely many neighbors.

The main properties of the $q$-rationals we know so far are the following.

\begin{enumerate}
\item[{\bf A}]
$q$-deformed rationals satisfy a {\it total positivity} property.
For every rational~$\frac{r}{s}>0$,
both polynomials~$\Rc$ and~$\Sc$ in the numerator and denominator of~$\left[\frac{r}{s}\right]_{q}$
have positive integer coefficients.
Moreover, for two rationals,~$\frac{r}{s}>\frac{r'}{s'}$, the polynomial
$\Rc\Sc'-\Sc\Rc'$ has positive integer coefficients (see Theorem~\ref{PosPropBis}).

\item[{\bf B}]
There is an ``enumerative interpretation'' of the $q$-rationals (Theorem~\ref{EnumThm}).
Coefficients of the polynomials in the numerator and denominator count 
so-called {\it closure sets} of a certain graph.
This is also equivalent to counting subrepresentations of the maximal
indecomposable quiver representation.

\item[{\bf C}]
A significant property of $q$-rationals is that of {\it convergence}.
Given an irrational real number $x$ and a sequence of rationals $\left(\frac{r_i}{s_i}\right)_{i\geq1}$ converging to~$x$,
the coefficients of the Taylor series at $q=0$ of the rational functions $\frac{\Rc_i(q)}{\Sc_i(q)}$ stabilize as~$i$ grows.
Moreover, the limit series does not depend on the choice of the converging sequence of rationals.
It allows us to extend the notion of $q$-deformation to irrational numbers.
This property is proved in~\cite{QR}, and we will not discuss it here.

\end{enumerate}

In addition, we obtain $q$-analogues of several classical notions and results.

\begin{enumerate}

\item[(a)]
Theorem~\ref{Thmqa} provides a relationship between regular and negative continued fractions,
generalizing the formulas from~\cite{Hir,HZ}.

\item[(b)]
The {\it weighted Farey graph}.
Theorem~\ref{WeightThm} generalizes the result of~\cite{Ser} that relates continued fractions and hyperbolic geometry.

\item[(c)]
The $q$-continued fractions and $q$-rationals
are represented by products of elementary matrices
$$
R_{q}:=
\begin{pmatrix}
q&1\\[4pt]
0&1
\end{pmatrix},
\qquad
L_{q}:=
\begin{pmatrix}
1&0\\[4pt]
1&q^{-1}
\end{pmatrix},
\qquad
S_{q}:=
\begin{pmatrix}
0&-q^{-1}\\[4pt]
1&0
\end{pmatrix},
$$
that are $q$-{\it deformed generators} of~$\PSL(2,\Z)$.

\item[(d)]
We obtain $q$-analogues of the continuants 
(i.e., the determinants that give the numerator and denominator of a continued fraction
in terms of its coefficients), and prove the Ptolemy relations
generalizing the Euler identity for the continuants (see Proposition~\ref{EulqProp}).
\end{enumerate}

To summarize, the results (a)--(d) of the above list give four equivalent
ways to define and calculate $q$-rationals.
These four interpretations can be used for different purposes.

Throughout the paper, we work only with the case $\frac{r}{s}>1$.
However, the following property
$$
\left[\frac{r}{s}+1\right]_q=q\left[\frac{r}{s}\right]_q+1,
$$
provided $\frac{r}{s}\geq1$,
was observed in~\cite{QR} as a simple corollary of the definition~\eqref{qa} below.
This formula allows us to extend the notion of $q$-deformed rational number 
to~$\frac{r}{s}\leq1$ (including the negative case).

We test the notion of $q$-rationals on the example of two particularly simple infinite sequences of rationals:
the ratios of consecutive Fibonacci and Pell numbers.
These are consecutive convergents of the {\it golden ratio} and those of the {\it silver ratio}.
Our approach leads to certain versions of $q$-Fibonacci and $q$-Pell numbers.
The Fibonacci case turns out to be related to the well-known Sequence A079487 of OEIS
(and its mirror A123245),
while the obtained version of $q$-Pell numbers seems to be new.
Note that $q$-analogues of the Fibonacci and Pell numbers is a vast subject,
and several different versions of them are known, but do not appear in our context.

The notion of $q$-rational introduced in this paper needs only quite elementary
constructions.
However, this notion arose at the interface of two theories,
that of continued fractions and cluster algebra. 
Continued fractions appeared recently in the theory of cluster algebras \cite{CS1, Rab}.
These motivating relations are presented in two appendices
in a brief form with no technical details.
In Appendix~\ref{jones}, we relate the $q$-rationals to the Jones polynomials of rational knots.
In Appendix~\ref{ptol}, we relate the $q$-rationals to the $F$-polynomials of a cluster algebra.

This paper is the second part of~\cite{FB1} where we discussed combinatorial properties of
continued fractions.
However, it can be read independently.
The subsequent paper~\cite{QR} explores ``$q$-deformed real numbers''
based on the notion of $q$-rationals developed here.

\section{Introducing $q$-deformed rationals and continued fractions}\label{ICFqSec}
In this section 
we define $q$-deformations of finite continued fractions and formulate their first basic properties.

\subsection{The main definition}
Given a rational number $\frac{r}{s}\in\Q$,
we always assume~$\frac{r}{s}>1$ and that $r$ and $s$ are positive coprime integers.
It has two different continued fraction expansions
$$
\frac{r}{s}
\quad=\quad
a_1 + \cfrac{1}{a_2 
          + \cfrac{1}{\ddots +\cfrac{1}{a_{2m}} } }
           \quad =\quad
c_1 - \cfrac{1}{c_2 
          - \cfrac{1}{\ddots - \cfrac{1}{c_k} } } ,
$$
with $c_i\geq2$ and $a_i\geq1$, denoted by
$[a_1,\ldots,a_{2m}]$ and $\llbracket{}c_1,\ldots,c_k\rrbracket{}$, respectively.
They are usually called {\it regular} and {\it negative} (or ``minus-'') 
continued fractions, respectively. Considering an even number of coefficients in the regular expansion and coefficients greater than 2 in the negative expansion make the expansions unique.

We define $q$-deformations of the regular and negative continued fractions.
We start with an explicit formula;
equivalent, but more combinatorial, definitions will be given in Sections~\ref{WTSec}
and~\ref{EmunSec}.

\begin{defn}
\label{qDefn}
{\rm
(a)
Given a regular continued fraction $[a_{1}, \ldots, a_{2m}]$,
we define its $q$-deformation by
\begin{equation}
\label{qa}
[a_{1}, \ldots, a_{2m}]_{q}:=
[a_1]_{q} + \cfrac{q^{a_{1}}}{[a_2]_{q^{-1}} 
          + \cfrac{q^{-a_{2}}}{[a_{3}]_{q} 
          +\cfrac{q^{a_{3}}}{[a_{4}]_{q^{-1}}
          + \cfrac{q^{-a_{4}}}{
        \cfrac{\ddots}{[a_{2m-1}]_q+\cfrac{q^{a_{2m-1}}}{[a_{2m}]_{q^{-1}}}}}
          } }} .
\end{equation}

(b)
Given a negative continued  fraction  $\llbracket{}c_1,\ldots,c_k\rrbracket$,
we define its $q$-deformation by 
\begin{equation}
\label{qc}
\llbracket{}c_1,\ldots,c_k\rrbracket_{q}:=
[c_1]_{q} - \cfrac{q^{c_{1}-1}}{[c_2]_{q} 
          - \cfrac{q^{c_{2}-1}}{\ddots \cfrac{\ddots}{[c_{k-1}]_{q}- \cfrac{q^{c_{k-1}-1}}{[c_k]_{q}} } }} 
\end{equation}
}.
\end{defn}

The following statement will be proved in Section~\ref{PT1Sec}.

\begin{thm}
\label{Thmqa}
If a positive continued fraction $[a_{1}, \ldots, a_{2m}]$
and a negative continued fraction $\llbracket{}c_1,\ldots,c_k\rrbracket$ represent the same rational number,
then their $q$-deformations also coincide:
\begin{equation*}
[a_{1}, \ldots, a_{2m}]_{q}=\llbracket{}c_1,\ldots,c_k\rrbracket_{q}.
\end{equation*}
\end{thm}

In other words, the explicit formula (see~\cite{Hir,HZ}) for conversion between the coefficients $a$'s and $c$'s
remains unchanged in the $q$-deformed case.
This formula will be recalled in Section~\ref{Trs}.

Given a rational number~$\frac{r}{s}=[a_{1}, \ldots, a_{2m}]=\llbracket{}c_1,\ldots,c_k\rrbracket$,
the rational function defined by~\eqref{qa} and~\eqref{qc} will be called
the {\it $q$-deformation} of~$\frac{r}{s}$, or simply a {\it $q$-rational}.
We will always write this rational function as the ratio of two coprime polynomials in $q$
with integer coefficients that we denote by $\Rc$ and $\Sc$, i.e.
\begin{equation}
\label{RcSc}
\left[\frac{r}{s}\right]_{q}=\frac{\Rc(q)}{\Sc(q)},
\end{equation}
with~$\Rc$ and~$\Sc$ in $\Z[q]$.
We choose the positive signs of the leading terms of $\Rc$ and $\Sc$
to guarantee their uniqueness.

We show furthermore
(cf. Corollary \ref{DegRSProp}) that the leading and the constant coefficients of $\Rc$ and $\Sc$
are equal to~$1$ and that
$$
\Rc(1)=r,
\quad
\Sc(1)=s.
$$
We also conjecture that the sequences of coefficients of $\Rc$ and $\Sc$ are unimodal.

Let us start with simple examples.

\begin{ex}
\label{FirstEx}

(a)
The cases~$s=1$ and~$s=r-1$ are the only cases where
$\left[\frac{r}{s}\right]_{q}$ coincides with the quotient of the $q$-deformed integers standing in the numerator and denominator:
$$
\left[\frac{r}{r-1}\right]_{q}=\frac{[r]_{q}}{[r-1]_{q}}.
$$

(b)
The first non-trivial examples are
$$
\begin{array}{lllllll}
\left[\frac{5}{2}\right]_{q}&=&
\llbracket3,2\rrbracket_{q}&=&
[2,2]_{q}&=&
\frac{1+2q+q^{2}+q^{3}}{1+q},\\[8pt]
\left[\frac{5}{3}\right]_{q}&=&
\llbracket2,3\rrbracket_{q}&=&
[1,1,1,1]_{q}&=&
\frac{1+q+2q^{2}+q^{3}}{1+q+q^{2}},\\[8pt]
\left[\frac{7}{3}\right]_{q}&=&
\llbracket3,2,2\rrbracket_{q}&=&
[2,3]_{q}&=&
\frac{1+2q+2q^2+q^3+q^4}{1+q+q^{2}},\\[8pt]
\left[\frac{7}{4}\right]_{q}&=&
\llbracket2,4\rrbracket_{q}&=&
[1,1,2,1]_{q}&=&
\frac{1+q+2q^2+2q^3+q^4}{1+q+q^2+q^3},\\[8pt]
\left[\frac{7}{5}\right]_{q}&=&
\llbracket2,2,3\rrbracket_{q}&=&
[1,2,1,1]_{q}&=&
\frac{1+q+2q^2+2q^3+q^4}{1+q+2q^2+q^3}.
\end{array}
$$

(c)
The $q$-rationals with denominator~$[2]_q$ are:
$$
\textstyle\left[\frac{2m+1}{2}\right]_{q}=
\frac{1+2q+2q^{2}+\cdots+2q^{m-1}+q^{m}+q^{m+1}}{1+q}.
$$

(d)
The $q$-rationals with denominator~$[3]_q$ are:
$$
\begin{array}{rcl}
\left[\frac{3m+1}{3}\right]_{q}&=&
\frac{1+2q+3q^{2}+3q^{3}+\cdots+3q^{m-1}+2q^{m}+q^{m+1}+q^{m+2}}{1+q+q^2},\\[6pt]
\left[\frac{3m+2}{3}\right]_{q}&=&\frac{1+2q+3q^{2}+3q^{3}+\cdots+3q^{m-1}+2q^{m}+2q^{m+1}+q^{m+2}}{1+q+q^2}.
\end{array}
$$

\end{ex}

We see that the degrees of the polynomials ~$\Rc$ and~$\Sc$ are in general lower than those of~$[r]_q$ and~$[s]_q$.

\subsection{Total positivity}

An important property of the polynomials $\Rc$ and $\Sc$ is the following property of positivity.

\begin{prop}
\label{PosProp}
For every $q$-rational~$\left[\frac{r}{s}\right]_{q}=\frac{\Rc(q)}{\Sc(q)}$,
the polynomials ${\Rc}$ and  ${\Sc}$ have positive integer coefficients.
\end{prop}

This property is not immediate from the expression \eqref{qc} nor from the expression \eqref{qa}. 
We give an elementary proof below in \S\ref{propRS}.
We will give a combinatorial interpretation of the coefficients;
see Section~\ref{EmunSec}.

Furthermore, Proposition~\ref{PosProp} has the following strengthened version.

\begin{thm}
\label{PosPropBis}
For every pair of $q$-rationals,~$\left[\frac{r}{s}\right]_{q}$ and~$\left[\frac{r'}{s'}\right]_{q}$,
the polynomial in~$q$
\begin{equation}
\label{PosDiagEq}
\cX_{\frac{r}{s},\frac{r'}{s'}}:=\Rc\Sc'-\Sc\Rc'
\end{equation}
has positive integer coefficients, provided~$\frac{r}{s}>\frac{r'}{s'}$.
\end{thm}

This statement is proved in Section~\ref{PT2Sec}.
The following corollary can be deduced by evaluating the polynomials at $q=1$.

\begin{cor}
\label{voisin}
Let $\frac{r}{s}>\frac{r'}{s'}$ be two rationals and let 
$\frac{\Rc}{\Sc}=\left[\frac{r}{s}\right]_{q}$ and $\frac{\Rc'}{\Sc'}=\left[\frac{r'}{s'}\right]_{q}$. 
One has, for some integer~$\alpha\geq 0$,
$$
\Rc\Sc'-\Sc\Rc' = q^{\alpha},
\quad 
\hbox{ if and only if }
\quad 
rs'-r's= 1.
$$
\end{cor}

The exact value of $\alpha$ is given in Proposition \ref{FormVoisin}.

\begin{rem}
The combinatorial viewpoint adopted in this paper makes it necessary to complete
the set of rationals with the ``greatest element'' $\infty=\frac{1}{0}$.
Its $q$-deformation is set $\left[\frac{1}{0}\right]_{q}:=\frac{1}{0}$.
Theorem~\ref{PosPropBis} then implies Proposition~\ref{PosProp},
when applied to~$\frac{r}{s}=\frac{1}{0}$, or~$\frac{r'}{s'}=\frac{0}{1}$.
\end{rem}

\subsection{Basic properties of the polynomials~$\Rc$ and~$\Sc$}\label{propRS}

Let $(c_{i})_{1\leq i \leq k}$ be a sequence of integers greater or equal to 2. Consider two sequences $(\Rc_i)_{i\geq 0}$ and $(\Sc_i)_{i\geq 0}$ of polynomials in $q$ defined by the following recursions
\begin{equation}
\label{LinRecEq}
\begin{array}{rcl}
\Rc_{i+1} &=& [c_{i+1}]_q\Rc_{i}-q^{c_i-1}\Rc_{i-1}, \\[4pt]
\Sc_{i+1} &=& [c_{i+1}]_q\Sc_{i}-q^{c_i-1}\Sc_{i-1},
\end{array}
\end{equation}
with the initial conditions $( \Rc_{0},\,\Rc_1)=(1,\,[c_1]_q)$ and $(\Sc_{0},\,\Sc_1)=(0,1)$.

The continued fraction $\llbracket{}c_1,\ldots,c_k\rrbracket_{q}$ given by~\eqref{qc}
 belongs to the class of ``generalized continued fractions''. 
 For this class it is well known that the  numerators and denominators  of the convergents
 satisfy recurrence relations given by the coefficients of the continued fractions. 
 In our case, one obtains the recursions~\eqref{LinRecEq},
 and therefore 
 $$
\llbracket{}c_1,\ldots,c_i\rrbracket_{q}=\frac{\Rc_i(q)}{\Sc_i(q)},
$$
for all $1\leq i\leq k$.

\begin{prop}
\label{Prop16}
For all $1\leq i\leq k$,

(i)
the leading term of $\Rc_{i}$ is $q^{c_1+\ldots +c_i-i}$ and the constant term is 1;

(ii)
the leading term of $\Sc_{i}$ is $q^{c_2+\ldots +c_i-i+1}$ and the constant term is 1;

(iii) 
$\Rc_{i}$ and $\Sc_{i}$ have positive integer coefficients;

(iv)
$\Rc_{i}$ and $\Sc_{i}$ are coprime.

\end{prop}

\begin{proof}
(i) and (ii) are obtained by easy induction.

Assume that $\Rc_{i}$ and $\Rc_{i}-q^{c_i-1}\Rc_{i-1}$ have positive integer coefficients for some fixed index $i$.
From the recursion one gets
$$
\Rc_{i+1}=q[c_{i+1}-1]_q\Rc_{i}+(\Rc_{i}-q^{c_i-1}\Rc_{i-1}),
$$
which implies that $\Rc_{i+1}$ also have positive integer coefficients, and one gets
$$
\Rc_{i+1}-q^{c_{i+1}-1}\Rc_{i}=q[c_{i+1}-2]_q\Rc_{i}+(\Rc_{i}-q^{c_i-1}\Rc_{i-1}),
$$
which implies that $
\Rc_{i+1}-q^{c_{i+1}-1}\Rc_{i}$ also have positive integer coefficients. 
(iii) follows then by induction.

From the following $2\times 2$ determinant computations
$$
\begin{vmatrix}
\Rc_{i+1}&\Rc_{i}\\[4pt]
\Sc_{i+1}&\Sc_{i}
\end{vmatrix}
=
\begin{vmatrix}
-q^{c_i-1}\Rc_{i-1}&\Rc_{i}\\[4pt]
-q^{c_i-1}\Sc_{i-1}&\Sc_{i}
\end{vmatrix}
=q^{c_i-1}
\begin{vmatrix}
\Rc_{i}&\Rc_{i-1}\\[4pt]
\Sc_{i}&\Sc_{i-1}
\end{vmatrix}
=
\ldots
=-q^{c_1+\ldots +c_i-i}
$$
one deduces that the only possible common divisors to $\Rc_{i}$ and $\Sc_{i}$ are powers of $q$. 
Since the polynomials have constant terms 1 they are not divisible by a power of $q$.
Hence (iv).
\end{proof}

In particular, for $i=k$ one obtains that $\Rc_{k}$ and $\Sc_{k}$ are precisely the polynomials $\Rc_{}$ and $\Sc_{}$ of \eqref{RcSc}, 
and one deduces the following.

\begin{cor}
\label{DegRSProp}
Let $\frac{\Rc(q)}{\Sc(q)}=[a_{1}, \ldots, a_{2m}]_{q}=\llbracket{}c_1,\ldots,c_k\rrbracket_{q}$.

(i)
The degrees of the polynomials ${\Rc}$ and ${\Sc}$ in~\eqref{RcSc} are 
$$
\begin{array}{rclcl}
\deg(\Rc)&=&c_{1}+\ldots+c_{k}-k&=&a_{1}+\ldots +a_{2m}-1,\\[4pt]
\deg(\Sc)&=&c_2+\ldots+c_{k}-k+1&=&a_{2}+\ldots +a_{2m}-1.
\end{array}
$$

(ii)
The coefficients of the leading terms of $\Rc$ and $\Sc$ are equal to~$1$.

(iii)
The constant terms of $\Rc$ and $\Sc$ are equal to~$1$, so that
$\Rc(0)=\Sc(0)=1$ and $\Rc(1)=r,\,\Sc(1)=s$.
\end{cor}

Let us mention that
Parts (ii) and (iii) of the corollary can also be deduced from the
$F$-polynomial interpretation (see Corollary~\ref{FPolCor}),
and, in this sense, this is an instance of the general theorem~\cite{DWZ}
that $F$-polynomials of cluster variables have constant term~$1$
and the leading coefficient~$1$.

\subsection{The special value~$q=-1$}\label{Valueq-1}

The value~$q=-1$ is special.
In this case we have the following.

\begin{prop}
\label{q-1Prop}
(i)
For every rational~$\frac{r}{s}$, the corresponding polynomials~$\Rc$ and~$\Sc$ evaluated at~$q=-1$ can take only three values:
$-1,0$, or~$1$.

(ii)
$\Rc(-1)=0$ if and only if $r$ is even.

(iii)
$\Sc(-1)=0$ if and only if $s$ is even.
\end{prop}

We will prove this statement in Section~\ref{Proof-1Sec}.

A corollary of Proposition~\ref{q-1Prop} Part (ii) is that, for even~$r$, the polynomial~$\Rc(q)$ is a multiple of~$(1+q)$,
and similarly for~$\Sc(q)$.

\section{Weighted triangulations and $q$-deformed continued fractions}\label{WTSec}

The goal of this section is to give an alternative, recurrent, way to define $q$-deformations of continued fractions.
We construct a combinatorial model for calculating the $q$-continued fractions
based on a simple triangulation of an $n$-gon.
The main difference comparing with the classical continued fractions is that the edges of the corresponding 
triangulation are weighted.
We show that the weights associated to $q$-deformed continued fractions can be read directly in the Farey graph.

\subsection{Continued fractions and the triangulations~$\mathbb{T}_{r/s}$}\label{Trs}
We briefly recall a relationship between
continued fractions and triangulations of $n$-gons
(for more details; see~\cite{FB1}).

Consider the following ``wrinkled triangulation'' of an $n$-gon with exactly two exterior triangles:
$$
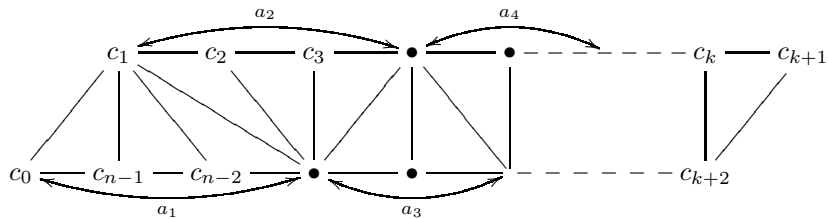
\xymatrix @!0 @R=0.8cm @C=1.3cm
{
&c_1\ar@{-}[ldd]\ar@{-}[dd]\ar@{-}[rdd]\ar@{-}[rrdd]\ar@{-}[r]\ar@/^0.8pc/@{<->}[rrr]^{a_2}
&c_2\ar@{-}[r]
\ar@{-}[rdd]&c_3\ar@{-}[r]\ar@{-}[dd]
&\bullet\ar@{-}[r]\ar@{-}[ldd]\ar@{-}[dd]\ar@{-}[rdd]\ar@{-}[r]\ar@/^0.8pc/@{<->}[rr]^{a_4}
&\bullet\ar@{--}[rr]\ar@{-}[dd]&& c_k\ar@{-}[r]\ar@{-}[dd]& c_{k+1}\ar@{-}[ldd]\\
\\
c_0\ar@{-}[r]\ar@/_0.8pc/@{<->}[rrr]_{a_1}
&c_{n-1}\ar@{-}[r]&c_{n-2}\ar@{-}[r]
&\bullet\ar@{-}[r]\ar@/_0.8pc/@{<->}[rr]_{a_3}
&\bullet\ar@{-}[r]&\ar@{--}[rr]&&c_{k+2}&
}
$$
Every wrinkled triangulation is defined by a sequence of positive integers,
$(a_1,\ldots,a_{2m})$.
The integers $a_i$ count the number of consecutive equally positioned triangles
(i.e., the triangles ``base-down'', or ``base-up''),
cf. Example~\ref{ExT75} below.

A wrinkled triangulation defines another sequence of positive integers, $(c_1,\ldots,c_n)$.
The integer $c_i$ counts the number of triangles incident at the vertex~$i$.

Let $k$ be the number of vertices between the exterior vertices.
The following statement is a reformulation of a formula from~\cite{Hir,HZ}.
{\it The regular and negative continued fractions defined by the sequences 
$(a_1,\ldots,a_{2m})$ and $(c_1,\ldots,c_k)$ coming from a triangulation, are equal to each other:}
$$
[a_1,\ldots,a_{2m}]=
\llbracket{}c_1,\ldots,c_k\rrbracket.
$$

Equivalently, the explicit formula of conversion for the coefficients of regular and
negative continued fractions is:
\begin{equation}
\label{HZRegEqt}
(c_1,\ldots,c_k)=
\big(a_1+1,\underbrace{2,\ldots,2}_{a_2-1},\,
a_3+2,\underbrace{2,\ldots,2}_{a_4-1},\ldots,
a_{2m-1}+2,\underbrace{2,\ldots,2}_{a_{2m}-1}\big).
\end{equation}

\noindent
{\bf Notation}.
The constructed triangulation will be denoted by~$\mathbb{T}_{r/s}$, where~$\frac{r}{s}$ is the rational number
defined by the above continued fractions, i.e., $\frac{r}{s}:=[a_1,\ldots,a_{2m}]=
\llbracket{}c_1,\ldots,c_k\rrbracket$.

\begin{rem}
Note that the number~$n$ is related with the coefficients of the continued fractions by
$$
\begin{array}{rcl}
n &=& a_1+a_2+\cdots+a_{2m}+2\\[4pt]
&=& c_1+c_2+\cdots+c_k-k+3.
\end{array}
$$
\end{rem}

\subsection{The Farey sum rule}\label{TrsFS}
The triangulation~$\mathbb{T}_{r/s}$ encodes the full sequence of convergents
of the continued fractions via the following {\it Farey sum} rule:
\begin{equation}
\label{FarSEq}
\frac{r}{s}\oplus\frac{r'}{s'}
:=\frac{r+r'}{s+s'}.
\end{equation}
We label the vertices of the $n$-gon by rationals
starting from $\frac{0}{1}$ and $\frac{1}{0}$
at vertices $0$ and $1$, and extending this labeling 
by the rule~\eqref{FarSEq}.
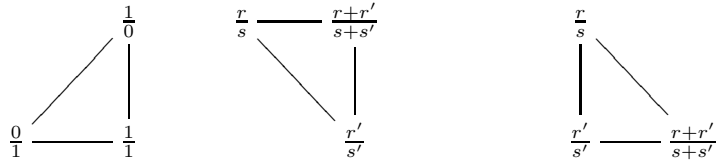
$$
\xymatrix @!0 @R=0.8cm @C=1.5cm
{
&\frac{1}{0}\ar@{-}[ldd]\ar@{-}[dd]
&\frac{r}{s}\ar@{-}[r]
\ar@{-}[rdd]&\frac{r+r'}{s+s'}\ar@{-}[dd]&&\frac{r}{s}\ar@{-}[dd]\ar@{-}[rdd]
&\\
\\
\frac{0}{1}\ar@{-}[r]
&\frac{1}{1}&&\frac{r'}{s'}
&&\frac{r'}{s'}\ar@{-}[r]
&\frac{r+r'}{s+s'}
}
$$

The following statement is equivalent to a theorem of C.~Series; see~\cite{Ser}, 
and it was first formulated in the context of hyperbolic geometry
and a Farey tessellation:
{\it the vertex $k+1$ of the triangulation~$\mathbb{T}_{r/s}$ has the label~$\frac{r}{s}$.}

\begin{ex}
\label{ExT75}
The rationals
$\frac{5}{2}=[2,2],\,\frac{5}{3}=[1,1,1,1],\,\frac{7}{5}=[1,2,1,1]$
correspond to the following triangulations
$$
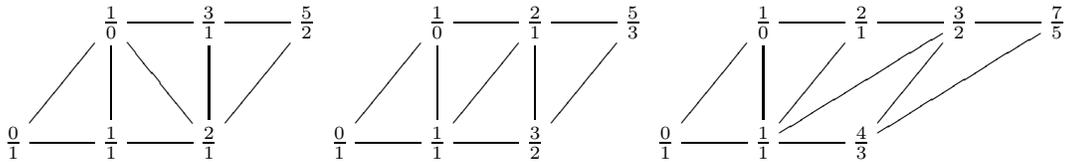
\xymatrix @!0 @R=0.8cm @C=1.3cm
{
&\frac{1}{0}\ar@{-}[ldd]\ar@{-}[dd]\ar@{-}[r]\ar@{-}[rdd]
&\frac{3}{1}\ar@{-}[r]\ar@{-}[dd]
&\frac{5}{2}\ar@{-}[ldd]
\\
\\
\frac{0}{1}\ar@{-}[r]
&\frac{1}{1}\ar@{-}[r]
&\frac{2}{1}
}
\xymatrix @!0 @R=0.8cm @C=1.3cm
{
&\frac{1}{0}\ar@{-}[ldd]\ar@{-}[dd]\ar@{-}[r]
&\frac{2}{1}\ar@{-}[ldd]\ar@{-}[r]\ar@{-}[dd]
&\frac{5}{3}\ar@{-}[ldd]
\\
\\
\frac{0}{1}\ar@{-}[r]
&\frac{1}{1}\ar@{-}[r]
&\frac{3}{2}
}
\xymatrix @!0 @R=0.8cm @C=1.3cm
{
&\frac{1}{0}\ar@{-}[ldd]\ar@{-}[dd]\ar@{-}[r]
&\frac{2}{1}\ar@{-}[ldd]\ar@{-}[r]
&\frac{3}{2}\ar@{-}[lldd]\ar@{-}[ldd]\ar@{-}[r]
&\frac{7}{5}\ar@{-}[lldd]\\
\\
\frac{0}{1}\ar@{-}[r]
&\frac{1}{1}\ar@{-}[r]
&\frac{4}{3}
}
$$
The coefficients in the negative continued fraction expansions
$\frac{5}{2}=\llbracket{}3,2\rrbracket,\,
\frac{5}{3}=\llbracket{}2,3\rrbracket,\,
\frac{7}{5}=\llbracket{}2,2,3\rrbracket$
are the numbers of the triangles attached to each vertex on the top line.
\end{ex}

\subsection{Weighted triangulations~$\mathbb{T}^q_{r/s}$}\label{WeTr}
Given a rational~$\frac{r}{s}$ and the corresponding triangulation~$\mathbb{T}_{r/s}$,
we define the {\it weighted triangulation} $\mathbb{T}^q_{r/s}$. 
Both, the vertices and the edges of~$\mathbb{T}^q_{r/s}$ are assigned a weight.

The edges of the triangulation are assigned a weight, 
which is a power of~$q$, while the vertices are labeled with $q$-rationals.

\begin{defn}
\label{WeightDef}
The edges are assigned weights according to the following rules.

(a)
The edges at the vertex with number~$1$ have the following weights:
$$
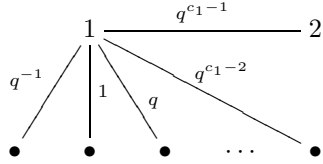
\xymatrix @!0 @R=0.8cm @C=1cm
{
&1\ar@{-}[ldd]_{q^{-1}}\ar@{-}[dd]^1\ar@{-}[rrr]^{q^{{c_1}-1}}\ar@{-}[rdd]^>>>>>q\ar@{-}[rrrdd]^{q^{{c_1}-2}}
&&&2
\\
\\
\bullet
&\bullet&\bullet&\cdots&\bullet
}
$$

(b)
The edges at an upper vertex with number~$i\geq2$ have the weights
defined according to the following rule:
$$
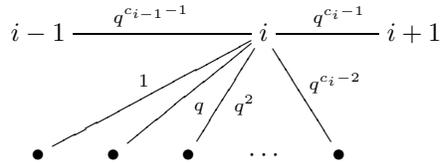
\xymatrix @!0 @R=0.8cm @C=1cm
{
i-1\ar@{-}[rrr]^{q^{c_{i-1}-1}}
&&&i\ar@{-}[ldd]^{q^2}\ar@{-}[lldd]^q\ar@{-}[llldd]_1\ar@{-}[rr]^{q^{c_{i}-1}}
\ar@{-}[rdd]^{q^{c_{i}-2}}
&&i+1
\\
\\
\bullet
&\bullet
&\bullet
&\cdots&\bullet
}
$$

(c)
All the edges between the bottom vertices have weight~$1$.

\end{defn}


\begin{rem}
(a)
The sides of the ``initial'' triangle $(\frac{0}{1},\frac{1}{1},\frac{1}{0})$ have weight~$1,1$ and~$q^{-1}$.
$$
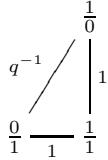
\xymatrix @!0 @R=0.8cm @C=1cm
{
&\frac{1}{0}\ar@{-}[ldd]_{q^{-1}}\ar@{-}[dd]^1
\\
\\
\frac{0}{1}\ar@{-}[r]_1
&\frac{1}{1}
}
$$
Note that the side of weight~$q^{-1}$ is not used in the computation of the continued fractions.
The weight~$q^{-1}$ will become clear in Section~\ref{qFSBSec}.

(b)
It follows from Definition~\ref{WeightDef} that the edges at a bottom vertex have the following weights:
$$
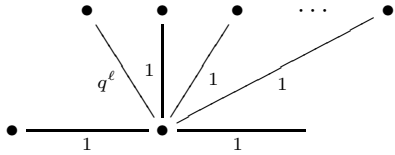
\xymatrix @!0 @R=0.8cm @C=1cm
{
&\bullet\ar@{-}[rdd]_{q^\ell}
&\bullet\ar@{-}[dd]_{1}
&\bullet\ar@{-}[ldd]^1&\cdots&\bullet\ar@{-}[llldd]^1
\\
\\
\bullet\ar@{-}[rr]_1
&&\bullet\ar@{-}[rr]_1&&
}
$$
In particular, at most one edge at a bottom vertex is of weight different from~$1$.

(c)
Sides of every triangle have the weights $(1,q^\ell,q^{\ell-1})$:
$$
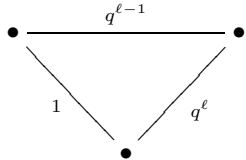
\xymatrix @!0 @R=0.8cm @C=1.5cm
{
\bullet\ar@{-}[rr]^{q^{\ell-1}}
\ar@{-}[rdd]_1&&\bullet\ar@{-}[ldd]^{q^\ell}
\\
\\
&\bullet
}
$$
for some~$\ell\geq1$.
\end{rem}

\begin{ex}
\label{ExT75W}
The weighted triangulations $\mathbb{T}^q_{5/2}, \mathbb{T}^q_{5/3},\mathbb{T}^q_{7/5}$ (see Example~\ref{ExT75}) are as follows
$$
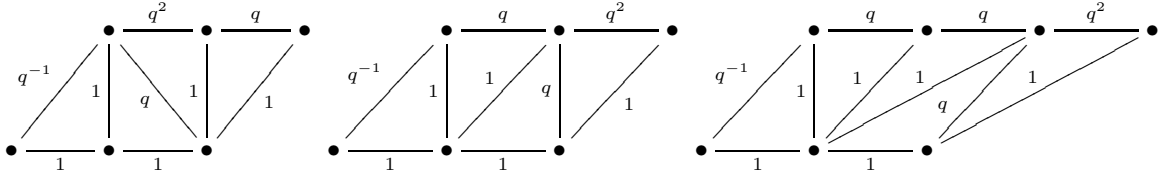
\xymatrix @!0 @R=0.8cm @C=1.3cm
{
&\bullet\ar@{-}[ldd]_{q^{-1}}\ar@{-}[dd]_1\ar@{-}[r]^{q^2}\ar@{-}[rdd]_q
&\bullet\ar@{-}[r]^q\ar@{-}[dd]_1
&\bullet\ar@{-}[ldd]^1
\\
\\
\bullet\ar@{-}[r]_1
&\bullet\ar@{-}[r]_1
&\bullet
}
\xymatrix @!0 @R=0.8cm @C=1.5cm
{
&\bullet\ar@{-}[ldd]_{q^{-1}}\ar@{-}[dd]_1\ar@{-}[r]^q
&\bullet\ar@{-}[ldd]_1\ar@{-}[r]^{q^2}\ar@{-}[dd]_q
&\bullet\ar@{-}[ldd]^1
\\
\\
\bullet\ar@{-}[r]_1
&\bullet\ar@{-}[r]_1
&\bullet
}
\xymatrix @!0 @R=0.8cm @C=1.5cm
{
&\bullet\ar@{-}[ldd]_{q^{-1}}\ar@{-}[dd]_1\ar@{-}[r]^q
&\bullet\ar@{-}[ldd]_1\ar@{-}[r]^q
&\bullet\ar@{-}[lldd]_1\ar@{-}[ldd]_<<<<<<<<<<<<<<q\ar@{-}[r]^{q^2}
&\bullet\ar@{-}[lldd]_1\\
\\
\bullet\ar@{-}[r]_1
&\bullet\ar@{-}[r]_1
&\bullet
}
$$
\end{ex}

\subsection{Labeling the vertices of~$\mathbb{T}^q_{r/s}$}\label{WeTrLV}

We will also need to label the vertices of the weighted triangulation~$\mathbb{T}^q_{r/s}$.

\begin{defn}
Let us label by the $q$-rational $\left[\frac{r'}{s'}\right]_{q}$
the vertex of $\mathbb{T}^q_{r/s}$, that corresponds to the rational~$\frac{r'}{s'}$ in $\mathbb{T}_{r/s}$. 
\end{defn}

The weights of $\mathbb{T}^{q}_{r/s}$ then encode algebraic relations between the fractions attached to the vertices.
More precisely, we have the following statement.

\begin{thm}
\label{WeightThm}
Every triangle of $\mathbb{T}^{q}_{r/s}$ is of the form
\begin{equation}\label{WFsum}
\xymatrix @!0 @R=0.8cm @C=1.5cm
{
\frac{\Rc'}{\Sc'}\ar@{-}[rr]^{q^{\ell-1}}
\ar@{-}[rdd]_1&&\frac{\Rc''}{\Sc''}\ar@{-}[ldd]^{q^\ell}
\\
\\
&\frac{\Rc'+q^\ell{}\Rc''}{\Sc'+q^\ell{}\Sc''}
}.
\end{equation}
\end{thm}
\begin{proof}
We suppose that $\frac{\Rc'}{\Sc'}=[\frac{r'}{s'}]_{q}$ and $\frac{\Rc''}{\Sc''}=[\frac{r''}{s''}]_{q}$ are attached to the vertices $j>k$ and $i\leq k$ respectively. 

Case 1. The local picture is as follows:
$$
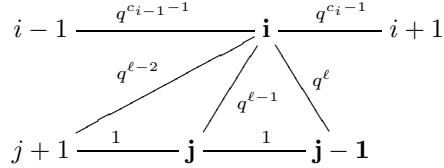
\xymatrix @!0 @R=0.8cm @C=1cm
{
i-1\ar@{-}[rrr]^{q^{c_{i-1}-1}}
&&&\mathbf{i}\ar@{-}[ldd]^{q^{\ell-1}}\ar@{-}[llldd]_{q^{\ell-2}}\ar@{-}[rr]^{q^{c_{i}-1}}
\ar@{-}[rdd]^{q^{\ell}}
&&i+1
\\
\\
j+1\ar@{-}[rr]^{1}
&
&\mathbf{j}\ar@{-}[rr]^{1}
&&\mathbf{j-1}
}
$$
We want to show that the fraction $\frac{\Rc'''}{\Sc'''}$ attached to the vertex $j-1$ is given by $\frac{\Rc'+q^\ell{}\Rc''}{\Sc'+q^\ell{}\Sc''}$.
Recall that the coefficients 
of the negative continued fractions are given by the number 
of triangles incident to the vertices (see \S \ref{Trs}). 
In particular, for the fraction $\frac{\Rc'}{\Sc'}$ or $\frac{\Rc'''}{\Sc'''}$ the coefficients are given 
by removing all the vertices and triangles between $i$ and $j$ or $i$ and $j-1$ respectively.
More precisely one gets
$$
\frac{\Rc''}{\Sc''}=\llbracket{}c_1,\ldots,c_{i-1}\rrbracket_{q}, \qquad 
\frac{\Rc'}{\Sc'}=\llbracket{}c_1,\ldots,c_{i-1},\ell\rrbracket_{q}, \qquad 
\frac{\Rc'''}{\Sc'''}=\llbracket{}c_1,\ldots,c_{i-1},\ell+1\rrbracket_{q}.
$$
 The desired relations between the numerators and between the denominators are then deduced from the recurrence relations~\eqref{LinRecEq}.

Case 2. The local picture is as follows:
$$
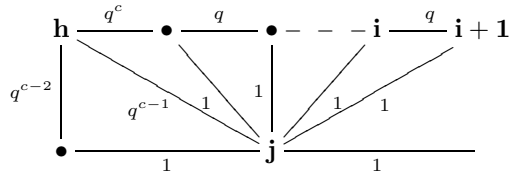
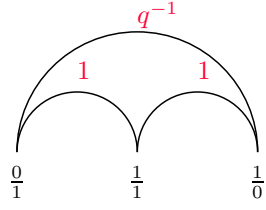
\xymatrix @!0 @R=0.8cm @C=1.4cm
{
\mathbf{h}\ar@{-}[rrdd]_{q^{c-1}}\ar@{-}[r]^{q^{c}}
&\bullet\ar@{-}[rdd]_{1}\ar@{-}[r]^{q}
&\bullet\ar@{-}[dd]_{1}\ar@{--}[r]
&\mathbf{i}\ar@{-}[ldd]^1\ar@{-}[r]^{q}&\mathbf{i+1}\ar@{-}[lldd]^1
\\
\\
\bullet\ar@{-}[rr]_1\ar@{-}[uu]^{q^{c-2}}
&&\mathbf{j}\ar@{-}[rr]_1&&
}
$$
Note that if $h=i-1$ one has $\ell=c=c_{i-1}-1>1$, otherwise $\ell=1$.

We want to show that the fraction $\frac{\Rc'''}{\Sc'''}$ attached to the vertex $i+1$
 is given by $\frac{\Rc'+q^\ell{}\Rc''}{\Sc'+q^\ell{}\Sc''}$. 
 We proceed by induction on $i-h$.
If $h=i-1$ one has
$$
\frac{\Rc'}{\Sc'}=\llbracket{}c_1,\ldots,c_{i-1}\rrbracket_{q}, \qquad 
\frac{\Rc''}{\Sc''}=\llbracket{}c_1,\ldots,c_{i-1}-1\rrbracket_{q}, \qquad 
\frac{\Rc'''}{\Sc'''}=\llbracket{}c_1,\ldots,c_{i-1},2\rrbracket_{q}.  
$$
Using the recurrence relations \eqref{LinRecEq}, one easily obtains the desired relation.

If we assumed established that  the fraction $\frac{\Rc^{(4)}}{\Sc^{(4)}}$ attached to the vertex $i-1$ satisfies $\frac{\Rc''}{\Sc''}=\frac{\Rc'+q\Rc^{(4)}}{\Sc'+q\Sc^{(4)}}$, then again one easily deduces the desired relation using the coefficients of the expansions and the recurrence relations.
\end{proof}

\subsection{Weighted Farey sum}
We generalize the Farey sum~\eqref{FarSEq} in the case of weighted triangulation,
using the formula that has already appeared in Theorem~\ref{WeightThm}.

\begin{defn}
\label{FSDef}
In the triangle \eqref{WFsum}, we call the fraction:
\begin{equation}
\label{WFSEq}
\frac{\Rc'}{\Sc'}\oplus_q\frac{\Rc''}{\Sc''}:=
\frac{\Rc'+q^\ell{}\Rc''}{\Sc'+q^\ell{}\Sc''},
\end{equation}
the \textit{{weighted Farey sum}} of 
$\frac{\Rc'}{\Sc'}$ and $\frac{\Rc''}{\Sc''}$.
 \end{defn} 

Theorem~\ref{WeightThm} then affirms that every
$q$-rational number~$\left[\frac{r}{s}\right]_q$ can be computed recursively
from the triangulation $\mathbb{T}^{q}_{r/s}$, applying the weighted Farey sum rule.

Proposition \ref{PosProp} and Corollary \ref{voisin} can be also obtained as consequences of Theorem~\ref{WeightThm}.

\begin{ex}
\label{ExqT75}
The $q$-rationals in $\mathbb{T}^q_{5/2}, \mathbb{T}^q_{5/3},\mathbb{T}^q_{7/5}$
can be obtained recursively. One obtains the following.
$$
\xymatrix @!0 @R=0.8cm @C=1.5cm
{
&\frac{1}{0}\ar@{-}[ldd]_{q^{-1}}\ar@{-}[dd]_{1}\ar@{-}[r]^>>>{q^2}\ar@{-}[rdd]_<<<<<<<{q}
&\frac{1+q+q^2}{1}\ar@{-}[rr]^q\ar@{-}[dd]_{1}
&&\frac{1+2q+q^2+q^3}{1+q}\ar@{-}[lldd]^1
\\
\\
\frac{0}{1}\ar@{-}[r]_1
&\frac{1}{1}\ar@{-}[r]_1
&\frac{1+q}{1}
}
\qquad
\xymatrix @!0 @R=0.8cm @C=1.5cm
{
&\frac{1}{0}\ar@{-}[ldd]_{q^{-1}}\ar@{-}[dd]_{1}\ar@{-}[r]^q
&\frac{1+q}{1}\ar@{-}[ldd]_1\ar@{-}[rr]^{q^2}\ar@{-}[dd]_q
&&\frac{1+q+2q^2+q^3}{1+q+q^2}\ar@{-}[lldd]^1
\\
\\
\frac{0}{1}\ar@{-}[r]_1
&\frac{1}{1}\ar@{-}[r]_1
&\frac{1+q+q^2}{1+q}
}
$$
The $q$-rational $\left[\frac{7}{5}\right]_{q}$ corresponds to
$$
\xymatrix @!0 @R=0.8cm @C=1.7cm
{
&\frac{1}{0}\ar@{-}[ldd]_{q^{-1}}\ar@{-}[dd]_{1}\ar@{-}[r]^q
&\frac{1+q}{1}\ar@{-}[ldd]_1\ar@{-}[r]^<<<<q
&\frac{1+q+q^2}{1+q}\ar@{-}[lldd]^1\ar@{-}[dd]^q\ar@{-}[rr]^<<<<<<<{q^2}
&&\frac{1+q+2q^2+2q^3+q^4}{1+q+2q^2+q^3}\ar@{-}[lldd]^{1}\\
\\
\frac{0}{1}\ar@{-}[r]_1
&\frac{1}{1}\ar@{-}[rr]_1
&&\frac{\;\;1+q+q^2+q^3}{1+q+q^2}
}
$$

\end{ex}

\subsection{Weighted Farey graph}\label{qFSBSec}

The classical {\it Farey graph} is the infinite graph with the set of vertices equal to
the set of rationals completed by the infinity~$\Q\cup\{\infty\}$,
the infinity being represented by~$\frac10$.
Two vertices,~$\frac{r}{s}$ and~$\frac{r'}{s'}$, are connected
if and only if $rs'-r's=\pm1$.
The Farey graph form an infinite triangulation of the hyperbolic plane,
called the {\it Farey tessellation}.
Every continued fraction can be understood as a walk along the edges of the Farey graph.
For the basic properties of the Farey graph and its relation to
continued fractions; see~\cite{HW}.

We assign a weight to every edge of the Farey graph,
and then label its vertices by the $q$-rational numbers
(instead of rational numbers).

\begin{defn}
\label{qFDef}
{\rm
(a)
The sides of the ``initial'' triangle $(\frac01,\frac11,\frac10)$ are weighted as follows
\begin{center}
\psscalebox{1.0 1.0} 
{
\begin{pspicture}(0,-1.315)(3.385,1.315)
\definecolor{colour0}{rgb}{1.0,0.0,0.2}
\psarc[linecolor=black, linewidth=0.02, dimen=outer](0.87,-0.685){0.8}{0.0}{180.0}
\psarc[linecolor=black, linewidth=0.02, dimen=outer](2.47,-0.685){0.8}{0.0}{180.0}
\psarc[linecolor=black, linewidth=0.02, dimen=outer](1.67,-0.685){1.6}{0.0}{180.0}
\rput[tl](0.87,0.515){\textcolor{colour0}{1}}
\rput[tl](2.47,0.515){\textcolor{colour0}{$1$}}
\rput[tl](1.67,1.315){\textcolor{colour0}{$q^{-1}$}}
\rput(0.07,-1.085){$\frac01$}
\rput(1.67,-1.085){$\frac11$}
\rput(3.27,-1.085){$\frac10$}
\end{pspicture}
}
\end{center}

(b)
The sides of every other triangle of the Farey graph are labeled according to the following rule.

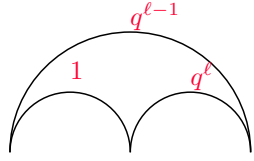
\begin{center}
\psscalebox{1.0 1.0} 
{
\begin{pspicture}(0,-1.315)(3.385,1.315)
\definecolor{colour0}{rgb}{1.0,0.0,0.2}
\psarc[linecolor=black, linewidth=0.02, dimen=outer](0.87,-0.685){0.8}{0.0}{180.0}
\psarc[linecolor=black, linewidth=0.02, dimen=outer](2.47,-0.685){0.8}{0.0}{180.0}
\psarc[linecolor=black, linewidth=0.02, dimen=outer](1.67,-0.685){1.6}{0.0}{180.0}
\rput[tl](0.87,0.515){\textcolor{colour0}{1}}
\rput[tl](2.47,0.515){\textcolor{colour0}{$q^\ell$}}
\rput[tl](1.67,1.315){\textcolor{colour0}{$q^{\ell-1}$}}
\end{pspicture}
}
\end{center}
}
\end{defn}

Clearly, the rules~(a) and~(b) uniquely determine the weights of all edges of the Farey graph.
Indeed, every edge is adjacent to two triangles in the Farey tessellation, and any two edges are related by a finite
sequences of triangles.
We will essentially be interested in the interval $[0,\infty)$ of the Farey graph.

Definition~\ref{qFDef} clarifies the weights in Definition~\ref{WeightDef}.
Indeed, the triangulation $\T^{q}_{r/s}$ can be embedded into the Farey graph (cf.~\cite{FB1}),
so that its weights are obtained as the pull-back of the weights on the Farey graph.
Therefore vertices are labeled according to the weighted Farey sum rule:
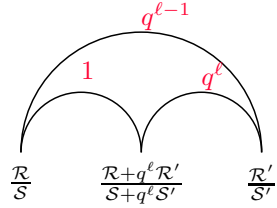
\begin{center}
\psscalebox{1.0 1.0} 
{
\begin{pspicture}(0,-1.315)(3.385,1.315)
\definecolor{colour0}{rgb}{1.0,0.0,0.2}
\psarc[linecolor=black, linewidth=0.02, dimen=outer](0.87,-0.685){0.8}{0.0}{180.0}
\psarc[linecolor=black, linewidth=0.02, dimen=outer](2.47,-0.685){0.8}{0.0}{180.0}
\psarc[linecolor=black, linewidth=0.02, dimen=outer](1.67,-0.685){1.6}{0.0}{180.0}
\rput[tl](0.87,0.515){\textcolor{colour0}{1}}
\rput[tl](2.47,0.515){\textcolor{colour0}{$q^\ell$}}
\rput[tl](1.67,1.315){\textcolor{colour0}{$q^{\ell-1}$}}
\rput(0.07,-1.085){$\frac{\Rc}{\Sc}$}
\rput(1.67,-1.085){$\frac{\Rc+q^{\ell}\Rc'}{\Sc+q^{\ell}\Sc'}$}
\rput(3.27,-1.085){$\frac{\Rc'}{\Sc'}$}
\end{pspicture}
}
\end{center}
Moreover, from the embedding of  $\T^{q}_{r/s}$ one deduces that $\ell$ is the last coefficient in the negative expansion of   $\frac{\Rc''}{\Sc''}:=\frac{\Rc+q^{\ell}\Rc'}{\Sc+q^{\ell}\Sc'}$, i.e. 
if  $\frac{\Rc''}{\Sc''}=\llbracket{}c_1,\ldots,c_{i}\rrbracket_{q}$ then $\ell=c_{i}$.

The top part of the weighted Farey graph is represented in Figure \ref{wtFg}.
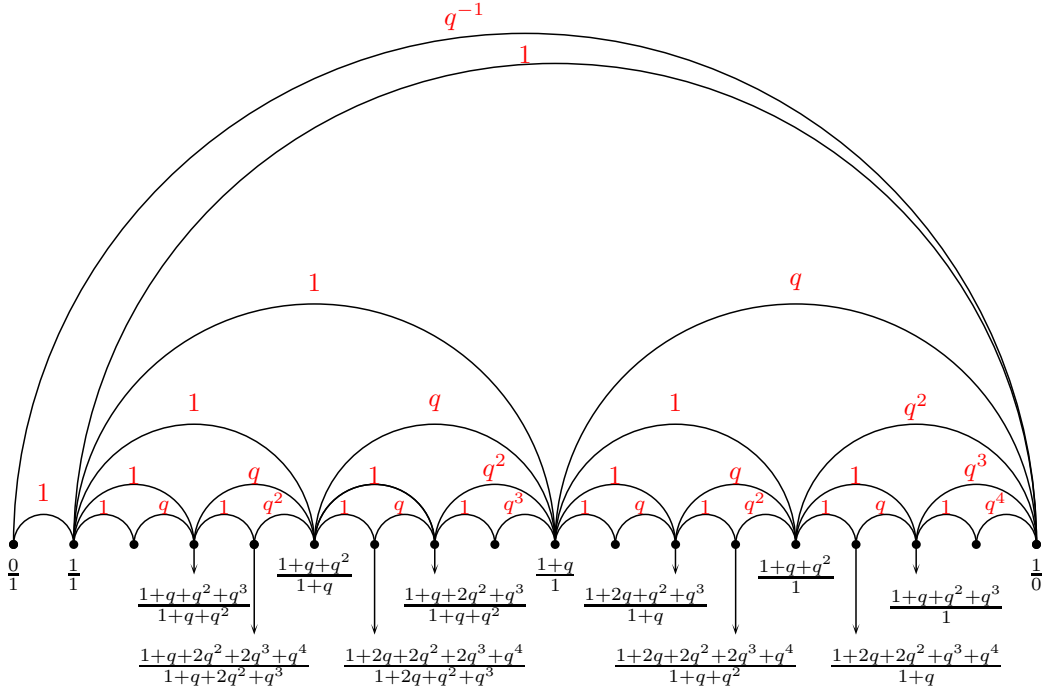
\begin{figure}[htbp]
\begin{center}
\psscalebox{1.0 1.0} 
{
\begin{pspicture}(0,-4.535)(13.757692,4.535)
\psdots[linecolor=black, dotsize=0.12](3.2788463,-2.665)
\psdots[linecolor=black, dotsize=0.12003673](3.2788463,-2.665)
\psdots[linecolor=black, dotsize=0.12](0.07884616,-2.665)
\psdots[linecolor=black, dotsize=0.12](13.678846,-2.665)
\psdots[linecolor=black, dotsize=0.12](2.478846,-2.665)
\psdots[linecolor=black, dotsize=0.12](1.6788461,-2.665)
\psdots[linecolor=black, dotsize=0.12](0.87884617,-2.665)
\psdots[linecolor=black, dotsize=0.12](0.07884616,-2.665)
\psdots[linecolor=black, dotsize=0.12](4.078846,-2.665)
\psdots[linecolor=black, dotsize=0.12](4.878846,-2.665)
\psdots[linecolor=black, dotsize=0.12](5.6788464,-2.665)
\psdots[linecolor=black, dotsize=0.12](6.478846,-2.665)
\psdots[linecolor=black, dotsize=0.12](7.2788463,-2.665)
\psdots[linecolor=black, dotsize=0.12](8.078846,-2.665)
\psdots[linecolor=black, dotsize=0.12](8.878846,-2.665)
\psdots[linecolor=black, dotsize=0.12](9.678846,-2.665)
\psdots[linecolor=black, dotsize=0.12](10.478847,-2.665)
\psdots[linecolor=black, dotsize=0.12](11.278846,-2.665)
\psdots[linecolor=black, dotsize=0.12](12.078846,-2.665)
\psdots[linecolor=black, dotsize=0.12](12.878846,-2.665)
\psarc[linecolor=black, linewidth=0.02, dimen=outer](6.878846,-2.665){6.8}{0.0}{180.0}
\psdots[linecolor=black, dotsize=0.12](13.678846,-2.665)
\psarc[linecolor=black, linewidth=0.02, dimen=outer](6.878846,-2.665){0.4}{0.0}{180.0}
\psarc[linecolor=black, linewidth=0.02, dimen=outer](7.6788464,-2.665){0.4}{0.0}{180.0}
\psarc[linecolor=black, linewidth=0.02, dimen=outer](8.478847,-2.665){0.4}{0.0}{180.0}
\psarc[linecolor=black, linewidth=0.02, dimen=outer](9.278846,-2.665){0.4}{0.0}{180.0}
\psarc[linecolor=black, linewidth=0.02, dimen=outer](10.078846,-2.665){0.4}{0.0}{180.0}
\psarc[linecolor=black, linewidth=0.02, dimen=outer](10.878846,-2.665){0.4}{0.0}{180.0}
\psarc[linecolor=black, linewidth=0.02, dimen=outer](11.678846,-2.665){0.4}{0.0}{180.0}
\psarc[linecolor=black, linewidth=0.02, dimen=outer](12.478847,-2.665){0.4}{0.0}{180.0}
\psarc[linecolor=black, linewidth=0.02, dimen=outer](13.278846,-2.665){0.4}{0.0}{180.0}
\psarc[linecolor=black, linewidth=0.02, dimen=outer](0.47884616,-2.665){0.4}{0.0}{180.0}
\psarc[linecolor=black, linewidth=0.02, dimen=outer](1.2788461,-2.665){0.4}{0.0}{180.0}
\psarc[linecolor=black, linewidth=0.02, dimen=outer](2.0788462,-2.665){0.4}{0.0}{180.0}
\psarc[linecolor=black, linewidth=0.02, dimen=outer](2.8788462,-2.665){0.4}{0.0}{180.0}
\psarc[linecolor=black, linewidth=0.02, dimen=outer](3.6788461,-2.665){0.4}{0.0}{180.0}
\psarc[linecolor=black, linewidth=0.02, dimen=outer](4.478846,-2.665){0.4}{0.0}{180.0}
\psarc[linecolor=black, linewidth=0.02, dimen=outer](5.2788463,-2.665){0.4}{0.0}{180.0}
\psarc[linecolor=black, linewidth=0.02, dimen=outer](6.078846,-2.665){0.4}{0.0}{180.0}
\psarc[linecolor=black, linewidth=0.02, dimen=outer](1.6788461,-2.665){0.8}{0.0}{180.0}
\psarc[linecolor=black, linewidth=0.02, dimen=outer](3.2788463,-2.665){0.8}{0.0}{180.0}
\psarc[linecolor=black, linewidth=0.02, dimen=outer](4.878846,-2.665){0.8}{0.0}{180.0}
\psarc[linecolor=black, linewidth=0.02, dimen=outer](6.478846,-2.665){0.8}{0.0}{180.0}
\psarc[linecolor=black, linewidth=0.02, dimen=outer](8.078846,-2.665){0.8}{0.0}{180.0}
\psarc[linecolor=black, linewidth=0.02, dimen=outer](9.678846,-2.665){0.8}{0.0}{180.0}
\psarc[linecolor=black, linewidth=0.02, dimen=outer](11.278846,-2.665){0.8}{0.0}{180.0}
\psarc[linecolor=black, linewidth=0.02, dimen=outer](12.878846,-2.665){0.8}{0.0}{180.0}
\psarc[linecolor=black, linewidth=0.02, dimen=outer](2.478846,-2.665){1.6}{0.0}{180.0}
\psarc[linecolor=black, linewidth=0.02, dimen=outer](4.878846,-2.665){0.8}{0.0}{180.0}
\psarc[linecolor=black, linewidth=0.02, dimen=outer](5.6788464,-2.665){1.6}{0.0}{180.0}
\psarc[linecolor=black, linewidth=0.02, dimen=outer](8.878846,-2.665){1.6}{0.0}{180.0}
\psarc[linecolor=black, linewidth=0.02, dimen=outer](12.078846,-2.665){1.6}{0.0}{180.0}
\psarc[linecolor=black, linewidth=0.02, dimen=outer](4.078846,-2.665){3.2}{0.0}{180.0}
\psarc[linecolor=black, linewidth=0.02, dimen=outer](10.478847,-2.665){3.2}{0.0}{180.0}
\psarc[linecolor=black, linewidth=0.02, dimen=outer](7.2788463,-2.665){6.4}{0.0}{180.0}
\rput[b](6.878846,3.735){\textcolor{red}{1}}
\rput[t](4.078846,0.935){\textcolor{red}{1}}
\rput[t](10.478847,0.935){\textcolor{red}{$q$}}
\rput[t](2.478846,-0.665){\textcolor{red}{1}}
\rput[t](5.6788464,-0.665){\textcolor{red}{$q$}}
\rput[t](8.878846,-0.665){\textcolor{red}{1}}
\rput[t](12.078846,-0.665){\textcolor{red}{$q^2$}}
\rput[t](12.878846,-1.465){\textcolor{red}{$q^3$}}
\rput[b](11.278846,-1.865){\textcolor{red}{1}}
\rput[b](8.078846,-1.865){\textcolor{red}{1}}
\rput[b](9.678846,-1.865){\textcolor{red}{$q$}}
\rput[b](1.6788461,-1.865){\textcolor{red}{1}}
\rput[b](3.2788463,-1.865){\textcolor{red}{$q$}}
\rput[t](6.478846,-1.465){\textcolor{red}{$q^2$}}
\rput[b](4.878846,-1.865){\textcolor{red}{1}}
\rput[t](0.47884616,-1.865){\textcolor{red}{1}}
\rput[t](6.078846,4.535){\textcolor{red}{$q^{-1}$}}
\rput(0.07884616,-3.065){$\frac01$}
\rput(0.87884617,-3.065){$\frac11$}
\rput(13.678846,-3.065){$\frac10$}
\rput(7.2788463,-3.065){$\frac{1+q}{1}$}
\rput(4.078846,-3.065){$\frac{1+q+q^2}{1+q}$}
\rput(10.478847,-3.065){$\frac{1+q+q^2}{1}$}
\rput(2.478846,-3.465){$\frac{1+q+q^2+q^3}{1+q+q^2}$}
\rput(2.8788462,-4.265){$\frac{1+q+2q^2+2q^3+q^4}{1+q+2q^2+q^3}$}
\rput(6.078846,-3.465){$\frac{1+q+2q^2+q^3}{1+q+q^2}$}
\rput(8.478847,-3.465){$\frac{1+2q+q^2+q^3}{1+q}$}
\rput(12.478847,-3.465){$\frac{1+q+q^2+q^3}{1}$}
\psline[linecolor=black, linewidth=0.02, arrowsize=0.05291667cm 2.2,arrowlength=1.5,arrowinset=0.6]{<-}(3.2788463,-3.865)(3.2788463,-2.665)
\psline[linecolor=black, linewidth=0.02, arrowsize=0.05291667cm 2.2,arrowlength=1.5,arrowinset=0.6]{<-}(2.478846,-3.065)(2.478846,-2.665)
\psline[linecolor=black, linewidth=0.02, arrowsize=0.05291667cm 2.2,arrowlength=1.5,arrowinset=0.6]{<-}(4.878846,-3.865)(4.878846,-2.665)(4.878846,-2.665)
\psline[linecolor=black, linewidth=0.02, arrowsize=0.05291667cm 2.2,arrowlength=1.5,arrowinset=0.6]{<-}(5.6788464,-3.065)(5.6788464,-2.665)
\psline[linecolor=black, linewidth=0.02, arrowsize=0.05291667cm 2.2,arrowlength=1.5,arrowinset=0.6]{<-}(8.878846,-3.065)(8.878846,-2.665)
\psline[linecolor=black, linewidth=0.02, arrowsize=0.05291667cm 2.2,arrowlength=1.5,arrowinset=0.6]{<-}(12.078846,-3.065)(12.078846,-2.665)
\rput[b](1.2788461,-2.265){\textcolor{red}{\footnotesize$1$}}
\rput[b](2.0788462,-2.265){\textcolor{red}{\footnotesize$q$}}
\rput[b](2.8788462,-2.265){\textcolor{red}{\footnotesize$1$}}
\rput[b](4.478846,-2.265){\textcolor{red}{\footnotesize$1$}}
\rput[b](6.078846,-2.265){\textcolor{red}{\footnotesize$1$}}
\rput[b](7.6788464,-2.265){\textcolor{red}{\footnotesize$1$}}
\rput[b](9.278846,-2.265){\textcolor{red}{\footnotesize$1$}}
\rput[b](10.878846,-2.265){\textcolor{red}{\footnotesize$1$}}
\rput[b](12.478847,-2.265){\textcolor{red}{\footnotesize$1$}}
\rput[br](3.6788461,-2.265){\textcolor{red}{\footnotesize$q^2$}}
\rput[br](5.2788463,-2.265){\textcolor{red}{\footnotesize$q$}}
\rput[br](6.878846,-2.265){\textcolor{red}{\footnotesize$q^3$}}
\rput[br](8.478847,-2.265){\textcolor{red}{\footnotesize$q$}}
\rput[br](10.078846,-2.265){\textcolor{red}{\footnotesize$q^2$}}
\rput[br](11.678846,-2.265){\textcolor{red}{\footnotesize$q$}}
\rput[br](13.278846,-2.265){\textcolor{red}{\footnotesize$q^4$}}
\rput(9.278846,-4.265){$\frac{1+2q+2q^2+2q^3+q^4}{1+q+q^2}$}
\rput(12.078846,-4.265){$\frac{1+2q+2q^2+q^3+q^4}{1+q}$}
\psline[linecolor=black, linewidth=0.02, arrowsize=0.05291667cm 2.2,arrowlength=1.5,arrowinset=0.6]{<-}(9.678846,-3.865)(9.678846,-2.665)
\psline[linecolor=black, linewidth=0.02, arrowsize=0.05291667cm 2.2,arrowlength=1.5,arrowinset=0.6]{<-}(11.278846,-3.865)(11.278846,-2.665)
\rput(5.6788464,-4.265){$\frac{1+2q+2q^2+2q^3+q^4}{1+2q+q^2+q^3}$}
\end{pspicture}
}
\caption{Upper part of the weighted Farey graph between~$\frac01$ and~$\frac10$}
\label{wtFg}
\end{center}
\end{figure}

\subsection{A comparison to the Gaussian $q$-binomial coefficients}\label{GaussSec}

There is a certain analogy between the $q$-continued fractions and Gaussian $q$-binomial coefficients.
Indeed, the $q$-binomials ${r\choose s}_q$ satisfy the relations
$$
{r\choose s}_q=
q^s{r-1\choose s}_q+{r-1\choose s-1}_q,
$$
and can be calculated recurrently using the ``weighted Pascal triangle'' that starts like this:
$$
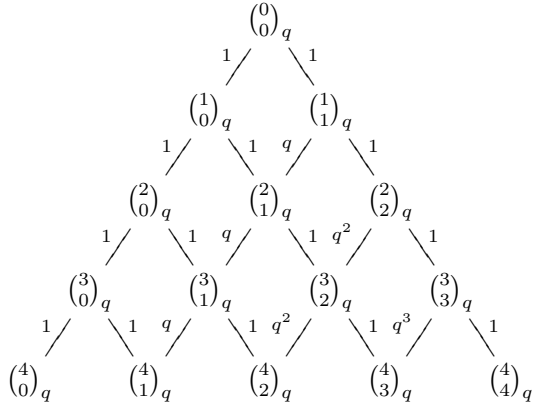
\xymatrix @!0 @R=0.6cm @C=0.8cm
{
&&&&{0\choose 0}_q\ar@{-}[rdd]^{1}\ar@{-}[ldd]_{1}
\\
\\
&&&{1\choose 0}_q\ar@{-}[rdd]^{1}\ar@{-}[ldd]_{1}
&&{1\choose 1}_q\ar@{-}[rdd]^{1}\ar@{-}[ldd]_{q}
\\
\\
&&{2\choose 0}_q\ar@{-}[rdd]^{1}\ar@{-}[ldd]_{1}
&&{2\choose 1}_q\ar@{-}[rdd]^{1}\ar@{-}[ldd]_{q}&&{2\choose 2}_q\ar@{-}[rdd]^{1}\ar@{-}[ldd]_{q^2}
\\
\\
&{3\choose 0}_q\ar@{-}[rdd]^{1}\ar@{-}[ldd]_{1}
&&{3\choose 1}_q\ar@{-}[rdd]^{1}\ar@{-}[ldd]_{q}
&&{3\choose 2}_q\ar@{-}[rdd]^{1}\ar@{-}[ldd]_{q^2}
&&{3\choose 3}_q\ar@{-}[rdd]^{1}\ar@{-}[ldd]_{q^3}
\\
\\
{4\choose 0}_q
&&{4\choose 1}_q&&{4\choose 2}_q&&{4\choose 3}_q&&{4\choose 4}_q
}
$$
with the first case ${4\choose 2}_q=1+q+2q^2+q^3+q^4$ where a $q$-binomial is not equal to a $q$-integer.

Let us mention that $q$-binomial coefficients do not occur as numerators 
(or denominators) of $q$-rationals,
except for the those $q$-binomials that are equal to $q$-integers.
The main resemblance between the $q$-rationals and $q$-binomials is the
recursion.
Note that every rational number is the middle point of exactly one Farey triangle,
and~\eqref{WFSEq} is then similar to the above recurrence for the $q$-binomials.
The main difference between $q$-rationals and $q$-binomials
is that, unlike the Pascal triangle, every vertex of the Farey graph is related with infinitely many other vertices.

\section{Counting on graphs}\label{EmunSec}

In this section we provide a combinatorial interpretation of the polynomials~$\Rc$ and~$\Qc$
of a $q$-rational~$\left[\frac{r}{s}\right]_q$.
Given a rational~$\frac{r}{s}$,
we consider an oriented graph of type~$A$ which is (projectively) dual to the
triangulation~$\T_{r/s}$.
We show that the coefficients of the polynomials~$\Rc(q)$ and~$\Qc(q)$ count 
some special subsets of vertices of the graph called ``the closures''.
Equivalently, this can be reformulated as counting of
subrepresentations of the maximal quiver representation of the associated oriented graph.

\subsection{Oriented graph associated with a rational}\label{Grs}
Given a rational~$\frac{r}{s}=[a_1,\ldots,a_{2m}]$,
consider the simple graph~$A_{n-3}$, where $n=a_1+\cdots+a_{2m}+2$, 
with $n-3$ vertices written in a row and connected by $n-4$ single edges:
$$
\xymatrix @!0 @R=0.8cm @C=1.3cm
{
\circ\ar@{-}[r]&\circ\ar@{-}[r]&\circ&\cdots
&\circ\ar@{-}[r]&\circ
}
$$
This graph is of course known as the Dynkin diagram of type $A_{n-3}$.

\begin{defn}
Let us choose the orientation of the edges according to the following rules:
\begin{enumerate}
\item[-]
first $a_1-1$ edges oriented 
left
\item[-]
next $a_2$ edges oriented 
right
\item[-]
next $a_3$ edges oriented 
left, 
\item[]
etc.
\item[-]
last $a_{2m}-1$ edges oriented right:
\end{enumerate}
$$
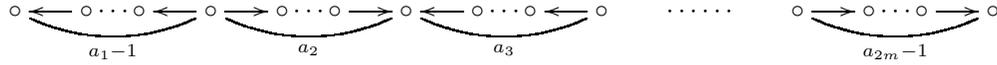
\xymatrix @!0 @R=0.8cm @C=1.3cm
{
\circ\ar@{<-}[r]\ar@/_0.8pc/@{-}[rr]_{a_1-1}&\circ\cdots\circ\ar@{<-}[r]&
\circ\ar@{->}[r]\ar@/_0.8pc/@{-}[rr]_{a_2}&\circ\cdots\circ\ar@{->}[r]&
\circ\ar@{<-}[r]\ar@/_0.8pc/@{-}[rr]_{a_3}&\circ\cdots\circ\ar@{<-}[r]&
\circ&\cdots\cdots&
\circ\ar@{->}[r]\ar@/_0.8pc/@{-}[rr]_{a_{2m}-1}&\circ\cdots\circ\ar@{->}[r]&\circ
}
$$
We denote by~$\G_{r/s}$ the constructed oriented graph.
This construction is the same as in \cite[Rem 6.3]{LeSc}.
\end{defn}

\begin{ex}
\label{GGraphEx}
(a)
The rationals~$\frac{5}{2},\frac{5}{3}$ and~$\frac{7}{5}$ correspond to
the graphs
$$
\xymatrix @!0 @R=0.8cm @C=1.3cm
{
\circ\ar@{<-}[r]&\circ\ar@{->}[r]&\circ
}
\qquad
\xymatrix @!0 @R=0.8cm @C=1.3cm
{
\circ\ar@{->}[r]&\circ\ar@{<-}[r]&\circ
}
\qquad
\xymatrix @!0 @R=0.8cm @C=1.3cm
{
\circ\ar@{->}[r]&\circ\ar@{->}[r]&\circ\ar@{<-}[r]&\circ
}
$$
respectively.

(b)
The rational~$\frac{25}{11}=[2,3,1,2]$ has the graph
$$
\xymatrix @!0 @R=0.8cm @C=1.3cm
{
\circ\ar@{<-}[r]&\circ\ar@{->}[r]&\circ\ar@{->}[r]&\circ\ar@{->}[r]&\circ\ar@{<-}[r]&\circ\ar@{->}[r]&\circ
}
$$
\end{ex}

We will also consider a shorter graph obtained from~$\G_{r/s}$ by removing the first~$a_1$ arrows.
It has the form:
$$
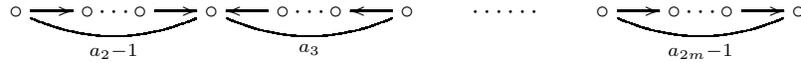
\xymatrix @!0 @R=0.8cm @C=1.3cm
{
\circ\ar@{->}[r]\ar@/_0.8pc/@{-}[rr]_{a_2-1}&\circ\cdots\circ\ar@{->}[r]&
\circ\ar@{<-}[r]\ar@/_0.8pc/@{-}[rr]_{a_3}&\circ\cdots\circ\ar@{<-}[r]&
\circ&\cdots\cdots&
\circ\ar@{->}[r]\ar@/_0.8pc/@{-}[rr]_{a_{2m}-1}&\circ\cdots\circ\ar@{->}[r]&\circ
}
$$
and will be denoted by~$\G'_{r/s}$.

\begin{rem}
\label{ClustRem}
The graph~$\G_{r/s}$ is just the standard graph that is usually associated to a triangulation.
In our case, the triangulation~$\T_{r/s}$ is particularly simple, so that
the corresponding graph is of type~$A$.
Note that~$\G_{r/s}$ may be thought of as {\it projectively dual}\footnote{
In projective duality straight lines become points, and the points of intersection of lines become lines joining them.}
to the triangulation~$\T_{r/s}$ in the following sense.
Vertices of~$\G_{r/s}$ are identified with diagonals of~$\T_{r/s}$,
and its edges are identified with the vertices of~$\T_{r/s}$ where the diagonals cross:
$$
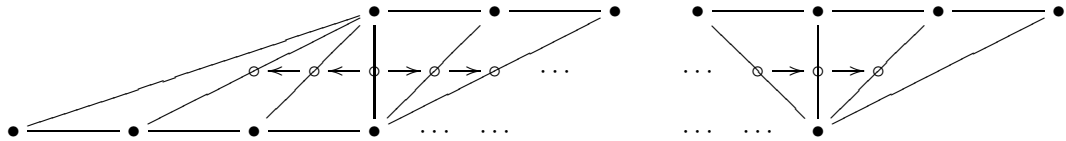
\xymatrix @!0 @R=0.8cm @C=0.8cm
{
&&&&
&&\bullet\ar@{-}[rr]\ar@{-}[dd]\ar@{-}[lldd]\ar@{-}[lllldd]\ar@{-}[lllllldd]
&&\bullet\ar@{-}[rr]\ar@{-}[lldd]&&\bullet\ar@{-}[lllldd]
\\
&&&&\circ\ar@{<-}[r]&\circ\ar@{<-}[r]&\circ\ar@{->}[r]&\circ\ar@{->}[r]&\circ&\cdots\\
\bullet\ar@{-}[rr]&&\bullet\ar@{-}[rr]&&\bullet\ar@{-}[rr]&&\bullet&\cdots&\cdots
}
\xymatrix @!0 @R=0.8cm @C=0.8cm
{
&\bullet\ar@{-}[rr]\ar@{-}[rrdd]
&&\bullet\ar@{-}[rr]\ar@{-}[dd]
&&\bullet\ar@{-}[lldd]\ar@{-}[rr]
&&\bullet\ar@{-}[lllldd]
\\
&\cdots&\circ\ar@{->}[r]&\circ\ar@{->}[r]&\circ\\
&\cdots&\cdots&\bullet
}
$$

The projectively dual graph associated with
a (more general) triangulation of an $n$-gon is constructed as follows.
Mark a puncture at every diagonal of the triangulation, and join by an edge
those punctures whose diagonals share a vertex of the $n$-gon.
Fix the orientation of every edge of the graph in such a way that the orientation of the 
triangle defined by the edge and the corresponding vertex of the $n$-gon be positive.
This leads to an oriented graph with $n-3$ vertices which is not necessarily bipartite and can contain cycles.

For example, the graph projectively dual to the following triangulated decagon
$$
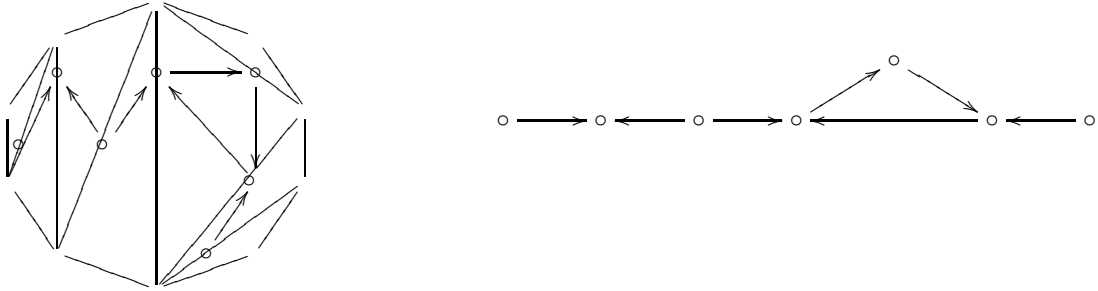
\xymatrix @!0 @R=0.48cm @C=0.66cm
 {
&&&\ar@{-}[lld]\ar@{-}[rrd]\ar@{-}[rrrddd]\ar@{-}[dddddddd]\ar@{-}[llddddddd]&
\\
&\ar@{-}[ldd]\ar@{-}[ldddd] \ar@{-}[dddddd]&&&& \ar@{-}[rdd]\\\
&\circ\ar@{<-}[lddd]&&\circ\ar@{->}[rr]&&\circ\ar@{->}[ddd]&&&&\\
\ar@{-}[dd]&&&&&&\ar@{-}[dd]\ar@{-}[lllddddd]\\
\;\;\;\circ&&\!\!\circ\ar@{->}[luu]\ar@{->}[ruu]&&&&&\\
\ar@{-}[rdd]&&&&&\!\!\!\circ\ar@{->}[lluuu]\ar@{<-}[ldd]&\ar@{-}[ldd]\ar@{-}[lllddd]\\
&&&&&&&\\
&\ar@{-}[rrd]&&&\circ& \ar@{-}[lld]\\
&&&&
}
\quad
\xymatrix @!0 @R=0.8cm @C=1.3cm
{
\\
&&&&\circ\ar@{->}[rd]\ar@{<-}[ld]
\\
\circ\ar@{->}[r]&\circ\ar@{<-}[r]&\circ\ar@{->}[r]&\circ\ar@{<-}[rr]&&\circ\ar@{<-}[r]&\circ
}
$$
contains a $3$-cycle.
This is due to the interior triangle (yet the graph is ``mutation equivalent'' to~$A_7$ 
in the sense of cluster algebra mutations).

\end{rem}

\subsection{Closure subsets}

Let us recall a standard notion of graph theory.

\begin{defn}
\label{CloDef}
{\rm
Given an oriented graph~$\G$,
a set of vertices $\cC\subset\G$ is called
a {\it closure} if there are no outgoing edges from the elements of~$\cC$ to vertices in the complement~$\G\setminus\cC$.
}
\end{defn}
A closure with $\ell$ elements will be called an \textit{$\ell$-vertex closure}.
A single vertex forms a $1$-vertex closure if and only if it is a {\it sink} of~$\G$.
More generally, an $\ell$-vertex closure can be thought of as a 
``generalized sink''.
Note that there can be edges between the vertices that belong to the closure.

The following statement is our main combinatorial result.

\begin{thm}
\label{EnumThm}
Let $\frac{r}{s}=[a_{1}, \ldots, a_{2m}]$
and $\frac{\Rc(q)}{\Sc(q)}=\left[\frac{r}{s}\right]_q$, then

(i)
the numerator is a polynomial
$
\Rc(q)=1+\r_1q+\r_2q^2+\cdots+\r_{n-4}q^{n-4}+q^{n-3},
$
such that
$$ 
\r_i =\#\{ \,i\text{-vertex closures of the graph}~\G_{r/s}\},
$$
 for all $i$;

(ii)
the denominator is a polynomial
$
\Sc(q)=1+\s_1q+\s_2q^2+\cdots+\s_{n-a_1-4}q^{n-a_1-4}+q^{n-a_1-3}
$
such that
$$ 
\s_i =\#\{ \,i\text{-vertex closures of the graph}~\G'_{r/s}\},
$$
 for all $i$.

\end{thm}
Theorem~\ref{EnumThm} will be proved in Section~\ref{PETSec}.

Let us give some examples.

\begin{ex}
(a)
Continuing Example~\ref{GGraphEx}, the graph~$\G_{5/2}$ has:
$$
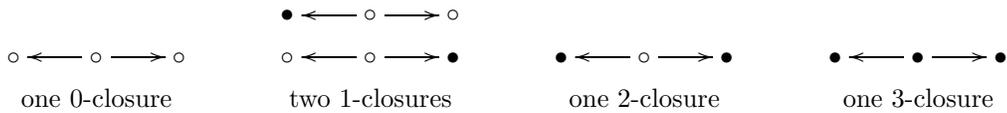
\begin{array}{c}
\\[4pt]
\xymatrix @!0 @R=0.8cm @C=1.1cm
{
\circ\ar@{<-}[r]&\circ\ar@{->}[r]&\circ
}
\\[4pt]
\hbox{one $0$-closure}
 \end{array}
 \qquad
 \begin{array}{c}
\xymatrix @!0 @R=0.8cm @C=1.1cm
{
\bullet\ar@{<-}[r]&\circ\ar@{->}[r]&\circ
}
\\[4pt]
\xymatrix @!0 @R=0.8cm @C=1.1cm
{
\circ\ar@{<-}[r]&\circ\ar@{->}[r]&\bullet
}
\\[4pt]
\hbox{two $1$-closures}
 \end{array}
\qquad 
\begin{array}{c}
\\[4pt]
\xymatrix @!0 @R=0.8cm @C=1.1cm
{
\bullet\ar@{<-}[r]&\circ\ar@{->}[r]&\bullet
}
\\[4pt]
\hbox{one $2$-closure}
 \end{array}
 \qquad
\begin{array}{c}
\\[4pt]
\xymatrix @!0 @R=0.8cm @C=1.1cm
{
\bullet\ar@{<-}[r]&\bullet\ar@{->}[r]&\bullet
}
\\[4pt]
\hbox{one $3$-closure}
 \end{array}
$$
 where the vertices forming a closure are colored black.
 The corresponding polynomial is $1+2q+q^2+q^3$, the numerator of~$\left[\frac52\right]_q$.
 
 (b)
Similarly, the graph~$\G_{5/3}$ has:
$$
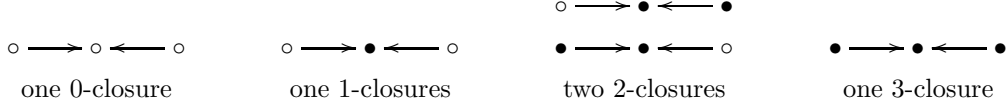
\begin{array}{c}
\\[4pt]
\xymatrix @!0 @R=0.8cm @C=1.1cm
{
\circ\ar@{->}[r]&\circ\ar@{<-}[r]&\circ
}
\\[4pt]
\hbox{one $0$-closure}
 \end{array}
 \qquad
 \begin{array}{c}
\\[4pt]
\xymatrix @!0 @R=0.8cm @C=1.1cm
{
\circ\ar@{->}[r]&\bullet\ar@{<-}[r]&\circ
}
\\[4pt]
\hbox{one $1$-closures}
 \end{array}
\qquad 
\begin{array}{c}
\xymatrix @!0 @R=0.8cm @C=1.1cm
{
\circ\ar@{->}[r]&\bullet\ar@{<-}[r]&\bullet
}
\\[4pt]
\xymatrix @!0 @R=0.8cm @C=1.1cm
{
\bullet\ar@{->}[r]&\bullet\ar@{<-}[r]&\circ
}
\\[4pt]
\hbox{two $2$-closures}
 \end{array}
 \qquad
\begin{array}{c}
\\[4pt]
\xymatrix @!0 @R=0.8cm @C=1.1cm
{
\bullet\ar@{->}[r]&\bullet\ar@{<-}[r]&\bullet
}
\\[4pt]
\hbox{one $3$-closure}
 \end{array}
$$
 The corresponding polynomial is $1+q+2q^2+q^3$, the numerator of~$\left[\frac53\right]_q$.
 
 (c)
 The graph~$\G_{25/11}$ (in which, for convenience, we number the vertices) is as follows
 $$
\xymatrix @!0 @R=0.8cm @C=1.3cm
{
1\ar@{<-}[r]&2\ar@{->}[r]&3\ar@{->}[r]&4\ar@{->}[r]&5\ar@{<-}[r]&6\ar@{->}[r]&7
}
$$
It obviously has one $0$-closure, one~$7$-closure, and three $1$-closures (three sinks, $1,5$ and $7$).
It furthermore has:
\begin{enumerate}
\item[-]
 four $2$-closures, namely the ``double sinks'' $(1,5),(1,7),(5,7)$ and~$(4,5)$;
\item[-]
 five $3$-closures $(1,5,7),(1,4,5),(3,4,5),(4,5,7),(5,6,7)$;
\item[-]
 five $4$-closures $(1,3,4,5),(1,4,5,6),(1,5,6,7),(3,4,5,7)$ and~$(4,5,6,7)$;
\item[-]
four $5$-closures obtained by removing one of the pairs of points:
$(1,2),(2,3),(1,6)$, or~$(6,7)$;
\item[-]
two $6$-closures, obtained by removing one of  the two sources $2$ or $6$.
 \end{enumerate}
The polynomial $\Rc$ in the numerator of $\left[\frac{25}{11}\right]_q$ is
indeed,
$$
\Rc(q)=1+3q+4q^2+5q^3+5q^4+4q^5+2q^6+q^7.
$$
(d) The graph $\G'_{25/11}$ is as follows
$$
\xymatrix @!0 @R=0.8cm @C=1.3cm
{
1\ar@{->}[r]&2\ar@{->}[r]&3\ar@{<-}[r]&4\ar@{->}[r]&5
}
$$
It has:
\begin{enumerate}
\item[-] 
one $0$-closure;
\item[-]
two $1$-closures (the two sinks, $3$ and $5$);
\item[-]
two  $2$-closures, namely the ``double sink'' $(3,5)$ and~$(2,3)$;
\item[-]
three $3$-closures $(1,2,3),(3,4,5), (2,3,5)$;
\item[-]
 two $4$-closures $(2,3,4,5),(1,2,3,5)$;
\item[-]
 one~$5$-closure.
 \end{enumerate}
On can compute the polynomial $\Sc$ corresponding to $\left[\frac{25}{11}\right]_q$ and check that
indeed,
$$
\Sc(q)=1+2q+2q^2+3q^3+2q^4+q^{5}.
$$
\end{ex}

\begin{rem}
(a)
As a corollary of Theorem~\ref{EnumThm}, we see that the total number of closure subsets 
of the graphs~$\G_{r/s}$ and~$\G'_{r/s}$
is equal to~$r$ and~$s$, respectively.
This observation provides an interesting combinatorial way to calculate the numerator and
denominator of a (classical) continued fraction.

(b)
It would be interesting to investigate counting of closure subsets of vertices
and the corresponding polynomials for more general graphs than~$\G_{r/s}$.
For instance, applied to the case of the graph projectively dual to a triangulation (see Remark~\ref{ClustRem})
this approach could lead to
a version of quantized Grassmannian~$G_{2,n}$.
\end{rem}

\subsection{A reformulation with quiver representations}\label{RepSec}

An equivalent reformulation of Theorem~\ref{EnumThm} uses the notion of quiver representations
(a ``quiver'' being another word for an oriented graph).
Let us recall a few standard definitions and facts.
We refer to~\cite{DW} for the general theory of quiver representations.

A {\it representation} of a quiver associates a vector space to each vertex, and a linear map to each edge.
A representation is {\it indecomposable} if it is not isomorphic to the direct sum of non-zero representations.
According to a theorem of Gabriel, indecomposable representations
of a quiver of finite (i.e., of Dynkin) type correspond to the positive roots of the root system of the Dynkin diagram.

The graph~$\G_{r/s}$ is a quiver of Dynkin type~$A$, it has one indecomposable representation of maximal dimension,
where every vector space is one-dimensional and every linear map is the identity map:
$$
\xymatrix @!0 @R=0.8cm @C=1.6cm
{
\C\ar@{<-}[r]^{\Id}&\C\cdots\C\ar@{<-}[r]^>>>{\Id}&
\C\ar@{->}[r]^<<<<{\Id}&\C\cdots\C\ar@{->}[r]^<<<<{\Id}&
\C&\cdots\cdots&
\C\ar@{->}[r]^<<<<{\Id}&\C\cdots\C\ar@{->}[r]^<<<<{\Id}&\C
}
$$
A {\it subrepresentation} of this representation replaces some of the spaces~$\C$ by~$\{0\}$ in such a way that
the identity maps remain consistent.

 \begin{prop}
 \label{ReProp}
There is a one-to-one correspondence between subrepresentations of the above maximal representation
and closure subsets.
\end{prop}

This statement is obvious from the definition.
Indeed, in this correspondence every vertex of a closure subset is assigned with~$\C$, and the vertices that
do not belong to the closure with~$\{0\}$.

As a consequence, we have the following.

\begin{cor}
The coefficients of the polynomials~$\Rc$ and~$\Sc$ corresponding to a $q$-rational
$\left[\frac{r}{s}\right]_q$ count subrepresentations
of the indecomposable representation of maximal dimension of~$\G_{r/s}$ and $\G'_{r/s}$,
respectively.
\end{cor}

This reformulation of Theorem~\ref{EnumThm} is more conceptual
and fits into the framework of $F$-polynomials and quiver representations;
see~\cite{DWZ} (for more details, cf. Remark~\ref{RB7} below).

\subsection{Proof of Theorem~\ref{EnumThm}}\label{PETSec}

We use the induction on~$n=a_1+\cdots+a_{2m}+2$.

The basis~$n=4$ is obvious: the graph~$\G_{2/1}$ consists in one point and has 
one empty and one $1$-vertex closure, that corresponds to the polynomial
$\Rc(q)=1+q$, while the graph~$\G'_{2/1}$ is empty and corresponds to the polynomial
$\Sc(q)=1$.

Let us prove the induction step.
Consider a rational~$\frac{r}{s}=[a_1,\ldots,a_{2m}]$.

Case {\bf A}. 
Assume that $a_{2m}\geq3$, then the triangulation~$\T_{r/s}$ ends with (at least) three
 base-up triangles, and the corresponding graph~$\G_{r/s}$ has (at least) two last edges
 oriented to the right.
 $$
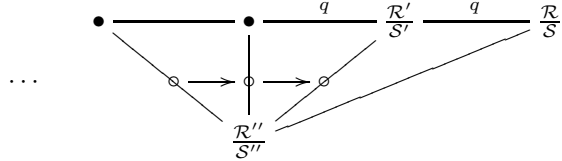
\xymatrix @!0 @R=0.8cm @C=1cm
{
&\bullet\ar@{-}[rr]\ar@{-}[rrdd]
&&\bullet\ar@{-}[rr]^q\ar@{-}[dd]
&&\frac{\Rc'}{\Sc'}\ar@{-}[lldd]\ar@{-}[rr]^q
&&\frac{\Rc}{\Sc}\ar@{-}[lllldd]
\\
\cdots&&\circ\ar@{->}[r]&\circ\ar@{->}[r]&\circ\\
&&&\frac{\Rc''}{\Sc''}
}
$$
By Definition~\ref{FSDef}, one has
\begin{equation}
\label{CompCloEq}
\Rc=q\Rc'+\Rc''.
\end{equation}

Consider the rational
$$
\frac{r'}{s'}=[a_1,\ldots,a_{2m-1},a_{2m}-1]
$$
whose triangulation~$\T_{r'/s'}$ (resp. graph~$\G_{r'/s'}$) is obtained by removing
one right triangle (resp. one right edge).
In the case of the rational 
$$
\frac{r''}{s''}=[a_1,\ldots,a_{2m-1}-1,1],
$$
the triangulation~$\T_{r''/s''}$ is obtained by removing~$a_{2m}$ base-up triangles in the right
and reversion of the last base-down triangle in the series of~$a_{2m-1}$ base-down triangles.
The graph~$\G_{r''/s''}$ is obtained by removing the last~$a_{2m}$ vertices.
The three graphs can be drawn in the same picture.
$$
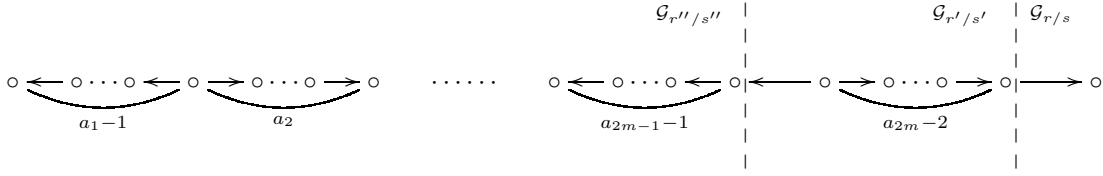
\xymatrix @!0 @R=1.2cm @C=1.2cm
{
&&&&&&&\ar@{}@<-2pt>[r]_{\G_{r''/s''}}&\ar@{--}@<4pt>[dd]&&\ar@{}@<-2pt>[r]_{\G_{r'/s'}}&\ar@{--}@<4pt>[dd]\ar@{}@<-2pt>[r]_{\G_{r/s}}&\\
\circ\ar@{<-}[r]\ar@/_0.8pc/@{-}[rr]_{a_1-1}&\circ\cdots\circ\ar@{<-}[r]&
\circ\ar@{->}[r]\ar@/_0.8pc/@{-}[rr]_{a_2}&\circ\cdots\circ\ar@{->}[r]&
\circ&\cdots\cdots&
\circ\ar@{<-}[r]\ar@/_0.8pc/@{-}[rr]_{a_{2m-1}-1}&\circ\cdots\circ\ar@{<-}[r]&\circ\ar@{<-}[r]&
\circ\ar@{->}[r]\ar@/_0.8pc/@{-}[rr]_{a_{2m}-2}&\circ\cdots\circ\ar@{->}[r]
& \circ\ar@{->}[r]&\circ\\
&&&&&&&&&&&&&&&
}
$$
Let us compare the closure subsets of~$\G_{r/s},\G_{r'/s'}$ and~$\G_{r''/s''}$.

\begin{lem}
One has:
\begin{equation}
\label{CompCloEqBis}
\r_{\ell+1}=\r'_{\ell}+\r''_{\ell+1},
\end{equation}
where $\r_{\ell+1},\r'_{\ell}$ and $r''_{\ell+1}$ are the numbers of
$(\ell+1)$-vertex closures of~$\G_{r/s}$, of $\ell$-vertex closures of~$\G_{r'/s'}$,
and $(\ell+1)$-vertex closures of~$\G_{r''/s''}$, respectively.
\end{lem}

\begin{proof}
Every $\ell$-vertex closure of~$\G_{r'/s'}$ corresponds to an $(\ell+1)$-vertex closure of~$\G_{r/s}$
by adding the last right vertex of~$\G_{r/s}$ (which is a sink).
Furthermore, every $\ell$-vertex closure of~$\G_{r''/s''}$ corresponds to 
an $\ell$-vertex closure of~$\G_{r/s}$,
that contains no vertex from the
right $a_{2m}-2$ vertices.
Hence the lemma.
\end{proof}

By the induction assumption, the coefficients of the polynomials~$R'(q)$ and~$R''(q)$ count
the closure subsets of the graphs~$\G_{r'/s'}$ and~$\G_{r''/s''}$, respectively.
Comparing formulas~\eqref{CompCloEq} and~\eqref{CompCloEqBis},
we conclude that the similar property also holds for~$\Rc(q)$.

The arguments in the case of the polynomial~$\Sc(q)$ are similar.
Theorem~\ref{EnumThm} is proved under the assumption~$a_{2m}\geq3$.

Case {\bf B}. 
Assume that $a_{2m}=2$, then the triangulation~$\T_{r/s}$ and the graph~$\G_{r/s}$
end as follows
 $$
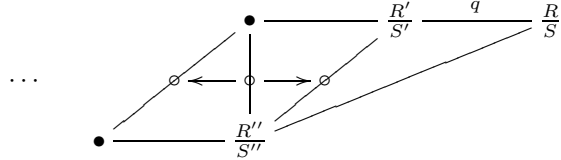
\xymatrix @!0 @R=0.8cm @C=1cm
{
&
&&\bullet\ar@{-}[rr]\ar@{-}[dd]\ar@{-}[lldd]
&&\frac{R'}{S'}\ar@{-}[lldd]\ar@{-}[rr]^q
&&\frac{R}{S}\ar@{-}[lllldd]
\\
\cdots&&\circ\ar@{<-}[r]&\circ\ar@{->}[r]&\circ\\
&\bullet\ar@{-}[rr]&&\frac{R''}{S''}
}
$$
The proof is similar in this case and is again based on~\eqref{CompCloEq} and~\eqref{CompCloEqBis}.

Case {\bf C}. 
Assume that~$a_{2m}=1$, 
then the triangulation~$\T_{r/s}$ and the graph~$\G_{r/s}$
end as follows
 $$
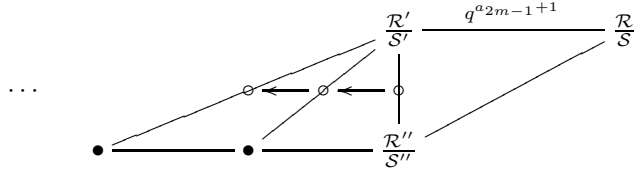
\xymatrix @!0 @R=0.8cm @C=1cm
{
&&&
&&\frac{\Rc'}{\Sc'}\ar@{-}[rrr]^{q^{a_{2m-1}+1}}\ar@{-}[dd]\ar@{-}[lldd]\ar@{-}[lllldd]
&&&\frac{\Rc}{\Sc}\ar@{-}[llldd]
\\
\cdots&&&\circ\ar@{<-}[r]&\circ\ar@{<-}[r]&\circ&\\
&\bullet\ar@{-}[rr]&&\bullet\ar@{-}[rr]&&\frac{\Rc''}{\Sc''}
}
$$
so that
$$
\Rc(q)=q^{a_{2m-1}+1}\Rc'(q)+\Rc''.
$$
The graph~$\G_{r'/s'}$ is obtained from~$\G_{r/s}$ by removing
$a_{2m-1}+1$ last vertices. 
Therefore, every $\ell$-vertex closure of~$\G_{r'/s'}$ corresponds to
$(\ell+a_{2m-1}+1)$-vertex closure of~~$\G_{r/s}$ by adding these vertices back.
On the other hand, an $\ell$-vertex closure of~$\G_{r''/s''}$ corresponds to
an $\ell$-vertex closure of~$\G_{r/s}$ that does not contain the last vertex.
We obtain, similarly to Lemma~\ref{CompCloEqBis},
$$
\r_{\ell}=\r'_{\ell-a_{2m-1}-1}+\r''_{\ell}.
$$
Theorem~\ref{EnumThm} is proved.

\section{Matrices of continued fractions and their $q$-deformations}\label{qMatSec}

It is convenient to represent continued fractions via unimodular $2\times2$~matrices
with integer coefficients.
In this section, we develop a similar technique in the $q$-deformed case.
This in particular allows us to prove Theorems~\ref{Thmqa},~\ref{PosPropBis} and several other technical statements.

\subsection{Matrices of convergents}\label{ClassMatSec}

We briefly recall the way to represent classical continued fractions by $2\times2$~matrices.
Given a rational~$\frac{r}{s}$ and its continued fractions expansions
$$
\frac{r}{s}=[a_1,\ldots,a_{2m}]=
\llbracket{}c_1,\ldots,c_k\rrbracket,
$$
consider the matrices
$$
\begin{array}{rcl}
M^+(a_1,\ldots,a_{2m}) &:=&
\left(
\begin{array}{cc}
a_1&1\\[4pt]
1&0
\end{array}
\right)
\left(
\begin{array}{cc}
a_2&1\\[4pt]
1&0
\end{array}
\right)
\cdots
\left(
\begin{array}{cc}
a_{2m}&1\\[4pt]
1&0
\end{array}
\right),
\\[16pt]
M(c_1,\ldots,c_k) &:=&
\left(
\begin{array}{cc}
c_1&-1\\[4pt]
1&0
\end{array}
\right)
\left(
\begin{array}{cc}
c_{2}&-1\\[4pt]
1&0
\end{array}
\right)
\cdots
\left(
\begin{array}{cc}
c_k&-1\\[4pt]
1&0
\end{array}
\right).
\end{array}
$$
They are well-known as the {\it matrices of convergents} of the continued fractions for the following reason.
One easily obtains:
$$
M^+(a_1,\ldots,a_{2m})=
\left(
\begin{array}{cc}
r&r'_{2m-1}\\[4pt]
s&s'_{2m-1}
\end{array}
\right),
\qquad
M(c_1,\ldots,c_k)=
\left(
\begin{array}{cc}
r&-r_{k-1}\\[4pt]
s&-s_{k-1}
\end{array}
\right),
$$
where
$\frac{r'_{2m-1}}{s'_{2m-1}}=[a_1,\ldots,a_{2m-1}]$,
and  $\frac{r_{k-1}}{s_{k-1}}=
\llbracket{}c_1,\ldots,c_{k-1}\rrbracket$.

In terms of the elementary matrices
$$
R=\left(
\begin{array}{cc}
1&1\\[4pt]
0&1
\end{array}
\right),
\qquad
L=\left(
\begin{array}{cc}
1&0\\[4pt]
1&1
\end{array}
\right),
\qquad
S=\left(
\begin{array}{cc}
0&-1\\[4pt]
1&0
\end{array}
\right),
$$
(note that $(R,S)$, or $(L,S)$ is the standard choice of generators of the group~$\SL(2,\Z)$)
the above matrices are expressed as follows:
$$
\begin{array}{rcl}
M^+(a_1,\ldots,a_{2m})&=&
R^{a_1}L^{a_{2}}
R^{a_{3}}L^{a_{4}}\cdots{}
R^{a_{2m-1}}L^{a_{2m}},
\\[4pt]
M(c_1,\ldots,c_k)&=&R^{c_1}S\,R^{c_{2}}S\cdots{}R^{c_k}S.
\end{array}
$$

\begin{rem}
Interest in the matrices $M^+(a_1,\ldots,a_{2m})$ and $M(c_1,\ldots,c_k)$ exceeds
the context of continued fractions; see, e.g.,~\cite{Zag,Kat,Boc,Ovs,FB1}.
Indeed, they appear in the theory of linear difference equations, as the
monodromy matrices. 
They play important roles in the theory of discrete integrable systems.
\end{rem}

\subsection{$q$-deformation of the matrices of convergents}\label{DefMatSec}

We define $q$-deformations of the matrices of convergents $M^{+}(a_{1}, \ldots, a_{2m})$ and $M(c_{1},\ldots, c_{n})$.

\begin{defn}
(i)
The matrix
\begin{equation}
\label{qNegMat}
M_{q}(c_{1},\ldots, c_{k})=
\begin{pmatrix}
[c_{1}]_{q}&-q^{c_{1}-1}\\[6pt]
1&0
\end{pmatrix}
\begin{pmatrix}
[c_{2}]_{q}&-q^{c_{2}-1}\\[6pt]
1&0
\end{pmatrix}
\cdots
\begin{pmatrix}
[c_{k}]_{q}&-q^{c_{k}-1}\\[6pt]
1&0
\end{pmatrix}
\end{equation}
is the $q$-analogue of the matrix $M(c_{1},\ldots, c_{k})$.

(ii)
The matrix
\begin{equation}
\label{qRegMat}
M^{+}_{q}(a_{1},\ldots, a_{2m})=
\begin{pmatrix}
[a_{1}]_{q}&q^{a_{1}}\\[6pt]
1&0
\end{pmatrix}
\begin{pmatrix}
[a_{2}]_{q^{-1}}& q^{-a_{2}}\\[6pt]
1&0
\end{pmatrix}
\cdots
\begin{pmatrix}
[a_{2m-1}]_{q}&q^{a_{2m-1}}\\[6pt]
1&0
\end{pmatrix}
\begin{pmatrix}
[a_{2m}]_{q^{-1}}&q^{-a_{2m}}\\[6pt]
1&0
\end{pmatrix}
\end{equation}
is the $q$-analogue of the matrix  $M^{+}(a_{1},a_{2}, \ldots, a_{2m})$.
\end{defn}
Note that the entries of the matrix ${M}^{+}_{q}(a_{1},a_{2}, \ldots, a_{2m})$ are polynomials in $q^{\pm1}$. 
We also define the following normalized matrix in order to get entries that are polynomials in $q$:
$$
\widetilde{M}^{+}_{q}(a_{1},\ldots, a_{2m}):=
q^{a_{2}+a_{4}+\ldots+a_{2m}}\,M^{+}_{q}(a_{1},\ldots, a_{2m}).
$$

The above definition is justified by the fact that the defined matrices are the matrices of convergents
of the $q$-deformed continued fractions~\eqref{qc} and~\eqref{qa}, respectively.

\begin{prop}
\label{FacMM}
(i) 
One has
$$M_{q}(c_{1},\ldots, c_{k})=
\begin{pmatrix}
\Rc&-q^{c_{k}-1}\Rc_{k-1}\\[6pt]
\Sc&-q^{c_{k}-1}\Sc_{k-1}
\end{pmatrix},
$$
where $\frac{\Rc(q)}{\Sc(q)}=
\llbracket{}c_{1}, \ldots,c_{k}\rrbracket_{q}$,  
$\frac{\Rc_{k-1}(q)}{\Sc_{k-1}(q)}=
\llbracket{}c_{1}, \ldots,c_{k-1}\rrbracket_{q}$.

(ii)
One has
$$
\widetilde{M}^{+}_{q}(a_{1},\ldots, a_{2m})=
\begin{pmatrix}
q\Rc&\Rc'_{2m-1}\\[6pt]
q\Sc&\Sc'_{2m-1}
\end{pmatrix}
$$
where $\frac{\Rc(q)}{\Sc(q)}=[a_{1}, \ldots, a_{2m}]_{q}$,  
$\frac{\Rc'_{2m-1}(q)}{\Sc'_{2m-1}(q)}=[a_{1}, \ldots,a_{2m-1}]_{q}$.
\end{prop}

\begin{proof}
This statement is an immediate consequence of the linear recurrence relations for the convergents;
cf. Section~\ref{propRS}.
\end{proof}

\subsection{$q$-deformation of the generators}\label{MatGenSec}

Let us introduce $q$-analogues of the matrices~$R,L$ and~$S$.

\begin{defn}
\label{RqLqSq}
{\rm
The $q$-analogues of the matrices~$R,L$ and~$S$ are
the following elements of~$\GL(2,\Z[q,q^{-1}])$
$$
R_{q}:=
\begin{pmatrix}
q&1\\[4pt]
0&1
\end{pmatrix},
\qquad
L_{q}:=
\begin{pmatrix}
1&0\\[4pt]
1&q^{-1}
\end{pmatrix},
\qquad
S_{q}:=
\begin{pmatrix}
0&-q^{-1}\\[4pt]
1&0
\end{pmatrix}.
$$
}
\end{defn}

\begin{prop}
\label{PropMM}
One has
\begin{eqnarray}
\label{MaGen}
M^{+}_{q}(a_{1},\ldots, a_{2m})&=&
R_{q}^{a_{1}}L_{q}^{a_{2}}\cdots R_{q}^{a_{2m-1}}L_{q}^{a_{2m}},\\[6pt]
\label{McGen}
M_{q}(c_{1},\ldots, c_{k})&=&
R_{q}^{c_{1}}S_{q}R_{q}^{c_{2}}S_{q} \cdots R_{q}^{c_{k}}S_{q}.
\end{eqnarray}
\end{prop}

\begin{proof}
One easily gets the following expressions for the iterates of $R_{q}$ and $L_{q}$:
$$
R^{a}_{q}:=
\begin{pmatrix}
q^{a}&[a]_{q}\\[4pt]
0&1
\end{pmatrix},
\qquad
L^{a}_{q}:=
\begin{pmatrix}
1&0\\[4pt]
[a]_{q^{-1}}&q^{-a}
\end{pmatrix}.
$$
By direct computations one checks
$$
R_{q}^{a_{i}}L_{q}^{a_{i+1}}=
\begin{pmatrix}
q^{a_{i}}+[a_{i}]_{q}[a_{i+1}]_{q^{-1}}&q^{-a_{i+1}}[a_{i}]_{q}\\[4pt]
[a_{i+1}]_{q^{-1}}&q^{-a_{i+1}}
\end{pmatrix}
=\begin{pmatrix}
[a_{i}]_{q}&q^{a_{i}}\\[6pt]
1&0
\end{pmatrix}
\begin{pmatrix}
[a_{i+1}]_{q^{-1}}& q^{-a_{i+1}}\\[6pt]
1&0
\end{pmatrix}.$$
Hence the formula for ${M}^{+}_{q}(a_{1},\ldots, a_{2m})$. 

Similarly, one easily checks
$$
R^{c_{i}}_{q}S_{q}=\begin{pmatrix}
[c_{i}]_{q}&-q^{c_{i}-1}\\[6pt]
1&0
\end{pmatrix}.
$$
Hence the formula for $M_{q}(c_{1},\ldots, c_{k})$. 
\end{proof}

\subsection{Local surgery operations}

We describe two useful ``local surgery'' operations.
These are $q$-analogues of the operations used in~\cite{Ovs,FB1}.

The first operation inserts~$1$ into the sequence $(c_1,c_2,\ldots,c_n)$, 
increasing the two neighboring entries by~$1$;
the second operation breaks one entry,~$c_i$, replacing it by
$c'_i,c''_i\in\Z_{>0}$ such that
$
c'_i+c''_i=c_i+1,
$ 
and inserts two consecutive~$1$'s between them.

\begin{lem}
\label{qLSLem}
The matrix $M_q(c_1,\ldots,c_n)$ has the following identities:
\begin{eqnarray}
\label{qFirstSurM}
 M_q(c_1,\ldots,c_i+1,\,1,\,c_{i+1}+1,\ldots,c_n) &=&
q\,M_q(c_1,\ldots,c_i,c_{i+1},\ldots,c_n),\\[6pt]
\label{qSecSurM}
M_q(c_1,\ldots,c'_i,\,1,\,\,1,\,c''_i,\ldots,c_n) &=&
-M_q(c_1,\ldots,c_i,\ldots,c_n),
\end{eqnarray}
where $c'_i+c''_i=c_i+1$.
\end{lem}

\begin{proof}
Formula~\eqref{qFirstSurM} follows from the following simple computation:
$$
\begin{pmatrix}
[c+1]_{q}&-q^{c}\\[4pt]
1&0
\end{pmatrix}
\begin{pmatrix}
1&-1\\[4pt]
1&0
\end{pmatrix}
\begin{pmatrix}
[d+1]_{q}&-q^{d}\\[4pt]
1&0
\end{pmatrix}=
q
\begin{pmatrix}
[c]_{q}&-q^{c-1}\\[4pt]
1&0
\end{pmatrix}
\begin{pmatrix}
[d]_{q}&-q^{d-1}\\[4pt]
1&0
\end{pmatrix}.
$$
To prove~\eqref{qSecSurM}, consider the product of matrices
$$
\begin{pmatrix}
[c'+1]_{q}&-q^{c'}\\[4pt]
1&0
\end{pmatrix}
\begin{pmatrix}
0&-1\\[4pt]
1&-1
\end{pmatrix}
\begin{pmatrix}
[c''+1]_{q}&-q^{c''}\\[4pt]
1&0
\end{pmatrix}=
\begin{pmatrix}
q^{c'}-[c'+1]_{q}-q^{c'}[c''+1]_{q}&q^{c'+c''}\\[4pt]
-1&0
\end{pmatrix}
$$
and observe that $q^{c'}-[c'+1]_{q}-q^{c'}[c''+1]_{q}=-[c'+c''+1]_{q}$.
\end{proof}

\begin{rem}
The operation~\eqref{qSecSurM} and the corresponding computation is a variation on the
following analogue of the ``sum of $q$-integers'':
$$
[a]_q\,+_q\,[b]_q:=
[a]_q+q^a[b]_q=
q^b[a]_q+[b]_q=
[a+b]_q,
$$
sometimes used in the $q$-differential calculus.
\end{rem}

\subsection{Solving the equation $M_q(c_1,\ldots,c_n)=
\pm{}q^\ell\,\Id$}\label{qCoCoIntSec}

The theorem of Conway and Coxeter~\cite{CoCo} and its generalization~\cite{Ovs}
classify positive integer solutions of the equation
$M(c_1,\ldots,c_n)=\pm\Id$.
They are formulated in terms of dissections of a convex $n$-gon. 
We give a $q$-analogue of these theorems.

We believe that Lemma~\ref{qLSLem} may have an independent interest.
Here we use it in order to solve the equation
$$
M_q(c_1,\ldots,c_n)=
\pm{}q^\ell\,\Id,
$$
where $c_i$ are positive integers.

Consider a convex $n$-gon and its dissection by non-crossing diagonals.
A class of dissections, such that the number of vertices of every sub-polygon is a multiple of~$3$
was considered in~\cite{Ovs} and called ``$3d$-dissections''.
Following~\cite{CoCo}, call a ``quiddity'' of a dissection the sequence of positive integers,
$(c_{1}, \ldots, c_{n})$, such that every~$c_i$ is the number of sub-polygons incident with the $i$-th vertex.

\begin{cor}
\label{qCoCoOvThm}
Let $(c_{1}, \ldots, c_{n})$ be a sequence of positive integers.

(i) If $(c_{1}, \ldots, c_{n})$ is the quiddity sequence of a triangulated $n$-gon then
$
M_{q}(c_{1},\ldots, c_{n})=
-q^{n-3}\,
\Id
$
.

(ii)
The equality
$
M_{q}(c_{1},\ldots, c_{n})=
\pm q^{n-3}\,
\Id
$
holds if and only if $(c_{1}, \ldots, c_{n})$ is the quiddity sequence of a $3d$-dissected $n$-gon.
\end{cor}

The proof of this statement is similar to the proof of Theorem~1 of~\cite{Ovs}, using the relations of Lemma~\ref{qLSLem}.

\subsection{Proof of Theorem~\ref{Thmqa}}\label{PT1Sec}

Our proof of Theorem~\ref{Thmqa} is a direct computation with the matrices of convergents.

\begin{prop}
\label{PropMMBis}
 If $\frac{r}{s}=[a_{1}, \ldots, a_{2m}]=\llbracket{}c_{1}, \ldots,c_{k}\rrbracket$, then
 the $q$-deformed matrices of the continued fractions satisfy
$$
\widetilde{M}^{+}_{q}(a_{1},\ldots, a_{2m})=M_{q}(c_{1},\ldots, c_{k})R_{q}.$$
\end{prop}

\begin{proof}
Suppose that $\frac{r}{s}=[a_{1}, \ldots, a_{2m}]=\llbracket{}c_{1}, \ldots,c_{k}\rrbracket$  and  use
the conversion formula \eqref{HZRegEqt}. One has
\begin{eqnarray*}
M_{q}(c_{1},\ldots, c_{k})&=&
M_{q}\big(a_1+1,\underbrace{2,\ldots,2}_{a_2-1},\,
a_3+2,\underbrace{2,\ldots,2}_{a_4-1},\ldots,
a_{2m-1}+2,\underbrace{2,\ldots,2}_{a_{2m-1}}\big)
\end{eqnarray*}
This can be rewritten using Proposition~\ref{PropMM} and changing the parentheses as
\begin{eqnarray*}
M_{q}(c_{1},\ldots, c_{k})R_{q}
&=&
R_{q}^{a_{1}+1}S_{q}\left(R^{2}_{q}S_{q}\right)^{a_{2}-1}
R_{q}^{a_{3}+2}S_{q}\left(R^{2}_{q}S_{q}\right)^{a_{4}-1}
\cdots 
R_{q}^{a_{2m-1}+2}S_{q}\left(R^{2}_{q}S_{q}\right)^{a_{2m}-1}R_{q}\\[4pt]
&=&
R_{q}^{a_{1}}{\left(R_{q}S_{q}R_{q}\right)^{a_{2}} }R_{q}^{a_{3}}
{\left(R_{q}S_{q}R_{q}\right)^{a_{4}} }
\cdots R_{q}^{a_{2m-1}}
{\left(R_{q}S_{q}R_{q}\right)^{a_{2m}} }.
\end{eqnarray*}
One then checks by a direct computation that
$$
R_{q}S_{q}R_{q}=qL_{q},
$$
and therefore $(R_{q}S_{q}R_{q})^{a}=q^{a}L_{q}^{a}$. This yields
\begin{eqnarray*}
M_{q}(c_{1},\ldots, c_{k})R_{q}
&=&q^{a_{2}+a_{4}+\ldots+a_{2m}}R_{q}^{a_{1}}L_{q}^{a_{2}}
R_{q}^{a_{3}}L_{q}^{a_{4}}\cdots R_{q}^{a_{2m-1}}L_{q}^{a_{2m}}\\[4pt]
&=&q^{a_{2}+a_{4}+\ldots+a_{2m}}{M}^{+}_{q}(a_{1},\ldots, a_{2m})\\[4pt]
&=&\widetilde{M}^{+}_{q}(a_{1},\ldots, a_{2m}).
\end{eqnarray*}
Proposition~\ref{PropMMBis} is proved.
\end{proof}

Theorem~\ref{Thmqa} readily follows from Proposition~\ref{PropMMBis} and Proposition~\ref{FacMM}.

\subsection{Proof of Theorem~\ref{PosPropBis}}\label{PT2Sec}
We now prove Theorem~\ref{PosPropBis} by establishing an explicit formula 
for the polynomial 
$$
\cX_{\frac{r}{s},\frac{r'}{s'}}=\Rc\Sc'-\Sc\Rc',
$$ 
up to a power of $q$.

Denote by $\mathbb{T}_{\frac{r}{s},\frac{r'}{s'}}$ the complete subgraph of the weighted 
Farey graph which is the union of ~$\mathbb{T}_{\frac{r}{s}}^{}$ and~$\mathbb{T}_{\frac{r'}{s'}}^{}$. 
The graph $\mathbb{T}_{\frac{r}{s},\frac{r'}{s'}}$ is a triangulated $N$-gon 
containing both fractions, $ \frac{r'}{s'}$  and $\frac{r}{s}$.
Figure \ref{doubletriang} below gives an example for $\frac{r}{s}=\frac75$, $\frac{r'}{s'}=\frac83$.

\begin{figure}[htbp]
\begin{center}
\psscalebox{1.0 1.0} 
{\psset{unit=0.75cm}
\begin{pspicture}(0,-3.7125)(13.757692,3.7125)
\definecolor{colour0}{rgb}{0.8,0.6,1.0}
\definecolor{colour1}{rgb}{0.8,0.8,1.0}
\definecolor{colour2}{rgb}{1.0,0.8,1.0}
\psarc[linecolor=black, linewidth=0.04, fillstyle=solid,fillcolor=colour0, dimen=outer](6.878846,-3.1075){6.8}{0.0}{180.0}
\psarc[linecolor=black, linewidth=0.04, fillstyle=solid, dimen=outer](0.47884616,-3.1075){0.4}{0.0}{180.0}
\psarc[linecolor=black, linewidth=0.04, fillstyle=solid,fillcolor=colour0, dimen=outer](7.2788463,-3.1075){6.4}{0.0}{180.0}
\rput(0.07884616,-3.5075){$\frac01$}
\rput(0.87884617,-3.5075){$\frac11$}
\rput(2.478846,-3.5075){$\frac43$}
\rput(3.2788463,-3.5075){$\frac75$}
\rput(4.078846,-3.5075){$\frac32$}
\rput(7.2788463,-3.5075){$\frac21$}
\rput(13.678846,-3.5075){$\frac10$}
\psarc[linecolor=black, linewidth=0.04, fillstyle=solid,fillcolor=colour1, dimen=outer](4.078846,-3.1075){3.2}{0.0}{180.0}
\psarc[linecolor=black, linewidth=0.04, fillstyle=solid,fillcolor=colour1, dimen=outer](2.478846,-3.1075){1.6}{0.0}{180.0}
\psarc[linecolor=black, linewidth=0.04, fillstyle=solid, dimen=outer](1.6788461,-3.1075){0.8}{0.0}{180.0}
\psarc[linecolor=black, linewidth=0.04, fillstyle=solid,fillcolor=colour1, dimen=outer](3.2788463,-3.1075){0.8}{0.0}{180.0}
\psarc[linecolor=black, linewidth=0.04, fillstyle=solid,fillcolor=colour2, dimen=outer](10.478847,-3.1075){3.2}{0.0}{180.0}
\psarc[linecolor=black, linewidth=0.04, fillstyle=solid, dimen=outer](5.6788464,-3.1075){1.6}{0.0}{180.0}
\psarc[linecolor=black, linewidth=0.04, fillstyle=solid, dimen=outer](2.8788462,-3.1075){0.4}{0.0}{180.0}
\psarc[linecolor=black, linewidth=0.04, fillstyle=solid, dimen=outer](3.6788461,-3.1075){0.4}{0.0}{180.0}
\psdots[linecolor=black, dotsize=0.16](2.478846,-3.1075)
\psdots[linecolor=black, dotsize=0.16](4.078846,-3.1075)
\psdots[linecolor=black, dotsize=0.16](0.87884617,-3.1075)
\psdots[linecolor=black, dotsize=0.16](0.07884616,-3.1075)
\psdots[linecolor=black, dotsize=0.16003673](3.2788463,-3.1075)
\psline[linecolor=black, linewidth=0.04, linestyle=dashed, dash=0.17638889cm 0.10583334cm](3.2788463,-3.1075)(3.2788463,3.6925)
\psarc[linecolor=black, linewidth=0.04, fillstyle=solid,fillcolor=colour2, dimen=outer](8.878846,-3.1075){1.6}{0.0}{180.0}
\psarc[linecolor=black, linewidth=0.04, fillstyle=solid,fillcolor=colour2, dimen=outer](9.678846,-3.1075){0.8}{0.0}{180.0}
\psarc[linecolor=black, linewidth=0.04, fillstyle=solid, dimen=outer](12.078846,-3.1075){1.6}{0.0}{180.0}
\psarc[linecolor=black, linewidth=0.04, fillstyle=solid, dimen=outer](9.278846,-3.1075){0.4}{0.0}{180.0}
\psarc[linecolor=black, linewidth=0.04, fillstyle=solid, dimen=outer](10.078846,-3.1075){0.4}{0.0}{180.0}
\psarc[linecolor=black, linewidth=0.04, fillstyle=solid, dimen=outer](8.078846,-3.1075){0.8}{0.0}{180.0}
\psdots[linecolor=black, dotsize=0.16](8.878846,-3.1075)
\psdots[linecolor=black, dotsize=0.16](10.478847,-3.1075)
\psdots[linecolor=black, dotsize=0.16](9.678846,-3.1075)
\psdots[linecolor=black, dotsize=0.16](7.2788463,-3.1075)
\psdots[linecolor=black, dotsize=0.16](13.678846,-3.1075)
\rput(10.478847,-3.5075){$\frac31$}
\psline[linecolor=black, linewidth=0.04, linestyle=dashed, dash=0.17638889cm 0.10583334cm](9.678846,-3.1075)(9.678846,3.6925)
\rput(8.878846,-3.5075){$\frac52$}
\rput(9.678846,-3.5075){$\frac83$}
\end{pspicture}
}
\caption{The triangulation $\mathbb{T}_{\frac75, \frac83}$ is a decagon with the quiddity sequence 
$(c_{1}, \ldots, c_{10})=(3,3,1,2,4,3,1,2,4,1)$, the polynomial $\cX_{\frac75,\frac83}$ 
is given up to a power of~$q$ by the numerator of $\llbracket 2, 4,3 \rrbracket_q$. }
\label{doubletriang}
\end{center}
\end{figure}
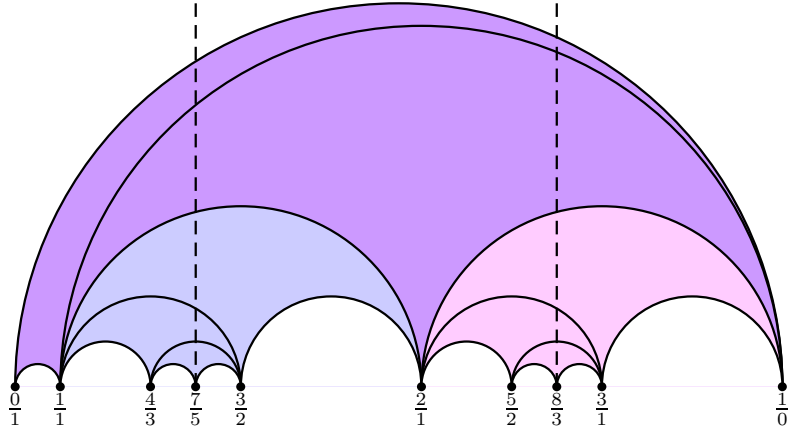

The vertices of the $N$-gon are numbered in decreasing order 
(the first vertex is $\frac10$ and $N^{\rm th}$ vertex is $\frac01$).

Let $k$ and $\ell$ be the numbers of the vertices corresponding to the fractions 
$ \frac{\Rc'}{\Sc'}$  and $\frac{\Rc}{\Sc}$, respectively. 
We assume that $\ell>k+1$ (if $\ell=k+1$ the fractions are linked in the Farey graph 
and this case is covered separetely with Corollary \ref{voisin} and  Proposition \ref{FormVoisin} below).

Let $(c_{1}, \ldots, c_{N})$ be the quiddity of  $\mathbb{T}_{\frac{r}{s},\frac{r'}{s'}}$ ,
 i.e., $c_{i}$ is the number of triangles incident with the $i$-th vertex. 
 By definition of  $\mathbb{T}_{\frac{r}{s},\frac{r'}{s'}}$,  one has 
 $$
 c_{k}=c_{\ell}=c_{N}=1,
 \qquad\hbox{and}\qquad
 c_{i}\geq 2
  \quad\hbox{for}\quad
  i\not=k,\ell,N.
  $$
 Denote by $\Rc_{k,\ell}$ the numerator of the $q$-rational defined as:
 $$
\frac{\Rc_{k,\ell}}{\Sc_{k,\ell}}:=\llbracket c_{k+1}, \ldots, c_{\ell-1}\rrbracket_q.
$$

The main ingredient of our proof is the following.

\begin{lem}
\label{NoLabPro}
There exists an integer $c\geq 0$ such that 
$$
\cX_{\frac{r}{s},\frac{r'}{s'}}=q^{c}\,\Rc_{k,\ell}.
$$
\end{lem}

\begin{proof}
By construction of $\mathbb{T}_{\frac{r}{s},\frac{r'}{s'}}$, one has 
$ \frac{\Rc}{\Sc}=\llbracket c_{1}, \ldots, c_{k-1}\rrbracket_q$ and 
$ \frac{\Rc'}{\Sc'}=\llbracket c'_{1}, \ldots, c'_{k'}\rrbracket_q$, where  
$(c'_{1}, \ldots, c'_{k'})$ is a sequence that can be obtained from $(c_{1}, \ldots, c_{\ell-1})$ 
by applying a sequence of local surgery operations~\eqref{qFirstSurM}. 
This corresponds to removing one after the other exterior triangles in~$\mathbb{T}_{\frac{r}{s},\frac{r'}{s'}}$ 
in order to obtain the triangulation $\mathbb{T}_{\frac{r'}{s'}}^{}$. 

We now use $2\times 2$-matrix computations. 
(When an entry in the matrix is not relevant for the computation we replace it by a ``star'' symbol.)
By Proposition \ref{FacMM} and formula \eqref{qFirstSurM}, one has 
$$
\begin{pmatrix}
\Rc&*\\[6pt]
\Sc&*
\end{pmatrix}=M_{q}(c_{1},\ldots, c_{k-1})
$$
and
$$
\begin{pmatrix}
\Rc'&*\\[6pt]
\Sc'&*
\end{pmatrix}=M_{q}(c'_{1},\ldots, c'_{k'})=q^{-c'}M_{q}(c_{1},\ldots, c_{k-1},1,c_{k+1},\ldots, c_{\ell-1}),
$$
where $c'$ is the number of local transformations applied on the sequence $(c_{1},\ldots, c_{\ell-1})$ (i.e. the number of triangles removed from $\mathbb{T}_{\frac{r}{s},\frac{r'}{s'}}$ to obtain $\mathbb{T}_{\frac{r'}{s'}}^{}$).
Thus, one has
$$
\begin{pmatrix}
\Rc'&*\\[6pt]
\Sc'&*
\end{pmatrix}=
q^{-c'}\begin{pmatrix}
\Rc&*\\[6pt]
\Sc&*
\end{pmatrix}
\begin{pmatrix}
1&-1\\[6pt]
1&0
\end{pmatrix}
M_{q}(c_{k+1},\ldots, c_{\ell-1}).
$$
Now, in the above matrix equation multiplying on the left by 
$$\left(\begin{pmatrix}
\Rc&*\\[6pt]
\Sc&*
\end{pmatrix}
\begin{pmatrix}
1&-1\\[6pt]
1&0
\end{pmatrix}\right)^{-1}
=\begin{pmatrix}
*&-\Rc\\[6pt]
*&-\Sc
\end{pmatrix}
^{-1}
=q^{-c''}
\begin{pmatrix}
-\Sc&\Rc\\[6pt]
*&*
\end{pmatrix},
$$ 
where $q^{c''}$ is the determinant of $M_{q}(c_{1},\ldots, c_{k-1})$, gives
$$
\begin{pmatrix}
\cX_{\frac{r}{s},\frac{r'}{s'}}&*\\[6pt]
*&*
\end{pmatrix}=q^{c}
M_{q}(c_{k+1},\ldots, c_{\ell-1}).
$$
Hence the result, according to Proposition \ref{FacMM}.
\end{proof}

By Proposition~\ref{PosProp}, $\Rc_{k,\ell}$ is a polynomial with positive coefficients. 
Lemma~\ref{NoLabPro} then implies that 
$\cX_{\frac{r}{s},\frac{r'}{s'}}$ is a polynomial with positive integer coefficients.

Theorem \ref{PosPropBis} is proved.

\subsection{The weight of two neighbours}\label{WVSec}

The exact value for the power of $q$ in the case where the fractions are connected by an edge in the Farey graph 
(i.e., the value of $\alpha$ in Corollary \ref{voisin}) 
can be computed explicitly using the coefficients of the expansions of the continued fractions and is given in the following proposition.

\begin{prop}\label{FormVoisin}
Let $\frac{\Rc}{\Sc}$ 
and $ \frac{\Rc'}{\Sc'}$ 
be two $q$-rationals, corresponding to the rationals 
$\frac{r}{s}=\llbracket c_{1}, \ldots, c_{k-1}\rrbracket$ 
and $\frac{r'}{s'}=\llbracket c'_{1}, \ldots, c'_{k'}\rrbracket$, respectively.
If $\frac{r}{s}$ and $\frac{r'}{s'}$ are linked in the Farey graph, and $\frac{r}{s}<\frac{r'}{s'}$,
then
$$\Rc\Sc'-\Sc\Rc' =
\left\{
\begin{array}{lcl}
  q^{c_{1}+\ldots+c_{k}-k},& \text{if}  & r>r',  \\[6pt]
 q^{c'_{1}+\ldots+c_{k'-1}-k'+1},& \text{if}  & r'>r.
 \end{array}\right.
$$
\end{prop}

\begin{proof}
Let the fractions~$\frac{r}{s}$ and~$\frac{r'}{s'}$ be linked in the Farey graph.
  In the case where $r>r'$, the edge between the two fractions forms a left side of a triangle, 
 and the two fractions belong to $\mathbb{T}_{r/s}^{}$. In the weighted Farey graph $\mathbb{T}_{r/s}^{q}$ 
 the local picture is:
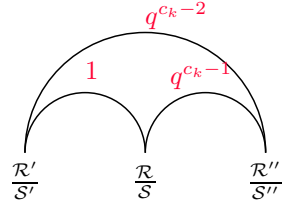
\begin{center}
\psscalebox{1.0 1.0} 
{
\begin{pspicture}(-1,-1.315)(3.385,1.315)
\definecolor{colour0}{rgb}{1.0,0.0,0.2}
\psarc[linecolor=black, linewidth=0.02, dimen=outer](0.87,-0.685){0.8}{0.0}{180.0}
\psarc[linecolor=black, linewidth=0.02, dimen=outer](2.47,-0.685){0.8}{0.0}{180.0}
\psarc[linecolor=black, linewidth=0.02, dimen=outer](1.67,-0.685){1.6}{0.0}{180.0}
\rput[tl](0.87,0.515){\textcolor{colour0}{1}}
\rput[tl](2,0.515){\textcolor{colour0}{$q^{c_{k}-1}$}}
\rput[tl](1.67,1.315){\textcolor{colour0}{$q^{c_{k}-2}$}}
\rput(0.07,-1.085){$\frac{\Rc'_{}}{\Sc'_{}}$}
\rput(1.67,-1.085){$\frac{\Rc_{}}{\Sc_{}}$}
\rput(3.27,-1.085){$\frac{\Rc''_{}}{\Sc''_{}}$}
\end{pspicture}\hspace{1cm}
}
\end{center}
One has by Proposition \ref{FacMM}
$$
M_{q}(c_{1},\ldots, c_{k})=
\begin{pmatrix}
\Rc&-q^{c_{k}-1}\Rc''_{}\\[6pt]
\Sc&-q^{c_{k}-1}\Sc''
\end{pmatrix}.
$$
Taking the determinants of the above matrices, one obtains
$$
\det M_{q}(c_{1},\ldots, c_{k})=
q^{c_k-1}\left(\Rc''\Sc-\Rc\Sc''\right)
= q^{c_{1}+\cdots+c_{k}-k}.
$$
On the other hand, according to the weighted Farey sum~\eqref{WFSEq}, one has
$$
\Rc'_{}=\Rc_{}-q^{c_{k}-1}\Rc''_{}\quad \text{and } \quad 
\Sc'_{}=\Sc_{}-q^{c_{k}-1}\Sc''_{}.
$$
So that we deduce
$$
\Rc\Sc'-\Rc'\Sc= q^{c_{1}+\cdots+c_{k}-k}.
$$

 In the case where $r'>r$, the edge between the two fractions forms 
 a right side of a triangle, and the two fractions belong to $\mathbb{T}_{\frac{r'}{s'}}^{q}$. 
 The local picture is:
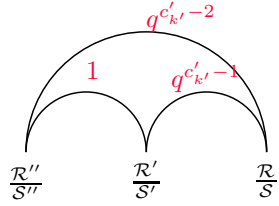
\begin{center}
\psscalebox{1.0 1.0}
{
\begin{pspicture}(0,-1.315)(3.385,1.315)
\definecolor{colour0}{rgb}{1.0,0.0,0.2}
\psarc[linecolor=black, linewidth=0.02, dimen=outer](0.87,-0.685){0.8}{0.0}{180.0}
\psarc[linecolor=black, linewidth=0.02, dimen=outer](2.47,-0.685){0.8}{0.0}{180.0}
\psarc[linecolor=black, linewidth=0.02, dimen=outer](1.67,-0.685){1.6}{0.0}{180.0}
\rput[tl](0.87,0.515){\textcolor{colour0}{1}}
\rput[tl](2,0.515){\textcolor{colour0}{$q^{c'_{k'}-1}$}}
\rput[tl](1.67,1.315){\textcolor{colour0}{$q^{c'_{k'}-2}$}}
\rput(0.07,-1.085){$\frac{\Rc''}{\Sc''}$}
\rput(1.67,-1.085){$\frac{\Rc_{}'}{\Sc'_{}}$}
\rput(3.27,-1.085){$\frac{\Rc_{}}{\Sc_{}}$}
\end{pspicture}
}
\end{center}
Computing the determinant of the matrix 
$$M_{q}(c'_{1},\ldots, c'_{k'})=
\begin{pmatrix}
\Rc'&-q^{c'_{k'}-1}\Rc_{}\\[6pt]
\Sc'&-q^{c'_{k'}-1}\Sc
\end{pmatrix},
$$
one gets $\Rc\Sc'-\Rc'\Sc= q^{c'_{1}+\cdots+c'_{k'-1}-(k'-1)}$.

Hence the result.
\end{proof}

\subsection{Proof of Proposition~\ref{q-1Prop}}\label{Proof-1Sec}

In the special case where~$q=-1$, the matrices~$R_q,L_q$ and~$S_q$ are as follows
$$
R_{-1}=
\begin{pmatrix}
-1&1\\[4pt]
0&1
\end{pmatrix},
\qquad
L_{-1}=
\begin{pmatrix}
1&0\\[4pt]
1&-1
\end{pmatrix}.
\qquad
S_{-1}:=
\begin{pmatrix}
0&1\\[4pt]
1&0
\end{pmatrix}.
$$
Therefore, $R^2_{-1}=L^2_{-1}=S^2_{-1}=\Id$ and $(R_{-1}S_{-1})^3=\Id$.
Formula~\eqref{McGen} then implies that the matrix
$M_{-1}(c_1,\ldots,c_k)$ is equal to either  one of the matrices
$$
\id,
\qquad
R_{-1},
\qquad
S_{-1},
\qquad
R_{-1}S_{-1},
\qquad
(R_{-1}S_{-1})^2.
$$
Each of these matrices has coefficients~$-1,0$, or~$1$. 
By Proposition~\ref{FacMM}, the numerator~$\Rc(-1)$ of a $q$-rational
equals the upper left coefficient of $M_{-1}(c_1,\ldots,c_k)$.
Hence Proposition~\ref{q-1Prop}, Part~(i).

To prove Part~(ii) we proceed by induction on the depth of the Farey graph.
Every $q$-rational~$\frac{\Rc}{\Sc}=\left[\frac{r}{s}\right]_q$ is obtained as the weighted Farey sum:
$$
\frac{\Rc}{\Sc}=\frac{\Rc'}{\Sc'}\oplus_q\frac{\Rc''}{\Sc''}.
$$
In particular, $\Rc=\Rc'+q^\ell\Rc''$. Let $r',r''$ be the nmerators of the corresponding rationals.
By induction assumption, we have:

(a) $\Rc'(-1)=\pm1$ and $\Rc''(-1)=\pm1$  if and only if $r'$ and $r''$ are odd.

(b) $\Rc'(-1)=0$ and $\Rc''(-1)=0$ if and only if $r'$ and $r''$ are even.
The integer $r$ is even if and only if $r',r''$ are either both odd, or both even.
In view of Part (i), this implies that $\Rc(-1)=0$  if and only if $r$ is even. 
And similarly for $\Sc$.

Proposition~\ref{q-1Prop} is proved.

\section{$q$-analogues of the continuants}\label{DetSec}

Continued fractions are related to the polynomials called the continuants.
These polynomials express the numerator and denominator of a continued fraction
$$
\frac{r}{s}=[a_1,\ldots,a_{2m}]=
\llbracket{}c_1,\ldots,c_k\rrbracket
$$
in terms of the coefficients~$a_i$ and~$c_i$.
They are calculated as determinants of certain tridiagonal matrices.
In this section we describe the $q$-continuants that provide similar formulas for calculating
the numerator and denominator of a $q$-continued fraction.
We prove a $q$-analogue of the Euler identity for the continuants.

\subsection{Euler's continuants}

Consider the tridiagonal determinants
$$
K_{k}(c_1,\ldots,c_k):=
\det\left(
\begin{array}{cccccc}
c_1&1&&&\\[4pt]
1&c_{2}&1&&\\[4pt]
&\ddots&\ddots&\!\!\ddots&\\[4pt]
&&1&c_{k-1}&\!\!\!\!\!1\\[4pt]
&&&\!\!\!\!\!1&\!\!\!\!c_{k}
\end{array}
\right)
$$
where the empty space of the matrix is filled by zero entries,
and
$$
K^+_{2m}(a_1,\ldots,a_{2m}):=
\det\left(
\begin{array}{cccccc}
a_1&1&&&\\[4pt]
-1&a_{2}&1&&\\[4pt]
&\ddots&\ddots&\!\!\ddots&\\[4pt]
&&-1&a_{2m-1}&\!\!\!\!\!1\\[6pt]
&&&\!\!\!\!\!-1&\!\!\!\!a_{2m}
\end{array}
\right).
$$
One then has:
$$
\left\{
\begin{array}{rcccl}
r&=&K^+_{2m}(a_1,\ldots,a_{2m})&=&K_{k}(c_1,\ldots,c_k),\\[4pt]
s&=&K^{+}_{2m-1}(a_{2}, \ldots,a_{2m})&=&K_{k-1}(c_2,\ldots,c_k).
\end{array}
\right.
$$

The polynomials $K_{k}(c_1,\ldots,c_k)$ and $K^+_{2m}(a_1,\ldots,a_{2m})$ are usually
called the {\it continuants}; see~\cite{BR,Concr}.
They were studied by Euler who proved a series of identities involving them.
The identities found by Euler are as follows
$$
\begin{array}{rcl}
K_{k-i-1}(c_{i+1},\ldots,c_{k-1})\,K_{\ell-j-1}(c_{j+1},\ldots,c_{\ell-1})
&=&K_{j-i-1}(c_{i+1},\ldots,c_{j-1})\,K_{\ell-k-1}(c_{k+1},\ldots,c_{\ell-1})\\[4PT]
&+&K_{i-\ell-1}(c_{i+1},\ldots,c_{\ell-1})\,K_{k-j-1}(c_{j+1},\ldots,c_{k-1}).
\end{array}
$$
provided~$1\leq{}i<j<k<\ell\leq{}n$.

\begin{rem}
The above identity can also be called the ``Euler-Ptolemy-Pl\"ucker relation''
since it is a particular case of a general class of relations
unified by the Pl\"ucker relations for the coordinates of the Grassmannian~$Gr_{2,N}$
(for more details; see~\cite{FB1} and references therein).
\end{rem}

\subsection{Introducing $q$-deformed continuants}\label{DefDetSec}

We call $q$-{\it continuants} the following 
determinants
\begin{equation}
\label{KEq}
K_{k}(c_1,\ldots,c_k)_{q}:=
\left|
\begin{array}{cccccccc}
[c_1]_{q}&q^{c_{1}-1}&&&\\[6pt]
1&[c_{2}]_{q}&q^{c_{2}-1}&&\\[4pt]
&\ddots&\ddots&\!\!\ddots&\\[4pt]
&&1&\!\!\![c_{k-1}]_{q}&q^{c_{k-1}-1}\\[6pt]
&&&\!\!\!\!\!\!1&\!\!\!\!\!\!\!\![c_{k}]_{q}
\end{array}
\right|,
\end{equation}
and 
\begin{equation}
\label{K+Eq}
K^{+}_{2m}(a_1,\ldots,a_{2m})_{q}:=
\left|
\begin{array}{ccccccccc}
[a_1]_{q}&q^{a_{1}}&&&\\[6pt]
-1&[a_{2}]_{q^{-1}}&q^{-a_{2}}&&\\[4pt]
&-1&[a_3]_{q}&q^{a_{3}}&&\\[4pt]
&&-1&[a_{4}]_{q^{-1}}&q^{-a_{4}}&&\\[4pt]
&&&\ddots&\ddots&\!\!\ddots&\\[4pt]
&&&&-1&[a_{2m-1}]_{q}&\!\!\!\!\!q^{a_{2m-1}}\\[6pt]
&&&&&-1&\!\!\!\![a_{2m}]_{q^{-1}}
\end{array}
\right|,
\end{equation}
where $c_{i}$ and $a_i$ are positive integers.
We will need to consider also the continuant $K^{+}$ over an odd number of variables. We denote

$$
K^{+}_{2m-1}(a_2,\ldots,a_{2m})_{q}:=
\left|
\begin{array}{ccccccccc}
[a_{2}]_{q^{-1}}&q^{-a_{2}}&&\\[4pt]
-1&[a_3]_{q}&q^{a_{3}}&&\\[4pt]
&-1&[a_{4}]_{q^{-1}}&q^{-a_{4}}&&\\[4pt]
&&\ddots&\ddots&\!\!\ddots&\\[4pt]
&&&-1&[a_{2m-1}]_{q}&\!\!\!\!\!q^{a_{2m-1}}\\[6pt]
&&&&-1&\!\!\!\![a_{2m}]_{q^{-1}}
\end{array}
\right|.
$$

One easily gets by induction the following relations between the $q$-continuants and the matrices of convergents $M_{q}$ and $M^{+}_{q}$.
\begin{prop}
\label{MqCont}(i)
Let $(c_{1}, \ldots, c_{k})$ be positive integers. One has
$$
M_{q}(c_{1}, \ldots, c_{k})=
\left(
\begin{array}{ccc}
K_{k}(c_1,\ldots,c_k)_{q}  &  -q^{c_{k}-1}K_{k-1}(c_1,\ldots,c_{k-1})_{q} \\[8pt]
K_{k-1}(c_2,\ldots,c_k)_{q}  &  -q^{c_{k}-1}K_{k-2}(c_2,\ldots,c_{k-1})_{q} 
\end{array}
\right).
$$
(ii) 
Let $(a_{1}, \ldots, a_{2m})$ be positive integers. One has
$$
M^{+}_{q}(a_{1}, \ldots, a_{2m})=
\left(
\begin{array}{ccc}
K^{+}_{2m}(a_1,\ldots,a_{2m})_{q} &  q^{a_{2m}}K^{+}_{2m-1}(a_1,\ldots,a_{2m-1})_{q^{-1}} \\[8pt]
K^{+}_{2m-1}(a_2,\ldots,a_{2m})_{q} & q^{a_{2m}}K^{+}_{2m-2}(a_2,\ldots,a_{2m-1})_{q^{-1}}\end{array}
\right).
$$
\end{prop}

Combining the above proposition with Proposition \ref{PropMMBis}, one gets the following formulas for the numerators and denominators of $q$-rationals.

\begin{prop}\label{fracont}
If
$
\left[\frac{r}{s}\right]_{q}=[a_1,\ldots,a_{2m}]_{q}=
\llbracket{}c_1,\ldots,c_k\rrbracket_{q}=\frac{\Rc(q)}{\Sc(q)}$, then
$$
\begin{array}{rcccl}
\Rc(q)&=&K_{k}(c_1,\ldots,c_k)_{q}&=&q^{a_{2}+a_{4}+\ldots+a_{2m}-1}K^{+}_{2m}(a_{1}, \ldots,a_{2m})_{q},\\[4pt]
\Sc(q)&=&K_{k-1}(c_2,\ldots,c_k)_{q}&=&q^{a_{2}+a_{4}+\ldots+a_{2m}-1}K^{+}_{2m-1}(a_{2}, \ldots,a_{2m})_{q}.
\end{array}
$$
(Note that $k=a_{2}+a_{4}+\ldots+a_{2m}$.)
\end{prop}

\begin{rem}\label{mirform}
(i) Computing $
M_{q}(c_{1}, \ldots, c_{k})^{-1}$ one easily obtains  the ``mirror formula'':
$$K_{k}(c_1,\ldots,c_k)_{q}=q^{c_{1}+\ldots+c_{k}-k}K_{k}(c_k,\ldots,c_1)_{q^{-1}}.$$

(ii) By Corollary \ref{qCoCoOvThm},
if $(c_{1}, \ldots, c_{n})$ is the quiddity sequence of a triangulated $n$-gon one has
$$ 
K_{k}(c_1,\ldots,c_k)_{q}=K_{n-k-2}(c_{k+2},\ldots,c_n)_{q}.$$
\end{rem}

\subsection{Euler's identities for the $q$-deformed continuants}\label{EIdSec}

We give an analogue of Euler's identity for the $q$-continuants. 

Fix a sequence of integers $(c_{i})_{1\leq i\leq n}$ and denote by~$K^{q}_{i,j}$
the continuant of order~$j-i-1$ generated by the segment between $c_{i+1}$ and~$c_{j-1}$:
$$
K^{q}_{i,j}:=K_{j-i-1}(c_{i+1},\ldots,c_{j-1})_{q} ,
$$
with the convention $K^{q}_{i,i}=0$ and $K^{q}_{i,i+1}=1$.

\begin{prop}
\label{EulqProp}
The sequence $(K^{q}_{i,j})_{1\leq i \leq j \leq n}$ satisfies the relation:
\begin{equation}\label{qPto}
K^{q}_{i,k}\,K^{q}_{j,l}=q^{c_{j}+\ldots+c_{k-1}-(k-j)}\,K^{q}_{i,j}\,K^{q}_{k,l}\,+\,K^{q}_{j,k}\,K^{q}_{i,l}
\end{equation}
provided~$1\leq{}i<j<k<\ell\leq{}n$.
\end{prop}

\begin{proof}
We consider a $2\times n$ matrix whose first row is given by $(\Rc_{i})_{1\leq i \leq n}$ and second row by $(\Sc_{i})_{1\leq i \leq n}$ where
$$
\Rc_{i}:=K_{i-1}(c_{1},\ldots,c_{i-1})_{q} ,\quad \Sc_{i}:=K_{i-2}(c_{2},\ldots,c_{i-1})_{q} .
$$
The $2\times 2$-minors of this matrix 
$$
\Delta_{i,j}:=\Rc_{i}\Sc_{j}-\Rc_{j}\Sc_{i},
$$
satisfy the Pl\"ucker-Ptolemy relations
\begin{equation}\label{DelPto}
\Delta_{i,k}\Delta_{j,l}=\Delta_{i,j}\Delta_{k,l}+\Delta_{j,k}\Delta_{i,l}
\end{equation}
provided~$1<{}i<{}j<k<{}\ell\leq{}n$.
We now express these minors in terms of $q$-continuants to obtain the desired indentity.
We use $2\times 2$-matrix computations. When an entry in the matrix is not relevant for the computation we replace it by a `star' symbol.
By Proposition \ref{FacMM} one has 
$$
\begin{pmatrix}
\Rc_{j}&*\\[6pt]
\Sc_{j}&*
\end{pmatrix}=M_{q}(c_{1},\ldots, c_{j-1}) =
M_{q}(c_{1},\ldots, c_{i-1})M_{q}(c_{i},\ldots, c_{j-1}).
$$
We multiply the above matrix equation on the left by 
$$
M_{q}(c_{i},\ldots, c_{i-1})^{-1}=
\begin{pmatrix}
\Rc_{i}&*\\[6pt]
\Sc_{i}&*
\end{pmatrix}
^{-1}
=
q^{-(c_{1}+\ldots+c_{i-1}-(i-1))}
\begin{pmatrix}
*&*\\[6pt]
-\Sc_{i}&\Rc_{i}
\end{pmatrix}
$$ 
and obtain, according to Proposition \ref{MqCont}
$$
\begin{pmatrix}
*&*\\[6pt]
\Delta_{i,j}&*
\end{pmatrix}=q^{c_{1}+\ldots+c_{i-1}-(i-1)}
M_{q}(c_{i},\ldots, c_{j-1})
=
q^{c_{1}+\ldots+c_{i-1}-(i-1)}\begin{pmatrix}
*&*\\[6pt]
K_{i,j}^{q}&*
\end{pmatrix}
$$
Rewriting \eqref{DelPto} using $\Delta_{i,j}=q^{c_{1}+\ldots+c_{i-1}-(i-1)}K_{i,j}^{q}$ gives \eqref{qPto}. 
\end{proof}
\section{Two examples of infinite sequences: $q$-Fibonacci and $q$-Pell numbers}\label{FiboSec}

In this section we consider $q$-analogues of two examples of
infinite sequences of rationals.

\begin{enumerate}
\item[{\bf A}]
The sequence of
convergents of the simplest infinite continued fraction representing the golden ratio
$\frac{1+\sqrt{5}}{2}=[1,1,1,1,\ldots]$.
More precisely, we have
$\frac{r_{n}}{s_n}=\frac{F_{n+1}}{F_n},$
where~$F_n$ is the~$n^{\rm th}$ Fibonacci number.

\item[{\bf B}]
The sequence of
convergents of the continued fraction
$1+\sqrt{2}=[2,2,2,2,\ldots]$ called the ``silver ratio''.
These are the ratios of consecutive Pell numbers:
$\frac{r_{n}}{s_n}=\frac{P_{n+1}}{P_n}$
(recall that the sequence~$(P_n)_{n\geq0}=(0, 1, 2, 5, 12, 29, 70, 169, 408,\ldots)$ is A000129 of OEIS).
\end{enumerate}

We calculate the $q$-deformations of these two sequences, according to our Definition~\ref{qDefn},
and this leads to $q$-deformations of Fibonacci and Pell numbers.
It turns out that the $q$-deformation of the Fibonacci numbers standing in the denominator is the
well-known Sequence A079487, while the $q$-deformation of the Fibonacci numbers
in the numerator is the ``mirror''  Sequence A123245 of OEIS.
Surprisingly, the $q$-deformation of the Pell numbers appears to be a new version.

Let us mention that $q$-deformations of Fibonacci and Pell numbers is an active subject
related to several branches of combinatorics and number theory
(see, e.g.,~\cite{Car,ABF,And,Pan} and references therein).

\subsection{Sequence A079487 and its mirror A123245}

Consider the infinite triangle of positive integers $\left(f_{i,j}\right)_{i,j\geq2}$, 
with the initial conditions $f_{2,0}=f_{3,0}=f_{3,1}=1$ and the convention $f_{i,j}=0$, for $j<0$,
and defined by the recurrence relation
\begin{equation}
\label{FibiRecEq}
\begin{array}{rcl}
f_{2i+1,j}&=&f_{2i,j-1}+f_{2i-1,j},\\[4pt]
f_{2i+2,j}&=& f_{2i+1,j} + f_{2i,j-2},
\end{array}
\end{equation}
for~$i\geq1$.
Sequence A079487 is known under the name of the second kind of Whitney numbers of Fibonacci lattices.
The sum of the numbers in~$i^{\rm th}$ row is equal to the Fibonacci number~$F_{i+1}$.

Sequence A123245 is the {\it mirror} of A079487, 
i.e., this is the triangle A079487 with reversed rows.
Using the notation~$\left(\tilde f_{i,j}\right)_{i,j\geq0}$, 
the recurrence for A123245 is the same as of~\eqref{FibiRecEq}
with the parity exchanged:
\begin{equation}
\label{FibiRecBisEq}
\begin{array}{rcl}
\tilde f_{2i+2,j}&=&\tilde f_{2i+1,j-1} +\tilde f_{2i,j},\\[4pt]
\tilde f_{2i+1,j}&=&\tilde f_{2i,j}+\tilde f_{2i-1,j-2},
\end{array}
\end{equation}
for~$i\geq1$.

The triangles A079487 and  A123245 start as follows
$$
\begin{array}{rcccccccc}
1\\
1&1\\
1&1&1\\
1&2&1&1\\
1&2&2&2&1\\
1&3&3&3&2&1\\
1&3&4&5&4&3&1\\
1&4&6&7&7&5&3&1\\
1&4&7&10&11&10&7&4&1\\
\cdots
\end{array}
\qquad\qquad
\begin{array}{rcccccccc}
1\\
1&1\\
1&1&1\\
1&1&2&1\\
1&2&2&2&1\\
1&2&3&3&3&1\\
1&3&4&5&4&3&1\\
1&3&5&7&7&6&4&1\\
1&4&7&10&11&10&7&4&1\\
\cdots
\end{array}
$$

Consider the polynomials
\begin{equation}
\label{Fibop}
\F_{n+1}(q):=\sum_{0\leq{}i\leq{}n}f_{n,i}\,q^i
\qquad\hbox{and}\qquad
\tilde \F_{n+1}(q):=\sum_{0\leq{}i\leq{}n}\tilde f_{n,i}\,q^i.
\end{equation}
This recurrences~\eqref{FibiRecEq} and~\eqref{FibiRecBisEq} then read
\begin{equation}
\label{FibopRel}
\begin{array}{rcl}
\F_{2\ell+1} &=& q\F_{2\ell}+\F_{2\ell-1},\\[4pt]
\F_{2\ell+2} &=& \F_{2\ell+1}+q^2\F_{2\ell},
\end{array}
\qquad
\begin{array}{rcl}
\tilde\F_{2\ell+1} &=& \tilde\F_{2\ell}+q^2\tilde\F_{2\ell-1},\\[4pt]
\tilde\F_{2\ell+2} &=& q\tilde\F_{2\ell+1}+\tilde\F_{2\ell},
\end{array}
\end{equation}
for~$\ell\geq1$, respectively.

The numbers~$f_{i,j}$ and~$\tilde f_{i,j}$ have many interesting combinatorial interpretations~\cite{MZS} 
(for a nice survey; see~\cite{KK}).
In particular, the numbers~$f_{2i,j}=\tilde f_{2i,j}$ are related to the following graph
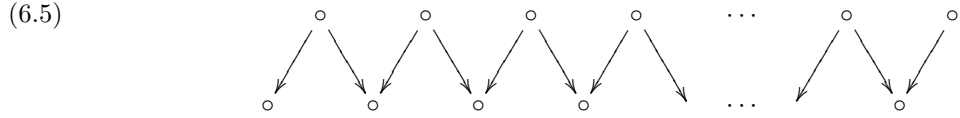
\begin{equation}
\label{FenceOneEq}
\xymatrix @!0 @R=1.2cm @C=0.7cm
{
&\circ\ar@{->}[rd]\ar@{->}[ld]&&\circ\ar@{->}[ld]
\ar@{->}[rd]&&\circ\ar@{->}[ld]
\ar@{->}[rd]&&\circ\ar@{->}[ld]\ar@{->}[rd]&&\cdots&&
\circ\ar@{->}[rd]\ar@{->}[ld]&&\circ\ar@{->}[ld]\\
\circ&&\circ&&\circ&&\circ&&&\cdots&&
&\circ&
}
\end{equation}
with $2i-2$ vertices, called the ``fence''.
More precisely,
the number~$f_{2i,j}$ is equal to the number of $j$-point closures (cf. Definition~\ref{CloDef}) of this graph.
Similarly, the number~$\tilde f_{2i+1,j}$  is equal to the number of $j$-point closures of the graph
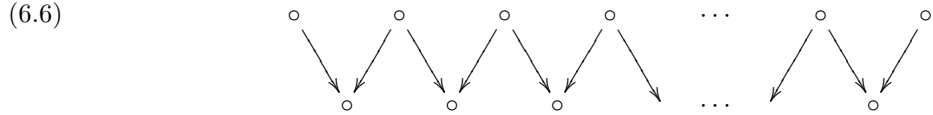
\begin{equation}
\label{FenceEq}
\xymatrix @!0 @R=1.2cm @C=0.7cm
{
\circ\ar@{->}[rd]&&\circ\ar@{->}[ld]
\ar@{->}[rd]&&\circ\ar@{->}[ld]
\ar@{->}[rd]&&\circ\ar@{->}[ld]\ar@{->}[rd]&&\cdots&&
\circ\ar@{->}[rd]\ar@{->}[ld]&&\circ\ar@{->}[ld]\\
&\circ&&\circ&&\circ&&&\cdots&&
&\circ&
}
\end{equation}
with $2i-1$ vertices.
Finally, the numbers~$f_{2i+1,j}$ count $j$-point closures of the same graph but with reversed orientation.

\subsection{Quantization of the ratio of Fibonacci numbers}
It turns out that the polynomials
arising from the $q$-deformations of the sequence
$\frac{F_{n+1}}{F_n}$ 
coincide with the polynomials~\eqref{Fibop}.
More precisely, one has the following.

\begin{prop}
\label{FibProp}
One has~$\left[\frac{F_{n+1}}{F_n}\right]_q=\frac{\tilde\F_{n+1}(q)}{\F_n(q)}$.
\end{prop}

\begin{proof}
The triangulation corresponding to the continued fraction $[1,1,1,1,\ldots]$ is
$$
\xymatrix @!0 @R=0.6cm @C=1cm
{
&\bullet\ar@{-}[ldd]\ar@{-}[dd]\ar@{-}[r]^{q}
&\bullet\ar@{-}[ldd]\ar@{-}[r]^{q^2}\ar@{-}[dd]_q
&\bullet\ar@{-}[ldd]\ar@{-}[r]^{q^2}\ar@{-}[dd]_q
&\bullet\ar@{-}[ldd]&&\bullet\ar@{-}[ldd]\ar@{-}[r]^{q^2}\ar@{-}[dd]_q
&\bullet\ar@{-}[ldd]
\\&&&&\cdots&\cdots
\\
\bullet\ar@{-}[r]
&\bullet\ar@{-}[r]
&\bullet\ar@{-}[r]
&\bullet&&\bullet\ar@{-}[r]
&\bullet
}
$$
cf. Sections~\ref{Trs} and~\ref{WeTr}.
(Note that we mark only those diagonals that have weights different from~$1$.)
Therefore, the polynomials in the numerator and denominator of~$\left[\frac{F_{n+1}}{F_n}\right]_q$ satisfy the recurrence~\eqref{FibopRel}.
\end{proof}

\begin{rem}
The graph $\G_{r_n/s_n}$ is~\eqref{FenceEq}.
It was proved in~\cite{MZS}, that the coefficients of the polynomials~$\F_n(q)$ count its closures.
Theorem~\ref{EnumThm} is a generalization of this result.
\end{rem}

\subsection{A $q$-deformation of the Pell numbers}

The $q$-continued fraction $[2,2,2,2,\ldots]_q$ corresponds to the triangulation
$$
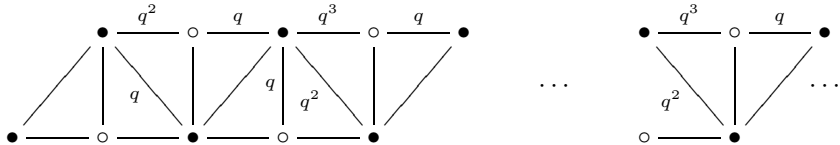
\xymatrix @!0 @R=0.7cm @C=1.2cm
{
&\bullet\ar@{-}[ldd]\ar@{-}[dd]\ar@{-}[r]^{q^2}\ar@{-}[rdd]_q
&\circ\ar@{-}[r]^{q}\ar@{-}[dd]
&\bullet\ar@{-}[ldd]\ar@{-}[dd]_q\ar@{-}[rdd]_{q^2}\ar@{-}[r]^{q^3}
&\circ\ar@{-}[r]^{q}\ar@{-}[dd]&\bullet\ar@{-}[ldd]&&\bullet\ar@{-}[rdd]_{q^2}\ar@{-}[r]^{q^3}&\circ\ar@{-}[r]^{q}\ar@{-}[dd]
&\bullet\ar@{-}[ldd]
\\&&&&&&\cdots&&&\cdots
\\
\bullet\ar@{-}[r]
&\circ\ar@{-}[r]
&\bullet\ar@{-}[r]
&\circ\ar@{-}[r]&\bullet&&&\circ\ar@{-}[r]
&\bullet
}
$$
where the black vertices correspond to the convergents of the continued fraction.

Therefore, the convergents of $[2,2,2,2,\ldots]_q$ are of the form 
$$
\frac{r_{n}}{s_n}=\frac{\Pc_{n+1}}{\tilde\Pc_n},
$$ 
where~$\left(\Pc_n(q)\right)_{n\geq1}$ is
the series of polynomials starting from
$\Pc_1(q)=1$, $\Pc_2(q)=1+q$ and satisfying the recurrence
$$
\begin{array}{rcl}
\Pc_{2\ell+1} &=& \left(1+q\right)\Pc_{2\ell}+q^4\,\Pc_{2\ell-1},\\[4pt]
\Pc_{2\ell+2} &=& \left(q+q^2\right)\Pc_{2\ell+1}+\Pc_{2\ell},
\end{array}
$$
for~$\ell\geq1$, and where $\tilde\Pc_n$ satisfy the same recurrence with reversed parity:
$$
\begin{array}{rcl}
\tilde\Pc_{2\ell+1} &=& \left(q+q^2\right)\tilde\Pc_{2\ell}+\tilde\Pc_{2\ell-1},\\[4pt]
\tilde\Pc_{2\ell+2} &=& \left(1+q\right)\tilde\Pc_{2\ell+1}+q^4\,\tilde\Pc_{2\ell},
\end{array}
$$
for~$\ell\geq1$.

The triangle of the coefficients,~$P_{n,j}$, of the polynomials~$\Pc_n$ starts as follows
$$
\begin{array}{rccccccccccccccccc}
1\\
1&1\\
1&2&1&1\\
1&2&3&3&2&1\\
1&3&5&6&6&5&2&1\\
1&3&7&11&13&13&11&7&3&1\\
1&4&10&18&25&29&29&24&16&9&3&1\\
1&4&12&25&41&56&65&65&56&41&25&12&4&1\\
1&5&16&37&67&101&131&148&146&126&95&61&32&14&4&1\\
1&5&18&46&94&160&233&297&335&335&297&233&160&94&46&18&5&1\\
\cdots
\end{array}
$$
The triangle of the coefficients,~$\tilde P_{n,j}$, of the polynomials~$\tilde\Pc_n$
is the mirror of the triangle~$P_{n,j}$.
The sums of the numbers in the rows are the classical Pell numbers:
$1,2,5,12,29,70,169,408,985,2378,\ldots$
Note that the above triangle of positive integers is not in the OEIS,
and we did not find the polynomials~$\Pc_n$ in the literature.

It follows from Theorem~\ref{EnumThm} that the coefficient~$P_{n,j}$ is equal to the number of
$j$-vertex closures of the following graph with $2n-4$ vertices
$$
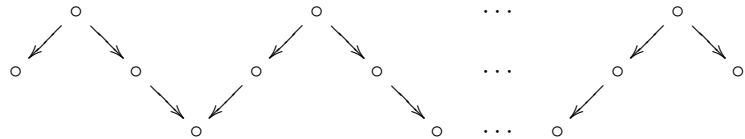
\xymatrix @!0 @R=0.8cm @C=0.8cm
{
&\circ\ar@{->}[ld]\ar@{->}[rd]&&&&\circ\ar@{->}[ld]\ar@{->}[rd]&&
&\cdots&&&\circ\ar@{->}[ld]\ar@{->}[rd]&\\
\circ&&\circ\ar@{->}[rd]&&\circ\ar@{->}[ld]&&\circ\ar@{->}[rd]&
&\cdots&&\circ\ar@{->}[ld]&&\circ\\
&&&\circ&&&&\circ
&\cdots&\circ&&&
}
$$

\begin{rem}
It also follows from Theorem~\ref{EnumThm} that the Pell number~$P_n$ 
counts the total number of closures of the above graph.
This seems to be a new interpretation of the classical Pell numbers.
\end{rem}

\section{Closing remarks and open questions}\label{quSec}

In this final section we formulate some open questions and conjectures that arose as a result of
experimentation with $q$-rationals.

 \medskip
{\bf Unimodality conjecture.}
For all examples of $q$-rationals we considered so far,
the numerator and denominator,~$\Rc$ and~$\Sc$, both have the following unimodality property.
The coefficients of the polynomials~$\Rc$ and~$\Sc$ monotonously increase from~$1$ to the maximal value
(that can be taken by one or more consecutive coefficients) and then monotonously decrease to~$1$.
We are convinced that the polynomials~$\Rc$ and~$\Sc$ have this property for an arbitrary $q$-rational.

Let us mention that the unimodality property can be an indication of relationship
of $q$-rationals to other theories where a similar phenomenon occurs,
such as cohomology of symplectic varieties,
associated with quivers.

 \medskip
{\bf Enumeration questions.}
It is a challenging problem to find more different combinatorial and geometric interpretations
of the polynomials~$\Rc$ and~$\Sc$, besides Theorem~\ref{EnumThm}.

A natural question is to connect the polynomials~$\Rc$ and~$\Sc$ with counting of points
in varieties defined over the finite fields $\mathbb{F}_{q}$.
This property would be similar to that of Gaussian $q$-binomial coefficients.

For every rational~$\frac{r}{s}$, there exist a combinatorial interpretation of~$r$ and~$s$ in terms of
the number of perfect matchings of so-called ``snake graphs''; see~\cite{CS1}.
Note that this counting is closely related with that of~\cite{BCI}.
(Recall that the entries of a Coxeter frieze pattern~\cite{Cox} are precisely the numerator/denominator
of the negative continued fraction with the coefficients in the second row of the frieze.)

Let us also mention that in one of the examples considered in Section~\ref{FiboSec} several combinatorial
interpretations are available.
The sequence A079487 has a number of different combinatorial interpretations
including some particular Motzkin paths and combinations of binomial coefficients;
cf.~\cite{KK}.
This sequence is also related to binary words, cf.~A078807 and to perfect matching 
(or ``dominoes''; see~\cite{KK}).

 \medskip
{\bf Connection to quantum Teichm\"uller space.}
The elementary matrices
$R_{q},L_{q},S_{q}$ (cf. Definition~\ref{RqLqSq}) are very similar to the matrices used in the context of
quantized Teichm\"uller space\footnote{
This was pointed out to us by Michael Shapiro.}; 
see, e.g.,~\cite{CF,CP} (cf. formulas (2.1)--(2.4) of~\cite{CP}).
The notion of $q$-rational is thus closely related to that of ``quantum geodesic length'',
evaluated at the diagonal (i.e., with all the parameters equal to~$q$).
It would be interesting to investigate this relationship.
In particular, some of the properties described in the present paper (and in~\cite{QR})
are still valid in the multi-parameter situation.

\appendix
\addtocontents{toc}{\protect\setcounter{tocdepth}{1}}
\section{Jones polynomial and  $q$-continued fractions}\label{jones}

As an application of the developed notion of $q$-rational numbers,
we relate this notion with the Jones polynomials of 
rational knots. 
This particular class of knots are parametrized by rational numbers, see e.g. \cite{KL}. 
It turns out that the Jones polynomial of the knot parametrized by the rational $\frac{r}{s}$ can be computed from the polynomials $\Rc$ and $\Sc$ of the $q$-rational 
$\left[\frac{r}{s}\right]_{q}$.\\

Recently, the
 Jones polynomial was computed using the combinatorics of snake graphs and cluster algebras~\cite{LeSc}, 
and in the the Kauffman bracket form it 
was computed from the combinatorics of Stern-Brocot tree~\cite{KoWa}. 
The combinatorics used in these papers is closely related to that
 of continued fractions, extensively used in the present paper.

The formula given in \cite[Thm 1.2 (b)]{LeSc} involves a certain $q$-continued fraction. It can be easily rewritten 
under the form \eqref{qa}. With this observation, we reformulate the result of~\cite{LeSc} in terms of $q$-rationals and $q$-continuants.

\subsection{Jones polynomial and $q$-rationals}
We denote by $C(\frac{r}{s})$ the class of rational knots parametrized by the rational $\frac{r}{s}>1$.
The Jones polynomial of $C(\frac{r}{s})$ is denoted by
$$
V_{\frac{r}{s}}\in t^{\half}\Z[t,t^{-1}]\cup\Z[t,t^{-1}]
$$ 

We describe below a polynomial $J_{\frac{r}{s}}(q)\in \Z_{>0}[q]$  satisfying
$$
V_{\frac{r}{s}}(t)=:\pm t^{p} J_{\frac{r}{s}}(-t^{-1}),
$$
where $t^{p}$ is the leading term of $V_{\frac{r}{s}}$ 
(i.e. the term with greatest positive power).

The `normalized Jones polynomial' $J_{\frac{r}{s}}(q)$ is related with the $q$-rational $\left[\frac{r}{s}\right]_{q}$.
More precisely one has the following statement,
which will be obtained in the next section as a corollary of Proposition \ref{JaJc}.

\begin{prop}\label{JRS}
For every rational $\frac{r}{s}$, the Jones polynomial of $C(\frac{r}{s})$ is 
$$
J_{\frac{r}{s}}(q)=q\Rc(q)+(1-q)\Sc(q).
$$
where 
$
\frac{\Rc(q)}{\Sc(q)}=\left[\frac{r}{s}\right]_{q}
$.
\end{prop}

\subsection{Jones polynomial and $q$-continuants}
Using the expansion as regular or negative continued fraction one obtains the following explicit formulas in terms of the $q$-continuants of Section \ref{DefDetSec}.
\begin{prop}\label{JaJc}
(i)
Let $\llbracket c_{1}, \ldots, c_{k}\rrbracket$ be the expansion as negative continued fraction of $\frac{r}{s}>1$. 
The Jones polynomial $J_{\frac{r}{s}}(q)$ is given by
\begin{eqnarray}\label{Jc1}
J_{\frac{r}{s}}(q)
&=&K_{k}(c_{1}+1,\ldots,c_{k})_{q} -qK_{k-1}(c_{2},\ldots,c_{k})_{q}\\[6pt]\label{Jc2}
&=&
\left|
\begin{array}{cccccc}
-q+[c_1+1]_{q}&q^{c_{1}}&&&\\[4pt]
1&[c_{2}]_{q}&q^{c_{2}-1}&&\\[4pt]
&\ddots&\ddots&\!\!\ddots&\\[4pt]
&&1&[c_{k-1}]_{q}&q^{c_{k-1}-1}\\[4pt]
&&&1&[c_{k}]_{q}
\end{array}
\right|.
\end{eqnarray}

(ii) 
Let $[a_{1}, \ldots, a_{2m}]$ be the expansion as regular continued fraction of  $\frac{r}{s}>1$. 
The Jones polynomial $J_{\frac{r}{s}}(q)$ is given by
\begin{eqnarray}
\label{Ja1}
&&J_{\frac{r}{s}}(q)=K^{+}_{2m}(a_1+1,\ldots,a_{2m})_{q}-qK^{+}_{2m-1}(a_2,\ldots,a_{2m})_{q}\\\label{Ja}
&&=q^{a_{2}+a_{4}+\ldots+a_{2m}-1}
\left|
\begin{array}{ccccccccc}
-q+[a_1+1]_{q}&q^{a_{1}+1}&&&\\[6pt]
-1&[a_{2}]_{q^{-1}}&q^{-a_{2}}&&\\[4pt]
&-1&[a_3]_{q}&q^{a_{3}}&&\\[4pt]
&&-1&[a_{4}]_{q^{-1}}&q^{-a_{4}}&&\\[4pt]
&&&\ddots&\ddots&\!\!\ddots&\\[4pt]
&&&&-1&[a_{2m-1}]_{q}&\!\!\!\!\!q^{a_{2m-1}}\\[6pt]
&&&&&-1&\!\!\!\![a_{2m}]_{q^{-1}}
\end{array}
\right|.
\end{eqnarray}
\end{prop}

\begin{proof}
Theorem 1.2 in \cite{LeSc} states that the Jones polynomial $J_{\frac{r}{s}}(q)$ can be computed as the numerator of a continued fraction. In terms of the continuant the formula in \cite{LeSc} is

$$J_{\frac{r}{s}}(q)= 
q^{\ell_{2m}}\, \det(M)$$

where 
$$M=
\left(
\begin{array}{ccccccccc}
-q+[a_1+1]_{q}&1&&&\\[6pt]
-1&[a_{2}]_{q}q^{-\ell_{2}}&1&&\\[4pt]
&-1&[a_3]_{q}q^{\ell_{2}+1}&1&\\[4pt]
&&-1&[a_{4}]_{q}q^{-\ell_{4}}&1&&\\[4pt]
&&&\ddots&\ddots&\!\!\ddots&\\[4pt]
&&&&-1&[a_{2m-1}]_{q}q^{\ell_{2m}+1}&1\\[6pt]
&&&&&-1&\!\!\!\![a_{2m}]_{q}q^{-\ell_{2m}}
\end{array}
\right).
$$
and  $\ell_{i}=a_{1}+a_{2}+\ldots+a_{i}$. 

It is easy to see that this determinant is the same as the one appearing in Eq.~\eqref{Ja}.
Indeed, one can
rescale the rows of $M$ by multiplying $M$ on the left by the diagonal matrix
$$
D_{1}=diag(1,1, q^{-(a_{1}+1_{})}, q^{a_{2}}, \ldots, q^{-(a_{1}+a_{3}+\ldots+a_{2j-1}+1)}, q^{a_{2}+a_{4}+\ldots+a_{2j}}, \ldots, 
q^{-(a_{2}+a_{4}+\ldots+a_{2m-2})}),
$$
and rescale the columns by multiplying $M$ on the left by the diagonal matrix
$$
D_{2}=diag(1, q^{a_{1}+1}, q^{-a_{2}}, \ldots, q^{a_{1}+a_{3}+\ldots+a_{2j-1}+1}, q^{-(a_{2}+a_{4}+\ldots+a_{2j})}, \ldots, q^{a_{1}+a_{3}+\ldots+a_{2m-1}+1} ).
$$
One concludes that
$q^{\ell_{2m}}\det(M)$ is the determinant appearing in formula \eqref{Ja}.
Expanding the determinant one deduces \eqref{Ja1}.

Comparing \eqref{Ja1} with Proposition \ref{fracont} we deduce that $J_{\frac{r}{s}}(q)$ is given as 
$$J_{\frac{r}{s}}(q)=\Rc'-q\Sc'$$
where 
$$
\frac{\Rc'}{\Sc'}=[a_{1}+1, \ldots, a_{2m}]_{q}=\left[\frac{r}{s}+1\right]_{q}=q\left[\frac{r}{s}\right]_{q}+1.
$$ 
This implies~\eqref{Jc1}, \eqref{Jc2}. 

Proposition~\ref{JRS} follows from 
$\Rc'=q\Rc+\Sc$, $\Sc'=\Sc$. 
\end{proof}

\subsection{Jones polynomial and counting on graphs}

Applying the result of Section \ref{EmunSec}, we deduce the following interpretation of the coefficients of the Jones polynomial.

\begin{prop}
Let $[a_{1}, \ldots, a_{2m}]$ be the expansion as regular continued fraction of  $\frac{r}{s}>1$. 
Denote by  $J_{\frac{r}{s}}(q)=\gamma_{0}+\gamma_{1}q+\ldots+\gamma_{d}q^{d}$  the corresponding Jones polynomial where $d=a_{1}+\ldots+a_{2m}$. For all $0\leq \ell \leq d$, $\gamma_{\ell}$ is the number of $\ell$-closures of the graph
$$
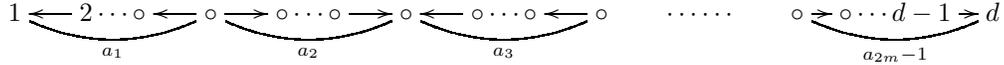
\xymatrix @!0 @R=0.8cm @C=1.3cm
{
1\ar@{<-}[r]\ar@/_0.8pc/@{-}[rr]_{a_1}&2\cdots\circ\ar@{<-}[r]&
\circ\ar@{->}[r]\ar@/_0.8pc/@{-}[rr]_{a_2}&\circ\cdots\circ\ar@{->}[r]&
\circ\ar@{<-}[r]\ar@/_0.8pc/@{-}[rr]_{a_3}&\circ\cdots\circ\ar@{<-}[r]&
\circ&\cdots\cdots&
\circ\ar@{->}[r]\ar@/_0.8pc/@{-}[rr]_{a_{2m}-1}&\circ\cdots d-1\ar@{->}[r]&d
}
$$
for which the pair of vertices $(1,2)$ is either in the closure or in its complement.
\end{prop}

\subsection{Examples}\label{ExJones}
One can use the online atlas of knots to study examples of rational knots. 

(a)
Consider the knot $7\_4$ in Rolfsen's  Knot Table (see {katlas.org}). 
This knot is a rational knot parametrized by $\frac{15}{4}=[3,1,2,1]=\llbracket4,4\rrbracket$. 
It can be represented with the following diagram:\\

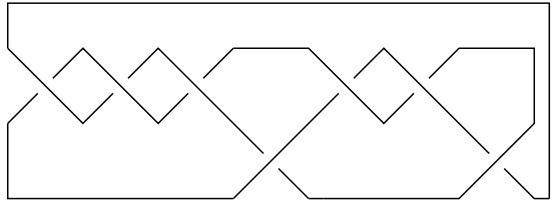
\begin{center}
\psscalebox{1.0 1.0} 
{\psset{unit=0.5cm}
\begin{pspicture}(0,-2.62)(14.44,2.62)
\psline[linecolor=black, linewidth=0.04](0.02,1.4)(0.02,1.4)(2.02,-0.6)(2.02,-0.6)(2.82,0.2)
\psline[linecolor=black, linewidth=0.04](3.22,0.6)(4.02,1.4)(4.02,1.4)(6.42,-1.0)(6.82,-1.4)
\psline[linecolor=black, linewidth=0.04](0.02,-0.6)(0.02,-0.6)(0.82,0.2)(0.82,0.2)
\psline[linecolor=black, linewidth=0.04](1.22,0.6)(1.22,0.6)(2.02,1.4)(4.02,-0.6)(4.82,0.2)(4.82,0.2)(4.82,0.2)
\psline[linecolor=black, linewidth=0.04](5.22,0.6)(6.02,1.4)(6.02,1.4)(6.02,1.4)
\psline[linecolor=black, linewidth=0.04](7.22,-1.8)(8.02,-2.6)(8.02,-2.6)
\psline[linecolor=black, linewidth=0.04](6.02,-2.6)(8.82,0.2)(8.82,0.2)
\psline[linecolor=black, linewidth=0.04](9.22,0.6)(10.02,1.4)(12.82,-1.4)
\psline[linecolor=black, linewidth=0.04](13.22,-1.8)(14.02,-2.6)(14.42,-2.6)
\psline[linecolor=black, linewidth=0.04](14.42,-2.6)(14.42,2.6)(0.02,2.6)(0.02,1.4)(0.02,1.4)(0.02,1.4)
\psline[linecolor=black, linewidth=0.04](6.02,1.4)(8.02,1.4)(10.02,-0.6)(10.82,0.2)(10.82,0.2)
\psline[linecolor=black, linewidth=0.04](11.22,0.6)(12.02,1.4)(14.02,1.4)(14.02,-0.6)(12.02,-2.6)(8.42,-2.6)
\psline[linecolor=black, linewidth=0.04](8.42,-2.6)(8.02,-2.6)
\psline[linecolor=black, linewidth=0.04](6.02,-2.6)(0.02,-2.6)(0.02,-0.6)(0.02,-0.6)
\end{pspicture}
}
\end{center}
The Jones polynomial is $V_{15/4}(t)=-t^{8}(1-t^{-1}+2t^{-2}-3t^{-3}+2t^{-4}-3t^{-5}+2t^{-6}-t^{-7})$.

One computes with Proposition \ref{JRS}
$$
J_{15/4}(q)=q
\begin{vmatrix}
[4]_{q}&q^{3}\\
1&[4]_{q}
\end{vmatrix}+(1-q)[4]_{q}=1+q+2q^{2}+3q^{3}+2q^{4}+3q^{5}+2q^{6}+q^{7}.
$$
One immediately checks $V_{15/4}(t)=-t^{8}J_{15/4}(-t^{-1})$.

The coefficients can be recovered by counting the closures that do not isolate the first vertex of the following graph
$$
\xymatrix @!0 @R=0.8cm @C=1.3cm
{
1\ar@{<-}[r]&2\ar@{<-}[r]&3\ar@{<-}[r]&4\ar@{->}[r]&5\ar@{<-}[r]&6\ar@{<-}[r]&7
}
$$
There are

-- 1 zero-closure,

-- 1 one-closure: the sink 5,

-- 2 two-closures: (1,2), (5,6),

-- 3 three-closures: (1,2,3), (1,2,5), (5,6,7),

-- 2 four-closures: (1,2,3,5), (1,2,5,6)

-- 3 five-closures: (1,2,3,4,5), (1,2,3,5,6), (1,2,5,6,7)

-- 2 six-closures (1,2,3,4,5,6), (1,2,3,5,6,7)

-- 1 seven-closure. \\

(b)
Consider the link $L5a1$ from Thistlethwaite Link Table (see {katlas.org}). 
This link is parametrized by $\frac{8}{3}=[2,1,1,1]=\llbracket 3,3\rrbracket$. It can be represented with the following diagram:\\
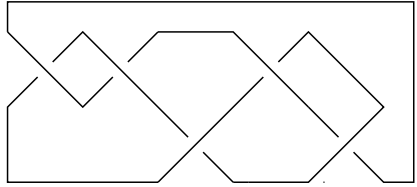
\begin{center}
\psscalebox{1.0 1.0} 
{\psset{unit=0.5cm}
\begin{pspicture}(0,-2.42)(10.84,2.42)
\psline[linecolor=black, linewidth=0.04](1.22,0.8)(2.02,1.6)(2.02,1.6)(4.42,-0.8)(4.82,-1.2)
\psline[linecolor=black, linewidth=0.04](0.02,1.6)(0.02,1.6)(0.02,1.6)(2.02,-0.4)(2.82,0.4)(2.82,0.4)(2.82,0.4)
\psline[linecolor=black, linewidth=0.04](3.22,0.8)(4.02,1.6)(4.02,1.6)(4.02,1.6)
\psline[linecolor=black, linewidth=0.04](5.22,-1.6)(6.02,-2.4)(6.02,-2.4)
\psline[linecolor=black, linewidth=0.04](4.02,-2.4)(6.82,0.4)(6.82,0.4)
\psline[linecolor=black, linewidth=0.04](4.02,1.6)(6.02,1.6)(8.02,-0.4)(8.02,-0.4)(8.02,-0.4)
\psline[linecolor=black, linewidth=0.04](6.42,-2.4)(6.02,-2.4)
\psline[linecolor=black, linewidth=0.04](0.82,0.4)(0.02,-0.4)(0.02,-2.4)(4.02,-2.4)
\psline[linecolor=black, linewidth=0.04](7.22,0.8)(8.02,1.6)(10.02,-0.4)(8.02,-2.4)(6.42,-2.4)
\psline[linecolor=black, linewidth=0.04](0.02,1.6)(0.02,2.4)(10.82,2.4)(10.82,-2.4)(10.02,-2.4)(9.22,-1.6)(9.22,-1.6)
\psline[linecolor=black, linewidth=0.04](8.82,-1.2)(8.02,-0.4)
\psline[linecolor=black, linewidth=0.04](8.42,-2.4)(8.42,-2.4)
\end{pspicture}
}
\end{center}

The Jones polynomial is $V_{8/3}(t)=-t^{3/2}(1-t^{-1}+2t^{-2}-t^{-3}+2t^{-4}-t^{-5})$.

One computes with Proposition \ref{JRS}
$$J_{8/3}(q)=q
\begin{vmatrix}
[3]_{q}&q^{2}\\
1&[3]_{q}
\end{vmatrix}+(1-q)[3]_{q}=1+q+2q^{2}+q^{3}+2q^{4}+q^{5}.
$$
One immediately checks $V_{8/3}(t)=-t^{3/2}J_{8/3}(-t^{-1})$.

The coefficients can be recovered by counting the closures that do not isolate the first vertex of the following graph
$$
\xymatrix @!0 @R=0.8cm @C=1.3cm
{
1\ar@{<-}[r]&2\ar@{<-}[r]&3\ar@{->}[r]&4\ar@{<-}[r]&5
}
$$
There are

-- 1 zero-closure,

-- 1 one-closure: the sink 4,

-- 2 two-closures: (1,2), (4,5),

-- 1 three-closures: (1,2,3),

-- 2 four-closures: (1,2,3,4), (1,2,4,5), 

-- 1 five-closure.

\section{Ptolemy relations and  $q$-continued fractions}\label{ptol}

We present here the $q$-analogue of the results of \cite[\S 1.5]{FB1}. 
This consists in recovering the polynomials $\Rc$ and $\Sc$ of the $q$-rational $[\frac{r}{s}]_{q}$ by applying Ptolemy relations in the triangulation $\mathbb{T}_{r/s}$. 
This will be also reformulated in the language of cluster algebras.

\subsection{Ptolemy weights}

Consider the triangulation $\mathbb{T}_{r/s}$ associated with the rational number 
$$\frac{r}{s}=[a_{1}, \ldots, a_{2m}]=\llbracket c_{1}, \ldots, c_{k}\rrbracket.$$
The vertices in the triangulation are numbered from $0$ to $n-1$ as in \S\ref{Trs}.
For all $i,j$ we assign a weight~$x_{i,j}$ to the edge joining the vertices~$i$ and~$j$, which is a value in~$\Z[q]$. Note that these weights are not the same as the ones defined 
on the weighted triangulation $\mathbb{T}^{q}_{r/s}$ in Section~\ref{WeTr}.

The weights $(x_{i,j})$ must satisfy the following system of Ptolemy relations
\begin{equation}
\label{PtolSys}
\left\{
\begin{array}{rcl}
x_{i,j}x_{k,\ell}&=&x_{i,k}x_{j,\ell}+x_{i,\ell}x_{k,j},
\qquad
\\[6pt]
x_{i,j}&=&x_{j,i},\\[6pt]
x_{i,i}&=&0.
\end{array}
\right.
\end{equation}
The weights $(x_{i,j})$ are uniquely determined from the ``initial weights'' which are the weights of the sides and diagonals of the triangulation $\mathbb{T}_{r/s}$.

\begin{figure}[htbp]
\begin{center}

\psscalebox{1.0 1.0} 
{
\begin{pspicture}(0,-1.6)(13.374142,1.6)
\definecolor{colour0}{rgb}{1.0,0.2,0.0}
\psline[linecolor=black, linewidth=0.02](2.4241421,1.2)(0.024142135,-1.2)(1.2241422,-1.2)(2.4241421,1.2)(2.4241421,-1.2)(1.2241422,-1.2)(1.2241422,-1.2)
\psline[linecolor=black, linewidth=0.02](2.4241421,1.2)(3.6241422,-1.2)(2.4241421,-1.2)
\psline[linecolor=black, linewidth=0.02](2.4241421,1.2)(4.824142,-1.2)(3.6241422,1.2)(2.4241421,1.2)(2.4241421,1.2)
\psline[linecolor=black, linewidth=0.02](3.6241422,1.2)(4.824142,1.2)(4.824142,-1.2)(6.0241423,1.2)(4.824142,1.2)(4.824142,1.2)
\psline[linecolor=black, linewidth=0.02, linestyle=dashed, dash=0.17638889cm 0.10583334cm](3.6241422,-1.2)(8.424142,-1.2)
\rput[bl](0.42414212,-1.6){\textcolor{blue}{$1$}}
\rput[bl](1.6241422,-1.6){\textcolor{blue}{$q$}}
\rput[bl](2.8241422,-1.6){\textcolor{blue}{$q^2$}}
\rput[bl](2.8241422,1.2){\textcolor{blue}{$q^{-a_1}$}}
\rput[bl](10.024142,-1.6){\textcolor{blue}{$q^{a_1+\ldots+a_{2m-1}-1}$}}
\psline[linecolor=black, linewidth=0.02](9.624142,1.2)(10.824142,1.2)(10.824142,-1.2)(12.024142,1.2)(10.824142,1.2)
\psline[linecolor=black, linewidth=0.02](12.024142,1.2)(13.224142,1.2)(10.824142,-1.2)
\psline[linecolor=black, linewidth=0.02](8.424142,1.2)(7.224142,-1.2)
\psline[linecolor=black, linewidth=0.02](4.824142,-1.2)(4.824142,-1.2)
\psline[linecolor=black, linewidth=0.02](8.424142,1.2)(8.424142,-1.2)(9.624142,-1.2)(8.424142,1.2)
\psline[linecolor=black, linewidth=0.02](8.424142,1.2)(10.824142,-1.2)(9.624142,-1.2)
\psline[linecolor=black, linewidth=0.02](9.624142,1.2)(10.824142,-1.2)
\psline[linecolor=black, linewidth=0.02, linestyle=dashed, dash=0.17638889cm 0.10583334cm](8.424142,1.2)(9.624142,1.2)
\rput[t](11.624142,1.6){\textcolor{blue}{$q^{-(n-2)}$}}
\rput[t](12.824142,1.6){\textcolor{blue}{$q^{-(n-3)}$}}
\rput[tl](4.0241423,1.6){\textcolor{blue}{$q^{-(a_1+1)}$}}
\rput[tl](5.624142,1.6){\textcolor{blue}{$q^{-(a_1+2)}$}}
\psline[linecolor=colour0, linewidth=0.02, linestyle=dashed, dash=0.17638889cm 0.10583334cm](6.0241423,1.2)(6.0241423,1.2)
\psline[linecolor=black, linewidth=0.02, linestyle=dashed, dash=0.17638889cm 0.10583334cm](6.0241423,1.2)(8.424142,1.2)(8.424142,1.2)
\psline[linecolor=black, linewidth=0.02, linestyle=dashed, dash=0.17638889cm 0.10583334cm](4.824142,-1.2)(7.224142,1.2)(6.0241423,-1.2)
\psline[linecolor=black, linewidth=0.02, linestyle=dashed, dash=0.17638889cm 0.10583334cm](7.224142,1.2)(7.224142,-1.2)
\end{pspicture}
}
\caption{The assignment \eqref{Ptolength} on a triangulation (the vertical edges have weight 1)}
\label{figptol}
\end{center}
\end{figure}
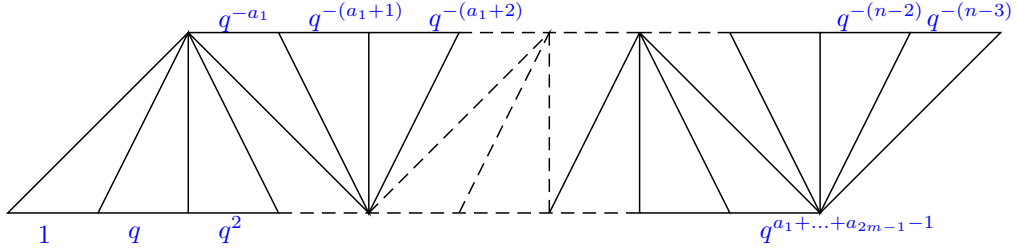

We assign the following initial values to the sides and diagonals of $\mathbb{T}_{rs}$, see Figure \ref{figptol}.
We number the triangles from $0$ to $n-3$ from left to right and we set
\begin{equation}\label{Ptolength}
x_{i,j}=
\left\{
\begin{array}{rl}
q^{\ell}, & \text{if  $[i,j]$ is the base of the $\ell$-th triangle and the triangle is base down} ; \\[6pt]
 q^{-\ell}, &\text{if  $[i,j]$ is the base of the $\ell$-th triangle and the triangle is base up};  \\[6pt]
 1, & \text{otherwise}. 
\end{array}
\right.
\end{equation}

The $q$-analogue of \cite[Fact 3 p7]{FB1} is the following statement.

\begin{prop}\label{propx}
If $(x_{i,j})$ are the weights of the edges of $\mathbb{T}_{r/s}$ satisfying \eqref{PtolSys} and \eqref{Ptolength},
then one has
\begin{eqnarray*}
q^{n-3}x_{0,k+1}&=&\left|
\begin{array}{cccccc}
[c_1]_{q}&q^{c_{1}-1}&&&\\[4pt]
1&[c_{2}]_{q}&q^{c_{2}-1}&&\\[4pt]
&\ddots&\ddots&\!\!\ddots&\\[4pt]
&&1&[c_{k-1}]_{q}&q^{c_{k-1}-1}\\[4pt]
&&&1&[c_{k}]_{q}
\end{array}
\right|=\Rc,\\[12pt]
q^{n-3}x_{1,k+1}&=&\left|
\begin{array}{lllll}
[c_{2}]_{q}&q^{c_{2}-1}&&\\[4pt]
1&\ddots&\!\!\ddots&\\[4pt]
&\ddots&[c_{k-1}]_{q}&q^{c_{k-1}-1}\\[4pt]
&&1&[c_{k}]_{q}
\end{array}
\right|
=\Sc.
\end{eqnarray*}
In particular, $\dfrac{x_{0,k+1}}{x_{1,k+1}}=\left[\frac{r}{s}\right]_{q}$ and $x_{0,k+1}=K_{k}(c_{1}, \ldots, c_{1})_{q^{-1}}$.
\end{prop}

\begin{proof}
This can be established by elementary computations using the Ptolemy relations.
We do the computations in two steps using the following lemmas.

\begin{lem}
For a system $(x_{i,j})$  satisfying \eqref{PtolSys} one has
\begin{equation}\label{PtolGen}
x_{i-1,j+1}=
\frac{1}{x_{i,i+1}x_{i+1,i+2}\cdots x_{j-1,j}}
\begin{vmatrix}
x_{i-1,i+1}&x_{i+1,i+2}&\\[4pt]
x_{i-1,i}&x_{i,i+2}&x_{i+2,i+3}&\\[4pt]
&\;\;\;\ddots&\ddots\\[6pt]
&&x_{j-3,j-2}&x_{j-2,j}&x_{j,j+1}&\\[4pt]
&&&x_{j-2,j-1}&x_{j-1,j+1}
\end{vmatrix}
\end{equation}
\end{lem}

\begin{proof}
This is checked by an easy induction.
\end{proof}

We first use this formula in order to compute $x_{i-1,i+1}$
\begin{lem}\label{lemx}
For the following triangulation of a $(c+2)$-gon
$$
\xymatrix @!0 @R=0.8cm @C=1cm
{
0\ar@{-}[rrr]^{q^{-\beta}}\ar@{-}[dd]_{1}
&&&c+1\ar@{-}[ldd]^{1}\ar@{-}[lldd]^1\ar@{-}[llldd]_1\ar@{-}[rrr]^{q^{-(\beta+c-1)}}\ar@{-}[rrrdd]^{1}
\ar@{-}[rdd]^{1}
&&&c
\\
\\
1\ar@{-}[r]_{q^{\beta+1}}
&2\ar@{-}[r]_{q^{\beta+2}}
&3
&\cdots&c-2 \ar@{-}[rr]_{q^{\beta+c-2}}&
&c-1\ar@{-}[uu]_{1}
}
$$
one has 
$$
x_{0,c}=[c]_{q}q^{-(\beta+c-1)}.
$$
\end{lem}
\begin{proof}
Applying the Ptolemy relations in quadrilaterals with vertices $i,i+1,i+2, c+1$,
one gets
$$
x_{0,2}=1+q, 
\qquad 
x_{i,i+2}=q^{\beta+i}(1+q),\;\; 0<i<c-2, 
\qquad 
x_{c-2,c}=1+q^{-1}.
$$
One then has $x_{i,i+1}=q^{\beta+i}$ for $1\leq i< c-2$, and can write
$$
\frac{x_{i+1,i+2}}{x_{i,i+1}}=q, 
\qquad 
\frac{x_{i,i+2}}{x_{i,i+1}}=1+q=[2]_{q}.
$$
We use this information to compute $x_{0,c}$ with Eq.~\eqref{PtolGen}.
We first rescale the columns, except the first and last, by dividing by $x_{i,i+1}$. We obtain
$$x_{0,c}=
\frac{1}{x_{c-2,c-1}}
\begin{vmatrix}
x_{0,2}&q&\\[4pt]
x_{0,1}&[2]_{q}&q&\\[4pt]
&1&[2]_{q}&q&\\[4pt]
&&\ddots&\ddots&\ddots\\[6pt]
&&&1&[2]_{q}&x_{c-1,c}&\\[4pt]
&&&&1&x_{c-2,c}
\end{vmatrix}
=
\frac{1}{q^{\beta+c-1}}
\begin{vmatrix}
[2]_{q}&q&\\[4pt]
1&[2]_{q}&q&\\[4pt]
&1&[2]_{q}&q&\\[4pt]
&&\ddots&\ddots&\ddots\\[6pt]
&&&1&[2]_{q}&q&\\[4pt]
&&&&1&1+q
\end{vmatrix}
$$
One easily checks that $K_{c-1}(2,2,\ldots,2)_{q}=[c]_{q}$. So that we deduce $x_{0,c}=[c]_{q}q^{-(\beta+c-1)}$.
\end{proof}

Going back to the assigment of Figure \ref{figptol}, by Lemma \ref{lemx} one gets
$$
x_{i-1,i+1}=q^{-(c_{1}+\ldots+c_{i}-i)}[c_{i}]_{q}
$$
for all $1\leq i \leq k$.  So that $\frac{x_{i-1,i+1}}{x_{i-1,i}}=[c_{i}]_{q}$. One then computes 
$x_{0,k+1}$ and $x_{1,k+1}$
using \eqref{PtolGen}, and obtains the desired formulas. 
The final formula $x_{0,k+1}=K_{k}(c_{1}, \ldots, c_{1})_{q^{-1}}$ is obtained using the mirror formula in Remark \ref{mirform}.
\end{proof}

\subsection{Cluster variables and $F$-polynomials}

We translate the above results in the language of cluster algebras~\cite{FZ4} for
the readers familiar with this language.
We will use the notation of~\cite{FZ4,DWZ}.

The weights $(x_{i,j})$, attached to  the edges of the triangulation $\mathbb{T}_{r/s}$ 
and satisfying \eqref{PtolSys}, 
can be viewed as cluster variables in the algebra $\mathcal{A}(\mathbf{x}, \mathbf{y}, B_{r/s})$, 
where $B_{r/s}$ is the opposite of the matrix of incidence of the graph $\G_{r/s}$ defined in Section~\ref{Grs}. 
The variables $\mathbf{x}=(x_{1}, \ldots, x_{n-3})$ are the weights of the diagonals $d_{1}, \ldots , d_{n-3}$ of $\mathbb{T}_{r/s}$ numbered from left to right. 
The variables $\mathbf{y}=(y_{1}, \ldots, y_{n-3})$ can be considered as products of `frozen variables' 
(i.e. weights on the sides of the triangulated $n$-gon), namely 
$$
y_{j}=\prod_{h}x_{h-1,h}\,\prod_{i}x_{i-1,i}^{-1},
$$
 where the edges $[h-1,h]$ and $[i-1,i]$ are sides belonging to the quadrilateral defined by $d_{j}$, which have a positive incidence, resp.  negative incidence, with the diagonal $d_{j}$. 

In this setting, the weights $x_{0,k+1}$ and $x_{1,k+1}$ are cluster variables. 
They correspond to the cluster variables with denominators $x_{1}x_{2}\ldots x_{n-3}$ and  $x_{c_{1}}x_{c_{1}+1}\ldots x_{n-3}$ respectively. 
Let us denote by 
$F_{0,k+1}(y_{1}, \ldots, y_{n-3})$ and $F_{1,k+1}(y_{1}, \ldots, y_{n-3})$ the corresponding $F$-polynomials. 

When the weights $(x_{i,j})$ satisfy the initial assignment \eqref{Ptolength},  one gets $y_{j}=q$ for all $j$ and
Proposition~\ref{propx} then can be reformulated as follows.

\begin{cor}
\label{FPolCor}
One has
$$
x_{0,k+1}=q^{3-n}F_{0,k+1}(q, \ldots, q), \qquad x_{1,k+1}=q^{3-n}F_{1,k+1}(q, \ldots, q).
$$
In particular one obtains 
$$
\Rc(q)=F_{0,k+1}(q, \ldots, q), \qquad \Sc(q)=F_{1,k+1}(q, \ldots, q),
$$
where $\frac{\Rc}{\Sc}=\left[\frac{r}{s}\right]_{q}$.
\end{cor}

\begin{ex}
We go over the examples of 
the rationals~$\frac{r}{s}=\frac{5}{2},\frac{5}{3}$ and~$\frac{7}{5}$, as in Examples \ref{ExT75} and \ref{GGraphEx}. The $F$-polynomials can be computed from the graphs $\G_{r/s}$ with opposite orientations and  Bernhard Keller's applet \cite{applet}.

For $\frac{r}{s}=\frac{5}{2}$, we consider the graph $1\rightarrow 2 \leftarrow 3$. We compute the $F$-polynomials $F_{03}$ and $F_{13}$ corresponding to the cluster variables of denominators $x_{1}x_{2}x_{3}$ and $x_{3}$ respectively. We find
$$F_{03}=1+y_{1}+y_{3}+y_{1}y_{3}+y_{1}y_{2}y_{3}, \qquad F_{13}=1+y_{3}.$$
One checks $[\frac52]_{q}=\frac{1+2q+q^{2}+q^{3}}{1+q}=\frac{F_{03}(q,q,q)}{F_{13}(q,q,q)}$.

For $\frac{r}{s}=\frac{5}{3}$, we consider the graph $1 \leftarrow2\rightarrow  3$. We compute the $F$-polynomials $F_{03}$ and $F_{13}$ corresponding to the cluster variables of denominators $x_{1}x_{2}x_{3}$ and $x_{2}x_{3}$ respectively. We find
$$F_{03}=1+y_{2}+y_{1}y_{2}+y_{2}y_{3}+y_{1}y_{2}y_{3}, \qquad F_{13}=1+y_{2}+y_{2}y_{3}.$$
One checks $[\frac53]_{q}=\frac{1+q+2q^{2}+q^{3}}{1+q+q^{2}}=\frac{F_{03}(q,q,q)}{F_{13}(q,q,q)}$.

For $\frac{r}{s}=\frac{7}{5}$, we consider the graph $1 \leftarrow2\leftarrow3\rightarrow  4$. We compute the $F$-polynomials $F_{04}$ and $F_{14}$ corresponding to the cluster variables of denominators $x_{1}x_{2}x_{3}x_{4}$ and $x_{2}x_{3}x_{4}$ respectively. We find
$$F_{04}=1+y_{3}+y_{2}y_{3}+y_{3}y_{4}+y_{1}y_{2}y_{3}+y_{2}y_{3}y_{4}+y_{1}y_{2}y_{3}y_{4},\qquad F_{14}=1+y_{3}+y_{2}y_{3}+y_{3}y_{4}+y_{2}y_{3}y_{4}.$$
One checks $[\frac75]_{q}=\frac{1+q+2q^{2}+2q^{3}+q^{4}}{1+q+2q^{2}+q^{3}}=\frac{F_{04}(q,q,q,q)}{F_{14}(q,q,q,q)}$.
\end{ex}

\begin{rem}
Let us mention that the $F$-polynomial $F_{0,k+1}(y_{1},\ldots, y_{n-3})$ can be described using the closures of the graph $\G_{r/s}$. For  $\mathcal{C}$ a closure of the graph we use the following notation
$$
\varepsilon_{i,\mathcal{C}}=\left\{
\begin{array}{ccc}
 1 &  \text{ if } & i\in \mathcal{C}  \\
 0 &  \text{ if } & i\not\in \mathcal{C}  
\end{array}
\right.
.
$$
One has the following expression of the $F$-polynomial
$$F_{0,k+1}(y_{1},\ldots, y_{n-3})=
\sum_{ \mathcal{C} }y_{1}^{\varepsilon_{1,\mathcal{C}}}y_{2}^{\varepsilon_{2,\mathcal{C}}}\ldots y_{n-3}^{\varepsilon_{n-3,\mathcal{C}}},$$
where the sum runs over all the closures $\mathcal{C}$ of $\G_{r/s}$.
This is a particular case of a result from~\cite{DWZ}.
\end{rem}

\begin{rem}
\label{RB7}
$F$-polynomials can be used to compute Jones polynomials of rational knots; see~\cite{LeSc}.
Let us consider weights $(x_{i,j})$ on the triangulation $\mathbb{T}_{r/s}$ satisfying the Ptolemy relations \eqref{PtolSys} 
and the same initial assignment \eqref{Ptolength} except for the weight on the side $[0,n-1]$ that we change to $x_{0,n-1}=q^{-1}$. In this case we obtain 
$$x_{0,k+1}=q^{3-n}F_{0,k+1}(q^{2},q,\ldots,q)$$
where as above $F_{0,k+1}$ is the $F$-polynomial corresponding to the longest cluster variable in $\mathcal{A}(\mathbf{x}, \mathbf{y}, B_{r/s})$. 
It was proved in \cite{LeSc} that the above specialization of $F_{0,k+1}$ gives precisely the normalized Jones polynomial $J_{\frac{r}{s}}(q)$.

For instance for  $\frac{r}{s}=\frac{8}{3}=[2,1,1,1]=\llbracket 3,3\rrbracket$ the corresponding graph is $1 \rightarrow2\leftarrow3\rightarrow  4$, and we compute
$$
F_{03}=1 + y_{1} + (1 +y_1 + y_1y_2)y_3 + ((1 + y_1 + y_1y_2)y_3)y_4
$$
On obtains
$$
F_{03}(q^{2},q,q,q)=1+q+2q^{2}+q^{3}+2q^{4}+q^{5}=J_{8/3}(q)
$$
see Section~\ref{ExJones}(b), whereas 
$$
F_{03}(q,q,q,q)=1+2q+2q^{2}+2q^{3}+q^{4}
$$
is the numerator of $[\frac{8}{3}]_{q}$, see Example \ref{FirstEx}(d).
\end{rem}
\bigbreak \noindent
{\bf Acknowledgements}.
We are grateful to Fr\'ed\'eric Chapoton and Vladimir Fock for their interest and a number of useful comments;
special thanks are due to Fr\'ed\'eric Chapoton for an efficient computer program.
This paper was partially supported by the ANR projects SC3A, ANR-15-CE40-0004-01 and PhyMath, ANR-19-CE40-0021.


\begin{thebibliography}{99}

\bibitem{And}
G. Andrews, 
q-series: their development and application in analysis, number theory, combinatorics, physics, and computer algebra. 
CBMS Regional Conference Series in Mathematics, {\bf 66}. 
American Mathematical Society, Providence, RI, 1986.

\bibitem{ABF}
G. Andrews, R. Baxter, P. Forrester,
{\it Eight-vertex SOS model and generalized Rogers-Ramanujan-type identities},
J. Statist. Phys. {\bf 35} (1984), 193--266.

\bibitem{BR}
J. Berstel, C. Reutenauer, 
Noncommutative rational series with applications,
Encyclopedia of Mathematics and its Applications, {\bf 137}. Cambridge University Press, Cambridge, 2011. 

\bibitem{Boc}
F. Boca, 
{\it Products of matrices 
$\left[
\begin{array}{cc}
1&1\\[4pt]
0&1
\end{array}
\right]$
 and $\left[
\begin{array}{cc}
1&0\\[4pt]
1&1
\end{array}
\right]$
 and the distribution of reduced quadratic irrationals}, 
J. Reine Angew. Math. 606 (2007), 149--165.

\bibitem{BCI}
D. Broline, D. Crowe, I. Isaacs, 
{\it The geometry of frieze patterns}, 
Geometriae Dedicata 3 (1974), 171Ð176.

\bibitem{CS1}
I. {C}anak{c}i ,  R. Schiffler, 
{\it Snake graphs and continued fractions}, 
arXiv:1711.02461.

\bibitem{Car}
L. Carlitz, 
{\it Fibonacci notes. III. q-Fibonacci numbers}, 
Fibonacci Quart. {\bf 12} (1974), 317--322.

\bibitem{CP}
L. Chekhov, R. Penner, 
{\it On quantizing Teichm\"uller and Thurston theories},
Handbook of Teichm\"uller theory. Vol. I, 579--645, Math. Soc., Z\"urich, 2007.

\bibitem{CoCo}
J.~H.~Conway, H.~S.~M.~Coxeter, 
{\it Triangulated polygons and frieze patterns},
Math.\ Gaz.\ {\bf 57} (1973), 87--94 and 175--183.

\bibitem{Cox} 
H.~S.~M.~Coxeter. 
{\it Frieze patterns},  Acta Arith.\ {\bf 18}  (1971), 297--310.

\bibitem{DW}
H. Derksen, J. Weyman, 
An introduction to quiver representations. Graduate Studies in Mathematics, {\bf 184}. 
American Mathematical Society, Providence, RI, 2017.

\bibitem{DWZ}
H. Derksen, J. Weyman, 
A. Zelevinsky, 
{\it Quivers with potentials and their representations II: applications to cluster algebras}, 
J. Amer. Math. Soc. {\bf 23} (2010), no. 3, 749--790.

\bibitem{CF}
V. Fok, L. Chekhov, 
{\it Quantum Teichm\"uller spaces}, 
Theoret. and Math. Phys. {\bf 120} (1999), 1245--1259.

\bibitem{FZ4}
S. Fomin, A. Zelevinsky, 
{\it Cluster algebras. IV. Coefficients.} Compos. Math. 143 (2007), no. 1, 112--164.

\bibitem{Concr}
R. Graham, D. Knuth, O. Patashnik, 
Concrete mathematics. A foundation for computer science. 
Addison-Wesley Publishing Company, Advanced Book Program, Reading, MA, 1989. xiv+625 pp.
 
 \bibitem{HW}
G. H. Hardy, E. M. Wright,  
An introduction to the theory of numbers. 
Sixth edition. Revised by D. R. Heath-Brown and J. H. Silverman. With a foreword by Andrew Wiles. 
Oxford University Press, Oxford, 2008, 621~pp.

\bibitem{Hir}
F.E.P. Hirzebruch, 
{\it Hilbert modular surfaces}, Enseign. Math. (2) 19  (1973), 183--281.

\bibitem{HZ}
F. Hirzebruch, D. Zagier, 
{\it Classification of Hilbert modular surfaces}, In: Complex Analysis and Algebraic Geometry,
43--77. Collected Papers II. Iwanami Shoten, Tokyo (1977).

\bibitem{Kat}
S. Katok, 
{\it Coding of closed geodesics after Gauss and Morse},
Geom. Dedicata 63 (1996), no. 2, 123--145.

\bibitem{KL}
 L. Kauffman, S. Lambropoulou, {\it On the classification of rational knots}. Enseign. Math. (2) 49 (2003), no. 3-4, 357Ð410.

\bibitem{applet}
B. Keller, Quiver mutation in Java, applet available at the author's home page.

\bibitem{KK} 
A. Khrabrov, K. Kokhas,
{\it Points on a line, shoelace and dominoes},
arXiv:1505.06309.

\bibitem{KoWa}
T. Kogiso, M. Wakui, 
{\it A Bridge between Conway-Coxeter Friezes and Rational Tangles through the Kauffman Bracket Polynomials},
 arXiv:1806.04840

\bibitem{LeSc}
K. Lee, R. Schiffler. {\it Cluster algebras and Jones polynomials}, 
Selecta Math. (N.S.) {\bf 25} (2019), no. 4, Art. 58, 41 pp. 

\bibitem{MZS} 
E. Munarini, N. Zagaglia Salvi, 
{\it On the rank polynomial of the lattice of order ideals of fences and crowns},
Discrete Math. {\bf 259} (2002), no. 1-3, 163--177. 

\bibitem{FB1} 
S.~Morier-Genoud, V.~Ovsienko, 
{\it Farey boat I.
Continued fractions and triangulations,
modular group and polygon dissections},
Jahresber. Dtsch. Math.- Ver., 121(2):91--136, 2019.

\bibitem{QR} 
S.~Morier-Genoud, V.~Ovsienko, 
{\it On q-Deformed Real Numbers},
Experimental Mathematics, DOI: 10.1080/10586458.2019.1671922, arXiv:1908.04365.

\bibitem{Ovs}
V. Ovsienko, 
{\it Partitions of unity in $\SL(2,\Z)$, negative continued fractions, and dissections of polygons}, 
Res. Math. Sci.~5 (2018), no. 2, Paper No. 21, 25 pp.

\bibitem{Pan}
H. Pan,
{\it Arithmetic properties of q-Fibonacci numbers and q-Pell numbers},
Discrete Math. {\bf 306} (2006), 2118--2127.


\bibitem{Rab}
M. Rabideau,
{\it  F-polynomial formula from continued fractions},
J. Algebra 509 (2018), 467--475.

\bibitem{Ser}
C. Series,
{\it The modular surface and continued fractions}, 
J. London Math. Soc. (2) 31 (1985), no. 1, 69--80.

\bibitem{Sta}
R. Stanley, 
Enumerative combinatorics. Volume 1. Second edition. 
Cambridge Studies in Advanced Mathematics, 49. 
Cambridge University Press, Cambridge, 2012. xiv+626 pp.

\bibitem{Zag} 
D. Zagier, 
{\it Nombres de classes et fractions continues}, 
Journ\'ees Arithm\'etiques de Bordeaux (Conference, Univ. Bordeaux, 1974), 
Ast\'erisque 24-25 (1975), 81--97.



\end{thebibliography}
\end{document}